\documentclass{amsart}
 \usepackage[foot]{amsaddr}
\let\oldtocsection=\tocsection
 
\let\oldtocsubsection=\tocsubsection 
 
\let\oldtocsubsubsection=\tocsubsubsection
 
\renewcommand{\tocsection}[2]{\vspace{0.5em}\hspace{0em}\oldtocsection{#1}{#2}}
\renewcommand{\tocsubsection}[2]{\vspace{0.5em}\hspace{1em}\oldtocsubsection{#1}{#2}}
\renewcommand{\tocsubsubsection}[2]{\vspace{0.5em}\hspace{2em}\oldtocsubsubsection{#1}{#2}}
\usepackage{graphicx,cancel,xcolor,hyperref,comment,graphicx,geometry}

\usepackage{graphicx,url,etoolbox}

\usepackage{lipsum}

 \makeatletter
\patchcmd{\@settitle}{center}{flushleft}{}{}
\patchcmd{\@settitle}{center}{flushleft}{}{}
\patchcmd{\@setauthors}{\centering}{\raggedright}{}{}
\patchcmd{\abstract}{3pc}{0pt}{}{} 
\makeatother

\makeatletter
\renewcommand*\@maketitle{%
  \normalfont\normalsize
  \@adminfootnotes
  \@mkboth{\@nx\shortauthors}{\@nx\shorttitle}%
  \global\topskip42\p@\relax 
  \@settitle
  \ifx\@empty\authors \else \@setauthors \fi
  \ifx\@empty\@date \else {\vskip 1em \vtop{\centering\large\@date\@@par}}\fi
  \ifx\@empty\@dedicatory
  \else
    \baselineskip18\p@
    \vtop{\centering{\footnotesize\itshape\@dedicatory\@@par}%
      \global\dimen@i\prevdepth}\prevdepth\dimen@i
  \fi
  \@setabstract
  \normalsize
  \if@titlepage
    \newpage
  \else
    \dimen@34\p@ \advance\dimen@-\baselineskip
    \vskip\dimen@\relax
  \fi
} 
\renewcommand*\@adminfootnotes{%
  \let\@makefnmark\relax  \let\@thefnmark\relax
  \ifx\@empty\@subjclass\else \@footnotetext{\@setsubjclass}\fi
  \ifx\@empty\@keywords\else \@footnotetext{\@setkeywords}\fi
  \ifx\@empty\thankses\else \@footnotetext{%
    \def\par{\let\par\@par}\@setthanks}%
  \fi
}
\makeatother
 
\usepackage{fancyhdr}
\usepackage{lastpage}

\def\nline{\\ \noalign{\medskip}}
\newtheorem{thm}{Theorem}[section]

\newtheorem{cor}[thm]{Corollary}

\newtheorem{defi}[thm]{Definition}

\newtheorem{lem}[thm]{Lemma}

\newtheorem{prop}[thm]{Proposition}
\newtheorem{rk}[thm]{Remark}

\numberwithin{equation}{section}
 \renewenvironment{proof}{{\bfseries \noindent Proof.}}{\demo}
\newcommand\xqed[1]{%
  \leavevmode\unskip\penalty9999 \hbox{}\nobreak\hfill
  \quad\hbox{#1}}
\newcommand\demo{\xqed{$\square$}}
\hypersetup{bookmarks, colorlinks, urlcolor=blue, citecolor=blue, linkcolor=blue, hyperfigures, pagebackref,
    pdfcreator=LaTeX, breaklinks=true, pdfpagelayout=SinglePage, bookmarksopen=true,bookmarksopenlevel=2}
\title[{\fontsize{6.5}{8}\selectfont Stability and Controllability  results for a Timoshenko system}]{Stability and Controllability  results for a Timoshenko system }
\author{Mohammad AKIL$^{1,3}$, Yacine Chitour$^2$,  Mouhammad Ghader$^{1,2}$ and Ali Wehbe$^1$  \vspace{0.58cm}\\
$^1$Lebanese University, Faculty of sciences 1, Khawarizmi Laboratory of Mathematics and Applications-KALMA, Hadath-Beirut. \\ \vspace{0.2cm}
$^2$ Paris-Saclay University, L2S, 3 Rue Joliot Curie, Gif-sur-Yvette, France. \\ \vspace{0.2cm}
$^3$ Insa de Rouen, LMI, 685 Avenue de l'Université, Rouen, France.\\ \vspace{0.2cm}
Email: mohammad.akil@insa-rouen.fr, yacine.chitour@l2s.centralesupelec.fr, mhammadghader@hotmail.com, ali.wehbe@ul.edu.lb.
}

\begin{document}
\begin{abstract}
In this paper, we study the indirect boundary stability and exact controllability of a one-dimensional Timoshenko system. In the first  part of the paper, we consider the Timoshenko system with only one boundary fractional damping.  We first show that the system is strongly stable but not uniformly stable. Hence, we look for a polynomial decay rate for smooth initial data. Using frequency domain arguments combined with the multiplier method, we prove that the energy decay rate depends on coefficients appearing in the PDE and  on the order of the fractional damping. Moreover, under the equal speed propagation condition, we obtain the optimal polynomial energy decay rate. In the second part of this paper, we study the indirect boundary exact controllability of the Timoshenko system with mixed Dirichlet-Neumann boundary conditions and boundary control.  Using non-harmonic analysis, we first establish a weak observability inequality, which depends  on the ratio of the waves propagation speeds. Next, using the HUM method, we prove that the system is exactly controllable in appropriate spaces and that the control time can be small.\\[0.1in]
\noindent \text{Keywords.} Timoshenko system,  Boundary Damping, Strong stability, Exponential stability, Polynomial stability, Observability inequality, Exact controllability.
\vspace{-0.5cm}
\end{abstract}
\pagenumbering{roman}
\maketitle
\tableofcontents
\clearpage
\pagenumbering{arabic}
\setcounter{page}{1}
\section{Introduction}
 \noindent In this work, we consider the Timoshenko  system  given by\begin{equation} \label{Part-0-01}
\begin{array}{lll}

\displaystyle{\rho_1 u_{tt}-k_1 \left(u_x+y\right)_x=0,}
&\displaystyle{(x,t)\in\left(0,1\right)\times \mathbb{R}_{+},}

         \nline
         
\displaystyle{ \rho_2y_{tt}-k_2y_{xx}+k_1\left(u_x+y\right)=0,}  &\displaystyle{(x,t)\in\left(0,1\right)\times \mathbb{R}_{+},}
\end{array}
\end{equation}
with several types of boundary conditions precised later on. Here the coefficients $\rho_1,\ \rho_2,\ k_1$, and $k_2$ are positive constants and we would like to understand precisely what is the influence of these coefficients on the indirect boundary stability and exact controllability of \eqref{Part-0-01}.\\[0.1in]
\noindent In the first part of this paper, we study the stability of the Timoshenko system \eqref{Part-0-01} with only one boundary fractional damping, i.e,  System \eqref{Part-0-01} is subject to the following  boundary conditions
\begin{equation}\label{Part-0-02}
\begin{array}{lll}

\displaystyle{u(0,t)=y_x(0,t)=y_x(1,t)=0,} & \displaystyle{t\in \mathbb{R}_+, }

\nline

\displaystyle{k_1\left(u_x(1,t)+y(1,t)\right)+\gamma\partial^{\alpha,\eta}_tu(1,t)=0,} & \displaystyle{t\in \mathbb{R}_+,}

\end{array}
\end{equation}
in addition to  the following initial conditions
\begin{eqnarray*}
u(x,0)=u_0(x),& u_t(x,0)=u_1(x),&x\in   (0,1),\nline
y(x,0)=y_0(x),& y_t(x,0)=y_1(x),&x\in(0,1).
\end{eqnarray*}
Here the coefficients $\eta$ and $\alpha$ are non negative and in $(0,1)$ respectively. Fractional calculus includes various extensions of the usual definition of derivative from integer to real order, including the Riemann-Liouville derivative, the Caputo derivative, the Riesz derivative, the Weyl derivative, cf. \cite{samkokilbasmarichev:93}. In this paper, we consider the Caputo's fractional derivative 
\begin{equation}\label{eq-1.7}
\left[D^{\alpha,\eta}\omega \right](t) =\partial_t^{\alpha,\eta}\omega(t)=\frac{1}{\Gamma(1-\alpha)}\int_0^t(t-s)^{-\alpha}e^{-\eta(t-s)}\frac{d\omega}{ds}(s)ds.
\end{equation}
 In the second part  of this paper, we study the exact controllability of the Timoshenko system \eqref{Part-0-01} where only one boundary  control  
 $v$ is applied on the right boundary of the first equation, the second equation is indirectly controlled by means of the coupling between the equations, i.e,   system \eqref{Part-0-01} is subject to the following  boundary conditions
\begin{equation}\label{Part-0-03}
\begin{array}{lll}

\displaystyle{u(0,t)=y_x(0,t)=y_x(1,t)=0,} & \displaystyle{t\in \mathbb{R}_+, }

\nline

\displaystyle{u\left(1,t\right)=v\left(t\right),} & \displaystyle{t\in \mathbb{R}_+,}

\end{array}
\end{equation}
in addition to  the following initial conditions
\begin{eqnarray*}
u(x,0)=u_0(x),& u_t(x,0)=u_1(x),&x\in   (0,1),\nline
y(x,0)=y_0(x),& y_t(x,0)=y_1(x),&x\in(0,1).
\end{eqnarray*}
The  Timoshenko system is usually considered as describing the transverse vibration of a beam and ignoring damping effects of any nature. Precisely, we have the following model, which was developed by Timoshenko   in 1921 (see in \cite{Timoshenko01}),
\begin{equation}\label{Tiom-00}
\left\{
\begin{array}{lll}

\displaystyle{\rho\varphi_{tt}= \left(K\left(\varphi_x-\psi\right)\right)_x,\quad\quad} &\displaystyle{\text{in }\left(0,L\right)\times\mathbb{R}_{+},}

\nline

\displaystyle{I_{\rho} \psi_{tt}=\left( EI\psi_x\right)_{x}-K\left(\varphi_x-\psi\right),\quad\quad} &\displaystyle{\text{in }\left(0,L\right)\times\mathbb{R}_{+},}

\end{array}
\right.
\end{equation}
where $\varphi$ is the transverse displacement of the beam, and $\psi$ is the rotation angle of the filament of the beam. The coefficients $\rho,\ I_\rho,\ E,\ I$
and $K$ are respectively the density (the mass per unit length), the polar moment of inertia of a cross section, Young's
modulus of elasticity, the moment of inertia of a cross section and the shear modulus respectively. \\[0.1in] 
The fractional derivatives are nonlocal and involve singular and non-integrable kernels ($t^{-\alpha}$, $0<\alpha<1$). We refer the readers to \cite{samkokilbasmarichev:93} and the rich references therein for the mathematical description of the fractional derivative. 
The  fractional order or, in general, of convolution type is not only important from the theoretical point of view but also for applications as they naturally arise in physical, chemical, biological, ecological phenomena see for example \cite{parkKang:11},  and references therein. They are used to describe memory and hereditary properties of various materials and processes. For example, in viscoelasticity, due to the nature of the material microstructure, both elastic solid and viscous fluid like response qualities are involved. Using Boltzmann assumption, we end up with a stress-strain relationship defined by a time convolution. Viscoelastic response occurs in a variety of materials, such as soils, concrete, rubber, cartilage, biological tissue, glasses, and polymers (see in \cite{bagleytorvik2:83,bagleytorvik3:83,bagleytorvik1:83} and \cite{mainardibonetti}). In our case, the fractional dissipation describes an active boundary viscoelastic damper designed for the purpose of reducing the vibrations (see in \cite{Mbodje:06,mbomon:95}).\\[0.1in]
The notion of indirect damping mechanisms has been introduced by Russell in \cite{Russell01}, and since this time, it retains the attention of many authors, for instance, let us quote the papers of Alabau \cite{Alabau04,Alabau05} for a general studies on the hyperbolic systems with indirect boundary stabilizations and \cite{Alabau12,Alabau08} for indirect boundary observability and controllability of weakly coupled hyperbolic systems. Note nevertheless that the above system does not enter in the framework of these papers. Let us now mention some known results related to the stabilization of the Timoshenko beam.
There are a number of publications concerning the stabilization of the Timoshenko system with different kinds of damping. Kim and Renardy in  \cite{Kim01} considered Timoshenko \eqref{Tiom-00} with two boundary controls of the form
\begin{equation}
\left\{
\begin{array}{lll}
\displaystyle{K \varphi(L,t)-K u_x(L,t)=\alpha \varphi_t(L,t)},&\displaystyle{\text{on }\mathbb{R}_{+},}\nline
\displaystyle{EI \psi_x(L,t)=-\beta \psi_t(L,t)},&\displaystyle{\text{on }\mathbb{R}_{+},}
\end{array}
\right.
\end{equation}
they establish an exponential decay result for the system \eqref{Tiom-00}. Raposo and al. in \cite{Santos03} studied Timoshenko \eqref{Tiom-00} with homogeneous Dirichlet boundary conditions and two linear frictional dampings; i.e, they considered the following system
 \begin{equation}\label{Tiom-01}
\left\{
\begin{array}{lll}

\displaystyle{\rho_1 u_{tt}-k_1 \left(u_x-y\right)_x+u_t=0,} &\displaystyle{\text{in }\left(0,L\right)\times\mathbb{R}_{+},}

\nline

\displaystyle{ \rho_2y_{tt}-k_2y_{xx}+k_1\left(u_x-y\right)+y_t=0,} &\displaystyle{\text{in }\left(0,L\right)\times\mathbb{R}_{+},}
\nline
\displaystyle{u(0,t)=u(L,t)=y(0,t)=y(L,t)=0,} & \displaystyle{\text{on }\mathbb{R}_+, }

\end{array}
\right.
\end{equation}
they showed that the Timoshenko system \eqref{Tiom-01}  is exponentially stable. Soufyane and Wehbe in \cite{WhebeSoufyane-2003} considered Timoshenko \eqref{Tiom-00} with one internal distributed dissipation law; i.e, they cosidered the following system:
\begin{equation}\label{Tiom-02}
\left\{
\begin{array}{lll}

\displaystyle{\rho\varphi_{tt}= \left(K\left(\varphi_x-\psi\right)\right)_x,\quad\quad} &\displaystyle{\text{in }\left(0,L\right)\times\mathbb{R}_{+},}

\nline

\displaystyle{I_{\rho} \psi_{tt}=\left( EI\psi_x\right)_{x}-K\left(\varphi_x-\psi\right)-b\, \psi_t,\quad\quad} &\displaystyle{\text{in }\left(0,L\right)\times\mathbb{R}_{+},}
\nline
\displaystyle{\varphi(0,t)=\varphi(L,t)=\psi(0,t)=\psi(L,t)=0,} & \displaystyle{\text{on }\mathbb{R}_+, }
\end{array}
\right.
\end{equation}
where $b$ is a positive continuous function such that
$$b(x)\geq b_0>0,\quad\forall x\in [a_0,a_1]\subset [0,L].$$
They showed that  the Timoshenko system \eqref{Tiom-02}  is exponentially stable if and only if the wave propagation speeds are equal (i.e., $\frac{k_1}{\rho_1}=\frac{\rho_2}{k_2}$), otherwise,  only the strong  stability holds. Indeed, Rivera and Racke  in \cite{Racke01} improved the previous results and showed an exponential decay of the solution of the system \eqref{Tiom-02} when the coefficient of the feedback admits an indefinite sign. Mu{\~{n}}oz Rivera and Racke in \cite{Racke02} studied nonlinear Timoshenko system of the form
 \begin{equation}\label{Tiom-03}
\left\{
\begin{array}{lll}

\displaystyle{\rho_1 \varphi_{tt}-\sigma\left(\varphi_x,\psi\right)_x=0}, &\displaystyle{\text{in }\left(0,L\right)\times\mathbb{R}_{+},}

\nline

\displaystyle{ \rho_2\psi_{tt}-b\psi_{xx}+K\left(\varphi_x+\psi\right)+\gamma \theta_x=0,} &\displaystyle{\text{in }\left(0,L\right)\times\mathbb{R}_{+},}
\nline
\displaystyle{\rho_3 \theta_t-K\theta_{xx}+\gamma \psi_{xt}=0, } & \displaystyle{\text{in }\left(0,L\right)\times\mathbb{R}_{+},}

\end{array}
\right.
\end{equation}
where $\theta$ is  the difference temperature. Under some conditions of  $\sigma,\ \rho_1,\ \rho_2,\ b, \ K$ and $\gamma$ they proved several exponential decay results for the linearized system and non-exponential stability result for the case of different wave speeds of propagation. Mu{\~{n}}oz Rivera and Racke in \cite{Racke-2003} studied nonlinear Timoshenko system of the form
 \begin{equation}\label{Tiom-04}
\left\{
\begin{array}{lll}

\displaystyle{\rho_1 \varphi_{tt}-\sigma\left(\varphi_x,\psi\right)_x=0}, &\displaystyle{\text{in }\left(0,L\right)\times\mathbb{R}_{+},}

\nline

\displaystyle{ \rho_2\psi_{tt}-b\psi_{xx}+K\left(\varphi_x+\psi\right)+d \psi_x=0,} &\displaystyle{\text{in }\left(0,L\right)\times\mathbb{R}_{+},}
\end{array}
\right.
\end{equation}
with homogeneous boundary conditions, they showed that  the Timoshenko system \eqref{Tiom-04}  is exponentially stable if and only if the wave propagation speeds are equal, otherwise,  only the polynomial  stability holds. Alabau-Boussouira \cite{Alabau03} extended the results of \cite{Racke-2003} to the case of nonlinear feedback  $\alpha(\psi_t)$, instead of $d \psi_t$, where $\alpha$  is a globally Lipchitz function satisfying some growth conditions at the origin. Indeed,
she considered the following system
 \begin{equation}\label{Tiom-05}
\left\{
\begin{array}{lll}

\displaystyle{\rho_1 \varphi_{tt}-K\left(\varphi_x+\psi\right)_x=0}, &\displaystyle{\text{in }\left(0,L\right)\times\mathbb{R}_{+},}

\nline

\displaystyle{ \rho_2\psi_{tt}-b\psi_{xx}+K\left(\varphi_x+\psi\right)+\alpha( \psi_x)=0,} &\displaystyle{\text{in }\left(0,L\right)\times\mathbb{R}_{+},}

\end{array}
\right.
\end{equation}
with homogeneous boundary conditions. In fact, if the wave propagation speeds are equal she  established a general semi-explicit formula for the decay rate of the energy at infinity. Otherwise, she  proved polynomial decay in the case of different speed of propagation for both linear and nonlinear globally Lipschitz feedbacks.  Ammar-Khodja and al. in \cite{Racke03} considered a linear Timoshenko  system with memory of the form
 \begin{equation}\label{Tiom-06}
\left\{
\begin{array}{lll}

\displaystyle{\rho_1 \varphi_{tt}-K\left(\varphi_x+\psi\right)_x=0}, &\displaystyle{\text{in }\left(0,L\right)\times\mathbb{R}_{+},}

\nline

\displaystyle{ \rho_2\psi_{tt}-b\psi_{xx}+K\left(\varphi_x+\psi\right)+\int_0^tg(t-s)\psi_{xx}(s)ds=0,} &\displaystyle{\text{in }\left(0,L\right)\times\mathbb{R}_{+},}

\end{array}
\right.
\end{equation}
 with homogeneous boundary conditions. They  proved that the system \eqref{Tiom-06} is uniformly stable if and only if the wave speeds are equal and $g$ decays uniformly. Also, they proved an exponential decay if $g$ decays at an exponential rate and polynomially if $g$ decays at a polynomial rate.  Ammar-Khodja and al. in \cite{Soufyane01}  studied the decay rate of the energy of the nonuniform Timoshenko beam with two boundary controls acting in the rotation-angle equation. In fact, under the equal speed wave propagation condition, they established exponential decay results up to an unknown finite dimensional space of initial data. In addition, they showed that the equal speed wave propagation condition is necessary for the exponential stability. However, in the case of non equal speeds, no decay rate has been discussed. The result in \cite{Soufyane01}  has been recently improved by Wehbe and al. in \cite{Wehbe06} where are established nonuniform stability and an optimal polynomial energy decay rate of the Timoshenko system with only one dissipation law on the boundary. 
In addition to the previously cited papers. The stability of the Timoshenko system with a different kind of damping has been also studied \cite{Racke03,Wehbe07, HugoRacke01,MessaoudGuesmia-2009,AlmeidaJnior2013,Messaoud02,Fatori01, Tebou-2015,Hao-2018,Wehbe06,BassamWehbe2016,Wehbe08}. For the stabilization of the Timoshenko beam with nonlinear term,  we mention  \cite{Racke02,Alabau03,Messaoud01,ZuazuaAraruna,Messaoud01,Cavalcanti-2013,Hao-2018}. In \cite{Benaissa2017}, Benaissa and Benazzouz considered the Timoshenko beam system with two dynamic control boundary conditions of fractional derivative type
\begin{equation} \label{abbas}
\left\{
\begin{array}{lll}

\displaystyle{\rho_1 u_{tt}-k_1 \left(u_x+y\right)_x=0,}
&\displaystyle{(x,t)\in\left(0,L\right)\times \mathbb{R}_{+},}

         \nline
         
\displaystyle{ \rho_2y_{tt}-k_2y_{xx}+k_1\left(u_x+y\right)=0,}  &\displaystyle{(x,t)\in\left(0,L\right)\times \mathbb{R}_{+},}

\nline

\displaystyle{u(0,t)=y(0,t)=0,} & \displaystyle{t\in \mathbb{R}_+, }

\nline 

\displaystyle{m_1 u_{tt}(L,t)+k_1 (u_x+y)(L,t)=-\gamma_1\partial^{\alpha,\eta}_tu(L,t),} & \displaystyle{t\in \mathbb{R}_+,}
\nline
\displaystyle{m_2 y_{tt}(L,t)+k_2y_x(L,t)=-\gamma_2\partial^{\alpha,\eta}_ty(L,t),} & \displaystyle{t\in \mathbb{R}_+,}

\end{array}
\right.
\end{equation}
where $m_1$, and $m_2$ are positive constant. They showed that the system \eqref{abbas} is not uniformly stable by a spectral analysis.  Hence, using the semigroup theory of linear operators and a result obtained by Borichev and Tomilov,  they established a  polynomial energy decay rate of type $t^{-\frac{1}{2-\alpha}}$.\\[0.1in]
\noindent  In the first part of this paper, unlike \cite{Benaissa2017}, we study the stability of the Timoshenko system \eqref{Part-0-01} with only one fractional derivative \eqref{Part-0-02}.  We show that the  energy of the system \eqref{Part-0-01}-\eqref{Part-0-02}  has a  polynomial  decay rate of type  $t^{-{\delta\left(\alpha\right)}}$,  where
\begin{equation*}
\delta(\alpha)=\left\{
\begin{array}{ll}
\frac{2}{1-\alpha},&\text{If } \frac{\rho_1}{k_1}=\frac{\rho_2}{k_2}\ \text{ and } \sqrt{\frac{k_1}{k_2}}\neq k\pi,\ \forall k\in\mathbb{N} ,

\\ \noalign{\medskip}

\frac{2}{5-\alpha},&\text{Otherwise. }
\end{array}
\right.
\end{equation*}
Moreover, in some cases,  we obtain the optimal order of polynomial stability (see theorem \ref{Theorem-1.12}).\\[0.1in]
We now turn to the second set of results of the paper, which addresses  controllability issues of the Timoshenko system with different types of control. For the boundary control, Zhang and Hu in \cite{ZhangHu} studied  the exact controllability of a Timoshenko beam with dynamic boundary controls. Since the controlled Timoshenko system connects with a rigid antenna at one end, the authors introduced two new variables in order to describe their actions. The  obtained system was described by two partial and two ordinary differential equations. By using the HUM method, the exact controllability of the system is proved   in the energy space. In \cite{ZuazuaAraruna}, Araruna and Zuazua  considered the dynamical one-dimensional Mindlin-Timoshenko system for beams. They analyzed  how its controllability properties depend on the modulus $k$ of elasticity in shear . In particular,  under some assumptions on the initial  conditions, they proved that the exact boundary controllability property of the Kirchhoff system is  obtained as a singular limit, as $k\to \infty.$ For the internal control, we mention  \cite{LeivaMendoza} and \cite{Martin01}. Note also that the observability and exact controllability of coupled waves equations, have been studied by an extensive number of publications (see\cite{Lions02, Alabau04, Alabau2001,Alabau08, ZhangZuazua01,Komornik01,Khodja02,RaoLiu02}). In addition to the previously cited papers,  we mention \cite{Alabau04,Alabau05,Alabau02,alabau2005,alabau2004, Alabau06,Huang01,kliu97, RaoLoreti01,LiuRao01} for the stabilization of coupled waves equations.\\[0.1in]
\noindent  In this paper,  we study the indirect boundary exact controllability of the Timoshenko system \eqref{Part-0-01} with the boundary conditions \eqref{Part-0-03} while waves propagate with equal or different  speeds. We use  the Hilbert Uniqueness Method introduced by Lions \cite{Lions01} (see also \cite{Lions02,Lions03,Lions04,Komornik01,Komornik02}). To this aim, by using Ingham's Theorem \cite{Komornik01},  we  first establish the inverse and the direct observability inequalities for the homogeneous Timoshenko system. Next, we use  the Hilbert Uniqueness Method,  to get the exact controllability  for the Timoshenko system \eqref{Part-0-01} with the boundary conditions \eqref{Part-0-03} in appropriate functional spaces of terminal data.\\[0.1in]
 Last but not least, in addition to the previously cited papers, the stability of wave equation with fractional damped, have been studied by an extensive number of publications. Mbodje in \cite{Mbodje:06} considered a $1$D wave equation with a boundary viscoelastic damper of the fractional derivative type. In that reference, it is proved that the energy  does not decay uniformly (exponentially) to zero but polynomial energy decay rate of type ${t^{-1}}$ is obtained. This result has been recently improved by Akil and Wehbe  in \cite{akilwehbe01} in the multi-dimensional, where it is shown that  energy of smooth solutions converges to zero as $t$ goes to infinity, as ${t^{-\frac{1}{1-\alpha}}}$. In \cite{akilwehbe02}, coupled wave equations with only one fractional dissipation law are considered and it is proved that the system is not uniformly (exponentially) stable but polynomially stable under arithmetic conditions on coefficients of the system, with optimal order in some cases.\\[0.1in]
\noindent   This paper is organized as follows: In section \ref{Section-1-Timoshenko+Fractional derivative}, we study the stability of the Timoshenko  with only one Fractional derivative. In subsection \ref{Section-1.1-Well-posedness}, we prove the well-posedness   of system \eqref{Part-0-01} with the boundary conditions \eqref{Part-0-02}. In subsection \ref{Section-1.2-Strong-stability}, we prove the strong stability of the system in the lack of the compactness of the resolvent of the generator. In subsection \ref{Section-1.3-Lack of exponential},  we prove that the Timoshenko system \eqref{Part-0-01} considered with the boundary conditions \eqref{Part-0-02} is non-uniformly stable when the speeds of the propagation of the waves are either equal or different. More precisely, we  show that an infinite number of  eigenvalues approach the imaginary axis. In subsection \ref{Section-1.4-Polynomial stability},  we prove the polynomial stability of the system, with a faster polynomial decay rate if  the waves propagate with equal speed: the energy of system \eqref{Part-0-01}-\eqref{Part-0-02}  has a  polynomial  decay rate of type  $t^{-{\delta\left(\alpha\right)}}$,  where
\begin{equation*}
\delta(\alpha)=\left\{
\begin{array}{ll}
\frac{2}{1-\alpha},&\text{If } \frac{\rho_1}{k_1}=\frac{\rho_2}{k_2}\ \text{ and } \sqrt{\frac{k_1}{k_2}}\neq k\pi,\ \forall k\in\mathbb{N} ,

\\ \noalign{\medskip}

\frac{2}{5-\alpha},&\text{Otherwise. }
\end{array}
\right.
\end{equation*}
 In section \ref{Section-2-Exact-controllability}, we study the exact controllability of the Timoshenko system \eqref{Part-0-01} with the boundary conditions \eqref{Part-0-03}. In subsection \ref{Section2-work4}, we set the framework of the homogeneous Timoshenko system \eqref{Part-0-01}  and we establish the characteristic equation satisfied by the eigenvalues of the operator ${\mathcal{A}_2}$. Next, in subsection \ref{Section3-work4}, we prove the  exact controllability of the system \eqref{Part-0-01} with the boundary conditions \eqref{Part-0-03}   while waves propagate with the same speed, i.e., $\frac{k_1}{\rho_1}=\frac{k_2}{\rho_2}$. Depending on number theoretical properties of the  constants  $k_1,\ k_2,\ \rho_1,\ \rho_2$,  we deduce the corresponding observability spaces. In subsection \ref{Section4-work4}, we consider the case where the waves propagate with different speeds and show the exact controllability  of the system \eqref{Part-0-01} with the boundary conditions \eqref{Part-0-03}  and  the corresponding observability spaces depending on the parameter $k_1,\ k_2,\ \rho_1,\ \rho_2.$

\section{Stability of Timoshenko system  with Fractional derivative}\label{Section-1-Timoshenko+Fractional derivative}
\noindent In this section, we study the stability of the Timoshenko system \eqref{Part-0-01} with the  boundary conditions \eqref{Part-0-02}. This system  defined in $\left(0,1\right)\times\left(0,+\infty\right)$ takes the following form
\begin{eqnarray}
\rho_1 u_{tt}-k_1(u_x+y)_x&=&0,\label{Part-1-01}\nline
\rho_2 y_{tt}-k_2 y_{xx}+k_1(u_x+y)&=&0,\label{Part-1-02}
\end{eqnarray}
with the following boundary conditions
\begin{eqnarray}
u(0,t)=y_x(0,t)=y_x(1,t)&=&0,\qquad t\in \mathbb{R}_+,\label{Part-1-03}\nline
k_1\left(u_x(1,t)+y(1,t)\right)+\gamma\partial^{\alpha,\eta}_tu(1,t)&=&0,\qquad t\in \mathbb{R}_+,\label{Part-1-04}
\end{eqnarray}
in addition to  the following initial conditions
\begin{eqnarray}
u(x,0)=u_0(x),& u_t(x,0)=u_1(x),&x\in   (0,1),\label{Part-1-05}\nline
y(x,0)=y_0(x),& y_t(x,0)=y_1(x),&x\in(0,1).\label{Part-1-06}
\end{eqnarray}
Here $\rho_1,\ \rho_2,\ k_1,\ k_2,\ \gamma$ are positive constants,  $\eta \geq 0$ and $\alpha\in (0,1)$. 
\subsection{Augmented model and well-posedness}\label{Section-1.1-Well-posedness}
In this part, using a semigroup approach, we establish well-posedness result for the system  \eqref{Part-1-01}-\eqref{Part-1-06}. For this purpose, we first recall Theorem 2 stated  in  \cite{Mbodje:06}.
\begin{thm}\label{Theorem-1.1} \rm{(See Theorem 2  in  \cite{Mbodje:06})
 Let $\alpha\in (0,1)$ and $\mu$ be the function defined almost everywhere on $\mathbb{R}$ by 
$$
\mu(\xi)=|\xi|^{\frac{2\alpha-1}{2}}.
$$
The relationship between the input $V$ and the output $O$ of the following system
$$ 
\begin{array}{lllll}
\omega_t(\xi,t)+\left(\xi^2+\eta\right)\omega(\xi,t)-V(t)\mu(\xi)&=&0,\quad  (\xi,t)\in  \mathbb{R}\times \mathbb{R}_+,\nline
\omega(\xi,0)&=&0, \quad  \xi\in  \mathbb{R},\nline
\displaystyle
O(t)-\kappa(\alpha)\int_{\mathbb{R}}\mu(\xi)\omega(\xi,t)d\xi&=&0, \quad  t\in  \mathbb{R}_+
\end{array}
$$
is given by 
$$
O=I^{1-\alpha,\eta}V,
$$
where
\begin{equation}\label{eq-1.8}
\kappa(\alpha)=\frac{\sin(\alpha\pi)}{\pi},\quad 
[I^{\alpha,\eta}V](t)=\int_0^t\dfrac{(t-\tau)^{\alpha-1}e^{-\eta(t-\tau)}}{\Gamma(\alpha)}V(\tau)d\tau.
\end{equation}
\xqed{$\square$}}
\end{thm}
\noindent Since $\alpha\in (0,1)$, one has that $\kappa(\alpha)>0$. From Equations \eqref{eq-1.7} and \eqref{eq-1.8} one clearly has
\begin{equation}\label{eq-1.9}
D^{\alpha,\eta}V=I^{1-\alpha,\eta}DV.
\end{equation}
\noindent Therefore, from Theorem \ref{Theorem-1.1} and Equation \eqref{eq-1.9}, System \eqref{Part-1-01}-\eqref{Part-1-06} can be rewritten as the following augmented model 
\begin{eqnarray}
\rho_1 u_{tt}-k_1(u_x+y)_x&=&0,\quad (x,t)\in (0,1)\times \mathbb{R}_+,\label{Part-1-07}\nline
\rho_2 y_{tt}-k_2 y_{xx}+k_1(u_x+y)&=&0,\quad (x,t)\in (0,1)\times \mathbb{R}_+,\label{Part-1-08}\nline
\omega_t(\xi,t)+(\xi^2+\eta)\omega(\xi,t)-u_t(1,t)\mu(\xi)&=&0,\quad  (\xi,t)\in  \mathbb{R}\times \mathbb{R}_+,\label{Part-1-09}
\end{eqnarray}
with the boundary conditions
\begin{eqnarray}
u(0,t)=y_x(0,t)=y_x(1,t)&=&0,\quad  t\in  \mathbb{R}_+,\label{Part-1-10}\nline
k_1\left(u_x(1,t)+y(1,t)\right)+\gamma\kappa(\alpha)\int_{\mathbb{R}}\mu(\xi)\omega(\xi,t)d\xi&=&0,\quad  t\in  \mathbb{R}_+.\label{Part-1-11}
\end{eqnarray}
System \eqref{Part-1-07}-\eqref{Part-1-09}  has to be completed with the following initial conditions 
\begin{eqnarray}
u(x,0)=u_0(x),& u_t(x,0)=u_1(x),& x\in (0,1),\label{Part-1-12}\nline
y(x,0)=y_0(x),& y_t(x,0)=y_1(x),&  x\in (0,1),\label{Part-1-13}\nline
&\omega(\xi,0)=0,& \xi\in \mathbb{R}.\label{Part-1-14}
\end{eqnarray}
The energy of System \eqref{Part-1-07}-\eqref{Part-1-11} is given by
\begin{equation*}
E_1\left(t\right)= \frac{1}{2}\int_0^1\left(\rho_1\left|u_t\right|^2+\rho_2\left|y_t\right|^2+k_1\left|u_x+y\right|^2+k_2\left|y_x\right|^2\right)dx+\frac{\gamma\kappa(\alpha)}{2}\int_{\mathbb{R}}\left|w(\xi,t)\right|^{2}d\xi.
\end{equation*}
 Let $\left(u,y,\omega\right)$ be a regular solution of \eqref{Part-1-07}-\eqref{Part-1-09}. Multiplying  \eqref{Part-1-07}, \eqref{Part-1-08} and \eqref{Part-1-09}   by $u_t,\ y_t,$ and  $\gamma \kappa(\alpha) w $ respectively, then  using  the boundary conditions \eqref{Part-1-10}-\eqref{Part-1-11}, we get
\begin{equation*}
E_1'\left(t\right)=-\gamma\kappa(\alpha)\int_{\mathbb{R}}(\xi^2+\eta)\left|\omega(\xi,t)\right|^2 d\xi\leq 0.
\end{equation*}
Thus  System  \eqref{Part-1-07}-\eqref{Part-1-11}  is dissipative in the sense that its energy is non increasing with respect to the time $t$.
Let us define the energy space  $\mathcal{H}_1$ by
\begin{equation*}
	\mathcal{H}_1=H_L^1\left(0,1\right)\times L^2\left(0,1\right)\times H_*^1\left(0,1\right)\times L^2\left(0,1\right)\times L^2\left(\mathbb{R}\right),
\end{equation*}
such that
$$H_L^1\left(0,1\right)=\left\{u\in H^1(0,1)\  |\ u(0)=0\right\}\ \ \ \text{and}\ \ \ H_*^1\left(0,1\right)=\left\{u\in H^1(0,1)\  |\ \int_0^1u dx=0\right\}.$$
It is easy to check that the spaces $H_L^1$ and $H_*^1$ are  Hilbert spaces over $\mathbb{C}$ equipped respectively  with the norms
$$\left\|u\right\|^2_{H_L^1\left(0,1\right)}=\left\|u_x\right\|^2\ \ \ \text{and}\ \ \  \left\|u\right\|^2_{H_*^1\left(0,1\right)}=\left\|u_x\right\|^2,$$
where  $\|\cdot\|$  denotes  the usual norm of $L^2\left(0,1\right)$. The energy space $\mathcal{H}_1$ is equipped with the inner product defined by
\begin{equation}\label{inner product}
\begin{array}{lll}
\displaystyle{
\left<U,U_1\right>_{\mathcal{H}_1}=\rho_1 \int_0^1v\overline{v}_1dx+\rho_2 \int_0^1z\overline{z}_1dx+k_1\int_0^1\left(u_x+y\right)\overline{\left((u_1)_x+y_1\right)}dx}

             \nline

\hspace{1.5cm}\displaystyle{+k_2\int_0^1 y_x\overline{(y_1)_x}dx
+\gamma \kappa(\alpha)\int_{\mathbb{R}}\omega(\xi)\overline{\omega_1(\xi) }d\xi,}
\end{array}
\end{equation}
for  all  $U=\left(u,v,y,z,\omega\right)$ and $U_1=\left(u_1,v_1,y_1,z_1,\omega_1\right)$ in $\mathcal{H}_1$. We use 
$\|U\|_{\mathcal{H}_1}$ to denote the corresponding norm. We define the linear unbounded operator  $\mathcal{A}_1$ in $\mathcal{H}_1$ by
\begin{equation*}
\begin{array}{lll}

\displaystyle{D\left(\mathcal{A}_1\right)=\bigg\{U=(u,v,y,z,\omega)\in \mathcal{H}_1\ |\ u\in H^2\left(0,1\right)\cap H^1_L\left(0,1\right),\ y \in H^2\left(0,1\right)\cap H^1_{*}\left(0,1\right),}

             \nline

\hspace{2cm}\displaystyle{  v\in H^1_L\left(0,1\right),\   z\in H^1_*\left(0,1\right),\  y_x\in H^1_0\left(0,1\right),\  -(\xi^2+\eta)\omega(\xi)+v(1)\mu(\xi)\in L^2\left(\mathbb{R}\right),}\nline

\hspace{6cm}\displaystyle{k_1\left(u_x(1)+y(1)\right)+\gamma\kappa(\alpha)\int_{\mathbb{R}}\mu(\xi)\omega(\xi)d\xi=0,\ |\xi|\omega(\xi)\in L^2(\mathbb{R}) }
\bigg\},
\end{array}
\end{equation*}
and
\begin{equation*}
	\mathcal{A}_1U=\left(v,\frac{k_1}{\rho_1}(u_x+y)_x,z,\frac{k_2}{\rho_2}y_{xx}-\frac{k_1}{\rho_2}(u_x+y),-(\xi^2+\eta)\omega(\xi)+v(1)\mu(\xi)\right),\ \  U=\left(u,v,y,z,\omega\right)\in D\left(\mathcal{A}_1\right).
\end{equation*}
\begin{rk}
\rm{The condition $|\xi|\omega(\xi)\in L^2(\mathbb{R})$ is imposed to insure the existence of $\displaystyle{\int_{\mathbb{R}}\mu(\xi)\omega(\xi)d\xi}$ in \eqref{Part-1-11}.}\xqed{$\square$}
\end{rk}
\noindent If $U=\left(u,u_t,y,y_t,\omega\right)$ is the state of \eqref{Part-1-07}-\eqref{Part-1-11}, then the Timoshenko system is transformed into the first order evolution equation on the Hilbert space $\mathcal{H}_1$ given by
\begin{equation}\label{C-u}
U_t(x,t)=\mathcal{A}_1U(x,t),\
	U\left(x,0\right)=U_0(x),
\end{equation}
where
\begin{equation*}
U_0\left(x\right)=\left(u_0(x),u_1(x),y_0(x),y_1(x),0\right).
	\end{equation*}
			\begin{prop}\label{Theorem-1.2}
		\rm{	The unbounded linear operator $\mathcal{A}_1$ is m-dissipative in the energy space $\mathcal{H}_1$.}
		\end{prop}
		\begin{proof}
			For $U=(u,v,y,z,\omega)\in D\left(\mathcal{A}_1\right)$, one has  
			\begin{equation*}
		\text{Re}\left<\mathcal{A}_1U,U\right>_{\mathcal{H}_1}=	-\gamma\kappa(\alpha)\int_{\mathbb{R}}(\xi^2+\eta)\left|\omega(\xi)\right|^2 d\xi\leq 0,
			\end{equation*}
			which implies that $\mathcal{A}_1$ is dissipative. Here $\text{Re}$ is used to denote the real part of a complex number.  We next prove the maximality of $\mathcal{A}_1$. For $F=(f_1,f_2,f_3,f_4,f_5)\in\mathcal{H}_1$, we show the existence of $U=(u,v,y,z,\omega)\in D(\mathcal{A}_1)$, unique solution of the equation 
			\begin{equation*}
						(I-\mathcal{A}_1)U=F.
			\end{equation*}
			Equivalently, one must consider the  system given by
			\begin{eqnarray}
			u-v&=&f_1,\label{Part-1-15}\\
			{\rho_1} v-{k_1}(u_x+y)_x&=&{\rho_1}f_2,\label{Part-1-16}\\
			y-z&=&f_3,\label{Part-1-17}\\
			{\rho_2}z-{k_2}y_{xx}+{k_1}(u_x+y)&=&{\rho_2}f_4,\label{Part-1-18}\\
			(1+\xi^2+\eta)\omega(\xi)-v(1)\mu(\xi)&=&f_5(\xi).\label{Part-1-19}
			\end{eqnarray}
From \eqref{Part-1-19} and \eqref{Part-1-15}, we get 
			\begin{equation}\label{Part-1-20}
			\omega(\xi)=\frac{f_5(\xi)}{1+\xi^2+\eta}+\frac{u(1)\mu(\xi)}{1+\xi^2+\eta}-\frac{f_1(1)\mu(\xi)}{1+\xi^2+\eta}.
			\end{equation}
			Inserting \eqref{Part-1-15} and  \eqref{Part-1-17} in \eqref{Part-1-16} and \eqref{Part-1-18}, we get 
			\begin{eqnarray}
			{\rho_1}u-{k_1}(u_x+y)_x&=&{\rho_1}\left(f_1+ f_2\right),\label{Part-1-21}\\
			 {\rho_2}y-{k_2}y_{xx}+{k_1}(u_x+y)&=&\rho_2\left(f_3+ f_4\right),\label{Part-1-22}
			\end{eqnarray}
			with the boundary conditions 
			\begin{equation}\label{Part-1-23}
			u(0)=0,\quad k_1\left(u_x(1)+y(1)\right)=-\gamma\kappa(\alpha)\int_{\mathbb{R}}\mu(\xi)\omega(\xi)d\xi\quad \text{and}\quad y_x(0)=y_x(1)=0.
			\end{equation}
			Let $(\varphi,\psi)\in H_L^1(0,1)\times H_*^1(0,1)$. Multiplying Equations \eqref{Part-1-21} and \eqref{Part-1-22}  by $\overline{\varphi}$ and $\overline{\psi}$ respectively, we obtain 
			\begin{eqnarray}
				\rho_1\int_0^1u\overline{\varphi} dx+k_1\int_0^1\left(u_x+y\right)\overline{\varphi}_{x}dx-k_1[\left(u_x+y\right)\overline{\varphi}]_0^1&=&\rho_1\int_0^1\left(f_1+f_2\right)\overline{\varphi} dx,\label{Part-1-24}\\
				\rho_2\int_0^1y\overline{\psi} dx+k_2\int_0^1y_x\overline{\psi}_x dx+k_1\int_0^1\left(u_x+y\right)\overline{\psi} dx&=&\rho_2\int_0^1\left(f_3+f_4 \right)\overline{\psi} dx\label{Part-1-25}.
			\end{eqnarray}
	 Using \eqref{Part-1-20} and \eqref{Part-1-23}, we get 
			\begin{equation}\label{Part-1-26}
			-k_1[\left(u_x+y\right)\overline{\varphi}]_0^1=\gamma\kappa(\alpha) \int_{\mathbb{R}}\mu(\xi)\omega(\xi)d\xi\hspace{0.05cm} \overline{\varphi}(1)=\overline{\varphi}(1)M_1(\eta,\alpha)+u(1)\overline{\varphi}(1)M_2(\eta,\alpha)+f_1(1)\overline{\varphi}(1)M_2,
			\end{equation}
			where 
			\begin{equation*}
			M_1(\eta,\alpha)= \gamma\kappa(\alpha)\int_{\mathbb{R}}\frac{\mu(\xi)f_5(\xi)}{1+\eta+\xi^2}d\xi \quad \text{and} \quad M_2(\eta,\alpha)=\gamma\kappa(\alpha)\int_{\mathbb{R}}\frac{\mu^2(\xi)}{1+\eta+\xi^2}d\xi.
			\end{equation*}
		By	using the fact that $f_5\in L^2(\mathbb{R})$, $\eta\geq 0$,  and $\alpha\in (0,1)$, we can easily check that  $M_1(\eta,\alpha)$ and $M_2(\eta,\alpha)$ are well defined. Adding Equations \eqref{Part-1-24} and \eqref{Part-1-25}, we obtain
			\begin{equation}\label{Part-1-27}
			a\left((u,y),(\varphi,\psi)\right)=L\left(\varphi,\psi\right),\quad \forall \ (\varphi,\psi)\in H_L^1(0,1)\times H_*^1(0,1),
			\end{equation}
			where
			\begin{equation}\label{Part-1-28}
					a\left((u,y),(\varphi,\psi)\right)=
				\int_0^1 \left(	\rho_1u\overline{\varphi}+\rho_2 y\overline{\psi}+k_1 \left(u_x+y \right)\left( \overline{\varphi}_x+\overline{\psi} \right)+k_2y_x\overline{\psi}_{x}\right)dx+u(1)\overline{\varphi}(1)M_2(\eta,\alpha)			
			\end{equation}
			and 
			\begin{equation}\label{Part-1-29}
			L(\varphi,\psi)=\int_0^1\left( \rho_1\left(f_1+f_2\right) \overline{\varphi}+\rho_2\left(f_3+f_4\right)\overline{\psi}\right) dx-\overline{\varphi}(1) M_1(\eta,\alpha)-\overline{\varphi}(1)f_1(1)M_2(\eta,\alpha).
			\end{equation}
	 Thanks to \eqref{Part-1-28}, \eqref{Part-1-29} 	and using the fact that $M_2(\eta,\alpha)>0$, we have that $a$ is a bilinear continuous coercive form on $\left( H_L^1(0,1)\times H_*^1(0,1)\right)^2$,  and $L$ is a linear continuous form on $H_L^1(0,1)\times H_*^1(0,1)$. Then, using Lax-Milgram Theorem, we deduce that there exists $(u,y)\in H_L^1(0,1)\times H_*^1(0,1)$ unique solution of the variational problem \eqref{Part-1-27}. Applying the classical elliptic regularity we deduce that $U=(u,v,y,z,\omega)\in D({\mathcal{A}_1})$, completing the proof of the proposition.
\end{proof}$\\[0.1in]$
\noindent	Thanks to  Lumer-Phillips Theorem (see \cite{LiuZheng01, Pazy01}), we deduce that $\mathcal{A}_1$ generates a  $C_0$-semigroup of contractions $\left(e^{t\mathcal{A}_1}\right)_{t\geq 0}$ in $\mathcal{H}_1$ and therefore  problem \eqref{Part-1-07}-\eqref{Part-1-11} is well-posed. Then, we have the following result.
\begin{thm}\label{Theorem-1.3}
\rm{For any $U_0\in\mathcal{H}_1$, the problem \eqref{Part-1-07}-\eqref{Part-1-11}  admits a unique weak solution
$$U\in C\left(\mathbb{R}_{+};\mathcal{H}_1\right).$$
Moreover, if $U_0\in D\left(\mathcal{A}_1\right),$ then
$$U\in C\left(\mathbb{R}_{+};D\left(\mathcal{A}_1\right)\right) \cap C^1\left(\mathbb{R}_{+};\mathcal{H}_1\right).$$\xqed{$\square$}}
\end{thm}
\subsection{Strong stability}\label{Section-1.2-Strong-stability}
 We introduce here the notions of stability that we encounter in this work.
\begin{defi}\label{Chapter pr-35}
\rm{Assume that $\mathcal{A}_1$ is the generator of a C$_0$-semigroup of contractions $\left(e^{t\mathcal{A}_1}\right)_{t\geq0}$  on a Hilbert space  $\mathcal{H}_1$. The  $C_0$-semigroup $\left(e^{t\mathcal{A}_1}\right)_{t\geq0}$  is said to be
\begin{enumerate}
\item[1.]  strongly stable if 
$$\lim_{t\to +\infty} \|e^{t\mathcal{A}_1}x_0\|_{\mathcal{H}_1}=0, \quad\forall \ x_0\in \mathcal{H}_1;$$
\item[2.]  exponentially (or uniformly) stable if there exist two positive constants $M$ and $\epsilon$ such that
\begin{equation*}
\|e^{t\mathcal{A}_1}x_0\|_{\mathcal{H}_1} \leq Me^{-\epsilon t}\|x_0\|_{\mathcal{H}_1}, \quad
\forall\  t>0,  \ \forall \ x_0\in {\mathcal{H}_1};
\end{equation*}
\item[3.] polynomially stable if there exists two positive constants $C$ and $\alpha$ such that
\begin{equation*}
 \|e^{t\mathcal{A}_1}x_0\|_{\mathcal{H}_1}\leq C t^{-\alpha}\|\mathcal{A}_1x_0\|_{\mathcal{H}_1},  \quad\forall\ 
t>0,  \ \forall \ x_0\in D\left(\mathcal{A}_1\right).
\end{equation*}
In that case, one says that solutions of \eqref{C-u} decay at a rate $t^{-\alpha}$.
\noindent The  $C_0$-semigroup $\left(e^{t\mathcal{A}_1}\right)_{t\geq0}$  is said to be  polynomially stable with optimal decay rate $t^{-\alpha}$ (with $\alpha>0$) if it is polynomially stable with decay rate $t^{-\alpha}$ and, for any $\varepsilon>0$ small enough, there exists solutions of \eqref{C-u} which do not decay at a rate $t^{-(\alpha-\varepsilon)}$.
\end{enumerate}}
\xqed{$\square$}
\end{defi}
\noindent We now look  for necessary conditions to show the strong stability of the $C_0$-semigroup $\left(e^{t\mathcal{A}_1}\right)_{t\geq0}$. We will rely on the following result obtained by Arendt and Batty in \cite{Arendt01}. 
\begin{thm}\label{chapter pr-37}
\rm{{$\left(\textbf{Arendt and Batty in }\text{\cite{Arendt01}}\right)$}
Assume that $\mathcal{A}_1$ is the generator of a C$_0-$semigroup of contractions $\left(e^{t\mathcal{A}_1}\right)_{t\geq0}$  on a Hilbert space $\mathcal{H}_1$. If
 \begin{enumerate}
 \item[1.]  $\mathcal{A}_1$ has no pure imaginary eigenvalues,
  \item[2.]  $\sigma\left(\mathcal{A}_1\right)\cap i\mathbb{R}$ is countable,
 \end{enumerate}
where $\sigma\left(\mathcal{A}_1\right)$ denotes the spectrum of $\mathcal{A}_1$, then the $C_0$-semigroup $\left(e^{t\mathcal{A}_1}\right)_{t\geq0}$  is strongly stable.}\xqed{$\square$}
\end{thm}
\noindent Our main result in this part is the following theorem.
\begin{thm}\label{strongtheorem}
\rm{Assume that $\eta\geq 0$, then the semigroup of contractions $\left(e^{t\mathcal{A}_1}\right)_{t\geq0}$ is strongly stable on ${\mathcal{H}_1}$ in the sense that $\displaystyle{\lim_{t\to +\infty}\|e^{t{\mathcal{A}_1}}U_0\|_{{\mathcal{H}_1}}=0}$ for all $U_0\in {\mathcal{H}_1}$ if  
\begin{equation}\tag{$A_1$}
\frac{k_1}{\rho_2}\neq
\frac{\left(\frac{k_2 m_1^2}{\rho_2}-\frac{k_1 m_2^2}{\rho_1}\right)\left(\frac{k_1 m_1^2}{\rho_1}-\frac{k_2 m_2^2}{\rho_2}\right)\pi^2}{\left(\frac{k_1}{\rho_1}+\frac{k_2}{\rho_2}\right)\left(m_1^2+m_2^2\right)},\quad\ \forall\ m_1,\ m_2\in \mathbb{Z}.
\end{equation}}
\end{thm}
\noindent The argument for Theorem \ref{strongtheorem} relies on the subsequent lemmas. 
\begin{lem}\label{Theorem-1.5}
\rm{Assume that $\eta \geq 0$ and  condition $(\rm A_1)$  holds. Then, one has
$$
\ker\left(i\lambda I-{\mathcal{A}_1}\right)=\{0\}, \forall \lambda\in \mathbb{R}.
$$ }
\end{lem}
\begin{proof}
Let $U\in D({\mathcal{A}_1})$ and $\lambda \in \mathbb{R}$, such that 
\begin{equation*}
{\mathcal{A}_1} U=i\lambda U.
\end{equation*}
Equivalently, we have 
\begin{eqnarray}
v&=&i\lambda u,\label{Part-1-30}\\
{k_1}(u_x+y)_x&=&i{\rho_1}\lambda v,\label{Part-1-31}\\
z&=&i\lambda y,\label{Part-1-32}\\
{k_2}y_{xx}-{k_1}(u_x+y)&=&i{\rho_2}\lambda z,\label{Part-1-33}\\
-(\xi^2+\eta)\omega(\xi)+v(1)\mu(\xi)&=&i\lambda \omega(\xi)\label{Part-1-34}.
\end{eqnarray}
First, a straightforward computation gives 
\begin{equation*}
0=\text{Re}\left<i\lambda U,U\right>_{{\mathcal{H}_1}}=\text{Re}\left<{\mathcal{A}_1} U,U\right>_{{\mathcal{H}_1}}=-\gamma\kappa(\alpha)\int_{\mathbb{R}}(\xi^2+\eta)|\omega(\xi)|^2d\xi,
\end{equation*}
consequently, we deduce that 
\begin{equation}\label{Part-1-35}
\omega=0\quad \text{a.e.\ in}\quad \mathbb{R}.
\end{equation}
From  \eqref{Part-1-30},  \eqref{Part-1-34}, \eqref{Part-1-35} and using the fact that $U\in D({\mathcal{A}_1})$, we get
\begin{equation}\label{Part-1-36}
u_x(1)+y(1)=0\quad \text{and}\quad \lambda u(1)=0.
\end{equation}
Substituting Equations \eqref{Part-1-30}, \eqref{Part-1-32} in Equations \eqref{Part-1-31}, \eqref{Part-1-33} and using Equation \eqref{Part-1-36}, we get
\begin{eqnarray}
\rho_1\lambda^2u+{k_1}(u_x+y)_x&=&0,\label{Part-1-37}\\
\rho_2\lambda^2y+{k_2}y_{xx}-{k_1}(u_x+y)&=&0,\label{Part-1-38}\\
u(0)=\lambda u(1)=u_x(1)+y(1)=y_x(0)=y_x(1)&=&0. \label{Part-1-39}
\end{eqnarray}
If $\lambda =0$, by elementary computations, one deduces that $u=0,\ y=0$, and consequently $U=0$.
If $\lambda \neq 0$, combining Equations \eqref{Part-1-37}-\eqref{Part-1-39}, we get the following system
\begin{eqnarray}
u_{xxxx}+\left( \frac{\rho_2}{{k}_2} +\frac{\rho_1}{k_1}\right)\lambda^2 u_{xx}+\frac{\rho_1\rho_2}{k_1{k}_2}\lambda^2 \left(\lambda^2-\frac{k_1}{\rho_2}\right) u&=&0,\label{Part-1-40}\\
u(0)=u_{xx}(0)&=&0,\label{Part-1-41}\\
u(1)=u_{xx}(1)&=&0,\label{Part-1-42}\\
u_{xxx}(1)+\left(\frac{\rho_1}{k_1}+\frac{\rho_2}{k_2} \right)\lambda^2 u_x(1)&=&0.\label{Part-1-43}
\end{eqnarray}
The characteristic equitation of System \eqref{Part-1-40} is
\begin{equation}\label{eq:carac0}
P(r):= r^4+\left( \frac{\rho_2}{{k}_2} +\frac{\rho_1}{k_1}\right)\lambda^2 r^2+\frac{\rho_1\rho_2}{k_1{k}_2}\lambda^2 \left(\lambda^2-\frac{k_1}{\rho_2}\right).
\end{equation}
Setting $$P_0(\chi):=\chi^2+\left( \frac{\rho_2}{{k}_2} +\frac{\rho_1}{k_1}\right)\lambda^2 \chi+\frac{\rho_1\rho_2}{k_1{k}_2}\lambda^2 \left(\lambda^2-\frac{k_1}{\rho_2}\right).$$ 
The polynomial $P_0$ has two distinct real roots $\chi_1$ and $\chi_2$ given by:
\begin{equation*}
\chi_1={\dfrac{-\left( \frac{\rho_2}{{k}_2} +\frac{\rho_1}{k_1}\right) \lambda^2-\sqrt{\left( \frac{\rho_2}{{k}_2} -\frac{\rho_1}{k_1}\right)^2\lambda^4+\frac{4\rho_1 \lambda^2}{{k}_2}}}{2}},\  \chi_2=\dfrac{-\left( \frac{\rho_2}{{k}_2} +\frac{\rho_1}{k_1}\right) \lambda^2+\sqrt{\left( \frac{\rho_2}{{k}_2} -\frac{\rho_1}{k_1}\right)^2\lambda^4+\frac{4\rho_1 \lambda^2}{{k}_2}}}{2}.
\end{equation*}
It is clear that $\chi_1 < 0$ and the sign of $\chi_2$ depends on the value of $\lambda^2$ with respect to $\frac{k_1}{\rho_2}.$  We hence distinguish the three cases: $\lambda^2 < \frac{k_1}{\rho_2}, \ \lambda^2=\frac{k_1}{\rho_2}$, and $\lambda^2>\frac{k_1}{\rho_2}$. \\[0.1in]
\textbf{Case 1.}  $\lambda^2< \frac{k_1}{\rho_2}$:  then $\chi_2>0$ and set
$$
r_1=\sqrt{-\chi_1}\ \ \ \text{and}\ \ \ r_2=\sqrt{\chi_2}.
$$
Then $P$ has four distinct roots $i r_1,\ -i r_1,\ r_2,\ - r_2$ and the general solution of \eqref{Part-1-40} is given by 
$$u(x)=c_1 \sin(r_1 x)+c_2 \cos(r_1 x)+c_3 \sinh(r_2 x) +c_4 \cosh(r_2 x),$$
 where $c_j\in \mathbb{C},\ \forall j=1,\ldots, 4.$ Using the boundary condition \eqref{Part-1-41} and the fact that $r_1^2+r_2^2\neq0$,  we get $c_2=c_4=0$, hence
$$
u(x)=c_1\sin(r_1x)+c_3\sinh(r_2x).
$$ 
Using the boundary conditions \eqref{Part-1-42} and \eqref{Part-1-43} and the fact that $\sinh(r_2)\neq0$, we get
$$
c_3=0,\ c_1\sin(r_1)=0,\ c_1\cos(r_1)=0,
$$
yielding that  $c_1=0$. Therefore System \eqref{Part-1-40}-\eqref{Part-1-43} admits only the zero and the proof of the lemma  is complete.\\[0.1in]
\textbf{Case 2.} $\lambda^2=\frac{k_1}{\rho_2}$:  in this case $\chi_2=0$, and one gets that 
$$
r_1=\sqrt{-\chi_1}=\sqrt{\left( \frac{\rho_2}{{k}_2} +\frac{\rho_1}{k_1}\right) \frac{k_1}{\rho_2}}.
$$
Then $P$ has two simple roots $ ir_1,\ -i r_1$, and $0$ as a double root. Hence the general solution of \eqref{Part-1-40} is
$$u(x)=c_1 \sin(r_1 x)+c_2 \cos(r_1 x)+c_3 x +c_4,$$
where $c_j\in\mathbb{C},\ j=1,\ldots,4$. From the boundary condition \eqref{Part-1-41}, we get $c_2=c_4=0$. Moreover, from boundary conditions \eqref{Part-1-42} and \eqref{Part-1-43}, we get
$$
c_3=0,\ c_1 \sin(r_1) =0.
$$
Assume first that  $\sin(r_1)=0$. It follows that 
\begin{equation*}
\sqrt{\left( \frac{\rho_2}{{k}_2} +\frac{\rho_1}{k_1}\right) \frac{k_1}{\rho_2}}=r_1=m_1\pi,\quad \text{where}\ m_1\in \mathbb{N}^*.
\end{equation*}
Therefore, after choosing $m_2=0$, one gets that
\begin{equation*}
\frac{k_1}{k_2}+\frac{\rho_1}{\rho_2} =m_1^2\pi^2,
\end{equation*}
which contradicts $\left(A_1\right)$.  Hence $\sin(r_1)\neq0$. It implies that $c_1=0$ and $u=0$. Consequently $U=0$ and one gets the conclusion.
\\[0.1in]
\textbf{Case 3.}  $\lambda^2> \frac{k_1}{\rho_2}$:  then $\chi_2<0$ and set 
 \begin{equation}\label{Part-1-44}
  r_1=\sqrt{-\chi_1}\ \ \ \text{and}\ \ \ r_2=\sqrt{-\chi_2}.
 \end{equation}
Then $P$ has again four distinct roots $i r_1,\ -i r_1,\ ir_2,-i r_2$. The general solution of \eqref{Part-1-40} is given by 
$$u(x)=c_1 \sin(r_1 x)+c_2 \cos(r_1 x)+c_3 \sin(r_2 x) +c_4 \cos(r_2 x),$$
  where $c_j\in \mathbb{C},\ \forall j=1,\ldots, 4.$ Using boundary conditions \eqref{Part-1-41} and the fact that $r_1^2-r_2^2\neq0$, we get $c_2=c_4=0$, hence
$$
u(x)=c_1\sin(r_1x)+c_3\sin(r_2x).
$$ 
Assume that  $\sin(r_1)=0$ and $\sin(r_2)=0$. It follows that 
\begin{equation}\label{Part-1-45}
r_1=m_1\pi\quad \text{and}\quad r_2=m_2\pi,\quad \text{where}\ m_1,m_2\in \mathbb{N}^*.
\end{equation}
From \eqref{Part-1-44} and \eqref{Part-1-45}, we get
\begin{equation}\label{Part-1-46}
r_1^2+r_2^2=\left(m_1^2+m_2^2\right)\pi^2=\left(\frac{\rho_1}{k_1}+\frac{\rho_2}{k_2} \right)\lambda^2\ \ \ \text{and}\ \ \  r_1^2r_2^2=m_1^2m_2^2\pi^4=\frac{\rho_1\rho_2}{k_1k_2}\lambda^2\left(\lambda^2-\frac{k_1}{\rho_2}\right).
\end{equation}
From \eqref{Part-1-46}, we get
\begin{equation*}
\frac{k_1}{\rho_2}=
\frac{\left(\frac{k_2 m_1^2}{\rho_2}-\frac{k_1 m_2^2}{\rho_1}\right)\left(\frac{k_1 m_1^2}{\rho_1}-\frac{k_2 m_2^2}{\rho_2}\right)\pi^2}{\left(\frac{k_1}{\rho_1}+\frac{k_2}{\rho_2}\right)\left(m_1^2+m_2^2\right)},
\end{equation*}
which contradicts $\left(A_1\right)$.  Hence, $\sin(r_1)\neq0$ or $\sin(r_2)\neq0$. Using boundary conditions \eqref{Part-1-42} and \eqref{Part-1-43}, we can easyly check that $u=0$. Consequently $U=0$ and the conclusion follows.
\end{proof}
\begin{lem}\label{Theorem-1.6}
\rm{Assume that $\eta=0$. Then, the operator $-{\mathcal{A}_1}$ is not invertible and consequently $0\in \sigma({\mathcal{A}_1})$.}
\end{lem}
\begin{proof}
Let $F=\left(\sin(x),0,0,0,0\right)\in {\mathcal{H}_1}$, and assume that there exists $U=(u,v,y,z,\omega)\in D({\mathcal{A}_1})$ such that 
$$
-{\mathcal{A}_1} U= F,
$$ 
it follows that 
\begin{equation*}
v=-\sin(x)\ \ \text{in}\ \ (0,1)\ \ \ \text{and}\ \ \ \xi^2\omega+\sin(1)\mu(\xi)=0.
\end{equation*}
Hence,  we deduce that $\omega(\xi)=\xi^{\frac{2\alpha-5}{2}}\sin(1)\notin L^2(\mathbb{R})$, which contradicts the fact that $U\in D({\mathcal{A}_1})$. Consequently, the operator $-{\mathcal{A}_1}$ is not invertible, as claimed.
\end{proof}$\\[0.1in]$
\noindent The following lemma is a technical result to be used in the proof of Lemma~\ref{Theorem-1.7} given below.
\begin{lem}\label{Theorem-2.7-12}
\rm{Assume that condition $(\rm A_1)$  holds and assume that either 
$(\eta,\lambda)\in\mathbb{R}_+^*\times\mathbb{R}$ or $\eta=0$ and $\lambda\in \mathbb{R}^*$. Then, for any $F=(F_1,F_2)\in (L^2(0,1))^2$, the following system 
				\begin{equation}\label{Part-1-56_001}
			\left\{\begin{array}{lllll}
			\displaystyle{ \lambda^2 {u}+\frac{k_1}{\rho_1}\left( {u}_x+ {y}\right)_x}&=&\displaystyle{ F_1,}
\\ \noalign{\medskip}
			\displaystyle{ \lambda^2  {y}+\frac{k_2}{\rho_2} {y}_{xx}-\frac{k_1}{\rho_2}\left( {u}_x+ {y}\right)}&=&\displaystyle{ F_2,}
\\ \noalign{\medskip}
			 \displaystyle{ {u}(0)={y}_x(0)= {y}_x(1)}&=&0,
\\ \noalign{\medskip}
			\displaystyle{ -k_1\left( {u}_x(1)+ {y}(1)\right)}&=&	\mathtt{I}_1(\lambda,\eta,\alpha)  {u}(1).
			\end{array}
			\right.
			\end{equation}
	admits a unique strong solution $(u,y)\in \left(H^2(0,1)\cap H_{L}^1(0,1)\right)\times \left(H^2(0,1)\cap H_*^1(0,1)\right)$, where 
$$\mathtt{I}_1(\lambda,\eta,\alpha)=i\lambda\gamma\kappa(\alpha)\int_{\mathbb{R}}\frac{\mu^2(\xi)}{i\lambda+\xi^2+\eta}d\xi.$$
}
\end{lem}
\begin{rk}\label{Theorem-2.8}
\rm{Since $\alpha\in (0,1)$, under the assumptions of the above lemma, it is easy to check that 
$$ \left|\mathtt{I}_1(\lambda,\eta,\alpha)\right|<\infty,\qquad 
\Re\left(\mathtt{I}_1(\lambda,\eta,\alpha)\right)>0.$$}
\end{rk}
\noindent \textbf{Proof of Lemma  \ref{Theorem-2.7-12}.} 
 We distinguish two cases.\\[0.1in]
 \textbf{Case 1.} $\eta>0$ and $\lambda=0$: System \eqref{Part-1-56_001} becomes
	\begin{equation}\label{B_01}
	\left\{\begin{array}{lllll}
			\displaystyle{ -{k_1}\left( {u}_x+ {y}\right)_x}&=&-{\rho_1}\displaystyle{ F_1,}
\\ \noalign{\medskip}
			\displaystyle{-{k_2}{y}_{xx}+{k_1}\left( {u}_x+ {y}\right)}&=&\displaystyle{ -{\rho_2} F_2,}
\\ \noalign{\medskip}
			 \displaystyle{ {u}(0)={y}_x(0)= {y}_x(1)}&=&0,
\\ \noalign{\medskip}
			\displaystyle{  {u}_x(1)+ {y}(1)}&=&0.
			\end{array}
			\right.	\end{equation}
Let $(\varphi,\psi)\in H_L^1(0,1)\times H_*^1(0,1)$. Multiplying  the first and the second equations of \eqref{B_01}  by $\overline{\varphi}$ and $\overline{\psi}$ respectively,  integrating in $(0, 1)$ and taking the sum, then using by parts integration and the boundary conditions in  \eqref{B_01}, we get
			\begin{equation}\label{B_02}
			\displaystyle{k_1\int_0^1\left(u_x+y\right)\overline{\left(\varphi_{x}+\psi\right)}dx+k_2\int_0^1y_x\overline{\psi_x} dx=-\rho_1 \int_0^1F_1\overline{\varphi}dx-\rho_2  \int_0^1F_2\overline{\psi}dx }.
			\end{equation}
The left hand side of  \eqref{B_02} is a bilinear continuous coercive form on $\left( H_L^1(0,1)\times H_*^1(0,1)\right)^2$,  and  the right  hand side of  \eqref{B_02} is a linear continuous form on $H_L^1(0,1)\times H_*^1(0,1)$. Using Lax-Milgram theorem, we deduce that there exists a unique solution $(u,y)\in H_L^1(0,1)\times H_*^1(0,1)$ of the variational problem \eqref{B_02}. Hence, by applying the classical elliptic regularity we deduce that System \eqref{B_01} has a unique strong solution $(u,y)\in \left(H^2(0,1)\cap H_{L}^1(0,1)\right)\times \left(H^2(0,1)\cap H_*^1(0,1)\right)$.\\[0.1in]
 \textbf{Case 2.} $\eta=0\ \text{and}\ \lambda\in \mathbb{R}^*$: we first define the linear unbounded operator ${\mathcal{L}}$ by 
\begin{equation*}
\begin{array}{ll}
\displaystyle{ D(\mathcal{L})=\bigg\{(u,y)\in \left(H^2(0,1)\cap H_L^1(0,1)\right)\times \left(H^2(0,1)\cap H_*^1(0,1)\right),\ y_x(0)=y_x(1)=0,}\\ \noalign{\medskip}\hspace{9cm}
 \displaystyle{-k_1\left( {u}_x(1)+ {y}(1)\right)	=\mathtt{I}_1(\lambda,0,\alpha)  {u}(1)\bigg\}}
\end{array}
\end{equation*}
and 
$$
\mathcal{L}\ \mathcal{U}=\left(-\frac{k_1}{\rho_1}\left( {u}_x+ {y}\right)_x,-\frac{k_2}{\rho_2} {y}_{xx}+\frac{k_1}{\rho_2}\left( {u}_x+ {y}\right)\right),\qquad   \mathcal{U}=(u,y)\in D({\mathcal{L}}).
$$
For any $G=(G_1,G_2)\in (L^2(0,1))^2$, let us consider the following system 
	\begin{equation}\label{eq-2.45}
	\left\{\begin{array}{lllll}
			\displaystyle{ -\frac{k_1}{\rho_1}\left( {u}_x+ {y}\right)_x}&=&\displaystyle{ G_1,}
\\ \noalign{\medskip}
			\displaystyle{-\frac{k_2}{\rho_2} {y}_{xx}+\frac{k_1}{\rho_2}\left( {u}_x+ {y}\right)}&=&\displaystyle{ G_2,}
\\ \noalign{\medskip}
			 \displaystyle{ {u}(0)={y}_x(0)= {y}_x(1)}&=&0,
\\ \noalign{\medskip}
			\displaystyle{ -k_1\left( {u}_x(1)+ {y}(1)\right)}&=&	\mathtt{I}_1(\lambda,0,\alpha)  {u}(1).
			\end{array}
			\right.	\end{equation}
Let $(\varphi,\psi)\in H_L^1(0,1)\times H_*^1(0,1)$. Multiplying  the first and the second equations of \eqref{eq-2.45}  by $\rho_1\overline{\varphi}$ and $\rho_2\overline{\psi}$ respectively,  integrating in $(0, 1)$ and taking the sum, we obtain 
			\begin{equation}\label{A_01}
			\begin{array}{ll}
			\displaystyle{k_1\int_0^1\left(u_x+y\right)\overline{\left(\varphi_{x}+\psi\right)}dx+k_2\int_0^1y_x\psi_x dx-k_1[\left(u_x+y\right)\overline{\varphi}]_0^1-k_2[y_x\overline{\psi}]_0^1 }
\\ \noalign{\medskip}
\displaystyle{=\rho_1 \int_0^1G_1\overline{\varphi}dx+\rho_2 \int_0^1G_2\overline{\psi}dx }.
			\end{array}
			\end{equation}
From the boundary conditions in \eqref{eq-2.45} and the fact that  $\varphi(0)=0\ \left(\varphi\in H_L^1(0,1)\right)$, we get
\begin{equation}\label{A_02}
-k_1[\left(u_x+y\right)\overline{\varphi}]_0^1=\mathtt{I}_1(\lambda,0,\alpha)  {u}(1)\overline{\varphi}(1)\ \ \ \text{and}\ \ \ -k_2[y_x\overline{\psi}]_0^1=0
. 
\end{equation}
Inserting \eqref{A_02} in \eqref{A_01}, we get
			\begin{equation}\label{A_03}
			a\left((u,y),(\varphi,\psi)\right)=L\left(\varphi,\psi\right),\quad \forall \  \left(\varphi,\psi\right) \in H_L^1(0,1)\times H_*^1(0,1),
			\end{equation}
			where
			\begin{equation}\label{A_04}
					a\left((u,y),(\varphi,\psi)\right)=
					k_1 \int_0^1\left(u_x+y \right)\overline{\left( \varphi_x+\psi \right)}dx+k_2\int_0^1y_x\overline{\psi_{x}}dx+\mathtt{I}_1(\lambda,0,\alpha)  {u}(1)\overline{\varphi}(1)			
			\end{equation}
			and 
			\begin{equation}\label{A_05}
			L(\varphi,\psi)=\rho_1 \int_0^1G_1\overline{\varphi}dx+\rho_2\int_0^1 G_2\overline{\psi}dx.
			\end{equation}
	 Thanks to \eqref{A_04}, \eqref{A_05} 	and using Remark \ref{Theorem-2.8}, we have that $a$ is a bilinear continuous coercive form on $\left( H_L^1(0,1)\times H_*^1(0,1)\right)^2$,  and $L$ is a linear continuous form on $H_L^1(0,1)\times H_*^1(0,1)$. Then, using Lax-Milgram theorem, we deduce that there exists $(u,y)\in H_L^1(0,1)\times H_*^1(0,1)$ unique solution of the variational Problem \eqref{A_03} and 
	deduce that System \eqref{eq-2.45} has a unique strong solution $(u,y)\in D({\mathcal{L}})$. In addition, we have 
\begin{equation*}
\|(u,y)\|_{H^2(0,1)\times H^2(0,1)}\leq C\|(G_1,G_2)\|_{L^2(0,1)\times L^2(0,1)},
\end{equation*}
where $C>0$. It follows, from the above inequality and the compactness of the embeddings $H_L^1(0,1)\times H_*^1(0,1)$ into $L^2(0,1)\times L^2(0,1)$, that the inverse operator $\mathcal{L}^{-1}$ is compact in $L^2(0,1)\times L^2(0,1)$. Then, applying $\mathcal{L}^{-1}$ to \eqref{Part-1-56_001}, we get  
    \begin{equation}\label{A_06}
    	\left(\lambda^2\mathcal{L}^{-1}-I\right)\mathcal{U}=\mathcal{L}^{-1}F,\quad \text{where}\quad \mathcal{U}=(u,y)\ \text{and}\ F=(F_1,F_2).
    \end{equation}
Consequently, by Fredholm's alternative, proving the existence of $\mathcal{U}$ solution of \eqref{A_06} reduces to proving     $\ker\left(\lambda^2\mathcal{L}^{-1}-I\right)=\{0\}$. Indeed, if $\left(\phi,\chi\right)\in \ker\left(\lambda^2\mathcal{L}^{-1}-I\right)$, then $\lambda^2(\phi,\chi)-\mathcal{L}(\phi,\chi)=0.$ It follows that 
 \begin{equation}\label{A_07}\left\{
			\begin{array}{ll}
			\displaystyle{\rho_1\lambda^2\phi+{k_1}(\phi_x+\chi)_x=0},
\\ \noalign{\medskip}
\displaystyle{\rho_2\lambda^2\chi+{k_2}\chi_{xx}-{k_1}(\phi_x+\chi)=0 },
\\ \noalign{\medskip}
\displaystyle{\phi(0)=\chi_x(0)=\chi_x(1)=0,\ -k_1\left( {\phi}_x(1)+ {\chi}(1)\right)	=\mathtt{I}_1(\lambda,0,\alpha)  {\phi}(1)}.

			\end{array}\right.
			\end{equation}
Multiplying  the first and the second equations of \eqref{A_07}  by $-\overline{\phi}$ and $-\overline{\chi}$ respectively,  integrating in $(0, 1)$ and taking the sum, then using by parts integration and the boundary conditions in  \eqref{A_07}, we get 
\begin{equation*}
				-\rho_1 \lambda^2 \int_0^1|\phi|^2dx-\rho_2\lambda^2\int_0^1|\chi|^2dx+	k_1 \int_0^1\left|\phi_x+\chi \right|^2dx+k_2\int_0^1|\chi_x|^2dx+\mathtt{I}_1(\lambda,0,\alpha)  |\phi(1)|^2=0.			
\end{equation*}
Hence, we have
\begin{equation}
\text{Im}\left(\mathtt{I}_1(\lambda,\eta,\alpha)\right)  |\phi(1)|^2=\lambda\,\gamma\kappa|\left( \alpha \right) \phi(1)|^2 \int_{\mathbb{R}}\ {\frac {\left( {\xi}^{2}+
\eta \right)\, \mu^2 (\xi) }{ \left( {\xi}^{2}+\eta \right) ^{2}+{\lambda}^{2}}}d\xi=0,
\end{equation}
where $\text{Im}$ stands for the imaginary part of a complex number. 
Since $\lambda\,\gamma\kappa\left( \alpha \right) \int_{\mathbb{R}}  {\frac {\left( {\xi}^{2}+
\eta \right)\, \mu^2 (\xi) }{ \left( {\xi}^{2}+\eta \right) ^{2}+{\lambda}^{2}}}d\xi\neq0$, we get
\begin{equation}\label{A_08}
\phi(1)=0.
\end{equation}
Inserting \eqref{A_08} in \eqref{A_07}, we get
\begin{equation}\label{A_09}\left\{
			\begin{array}{ll}
			\displaystyle{\rho_1\lambda^2\phi+{k_1}(\phi_x+\chi)_x=0},
\\ \noalign{\medskip}
\displaystyle{\rho_2\lambda^2\chi+{k_2}\chi_{xx}-{k_1}(\phi_x+\chi)=0 },
\\ \noalign{\medskip}
\displaystyle{\phi(0)=\phi(1)=\chi_x(0)=\chi_x(1)=0,\  {\phi}_x(1)+ {\chi}(1)=0.}

			\end{array}\right.
			\end{equation}
It is now easy to see that if $(\phi,\chi)$ is a solution of System \eqref{A_09}, then the vector $\tilde{V}$ defined by
$$\tilde{V}=\left(\phi,i\lambda \phi,\chi,i\lambda \chi,0\right)$$
belongs to $D\left(\mathcal{A}_1\right)$, and $i\lambda \tilde{V}-\mathcal{A}_1\tilde{V}=0.$ Therefore, $\tilde{V}\in\ker\left(i\lambda I-{\mathcal{A}_1}\right)  $. Using Lemma \ref{Theorem-1.5}, we get $\tilde{V}=0$. This implies that System \eqref{A_06} admits a unique solution  due to Fredholm's alternative,  hence System \eqref{Part-1-56_001} admits a unique solution $(u,y)\in \left(H^2(0,1)\cap H_{L}^1(0,1)\right)\times \left(H^2(0,1)\cap H_*^1(0,1)\right)$.     Thus, the proof of the lemma is complete. \xqed{$\square$}$\\[0.1in]$
We use the previous lemma to deduce the following one.
\begin{lem}\label{Theorem-1.7}
\rm{Assume that either 
$(\eta,\lambda)\in\mathbb{R}_+^*\times\mathbb{R}$ or $\eta=0$ and $\lambda\in \mathbb{R}^*$. Then 
$i\lambda I-\mathcal{A}_1$ is surjective.}
\end{lem}
\begin{proof}
Let $F=(f_1,f_2,f_3,f_4,f_5) \in \mathcal{H}_1$, we look for $U=(u,v,y,z,\omega) \in D(\mathcal{A}_1)$ solution of 
			\begin{equation*}
			(i\lambda U-\mathcal{A}_1)U=F.
			\end{equation*}
Equivalently, we consider the following system
			\begin{eqnarray}
			v=i\lambda u-f_1,\ z=i\lambda y-f_3,\label{Part-1-47}\\ \noalign{\medskip}
			\omega(\xi)=i\lambda u(1)\frac{\mu(\xi)}{i\lambda+\xi^2+\eta}+\frac{f_5(\xi)}{i\lambda+\xi^2+\eta}-f_1(1)\frac{\mu(\xi)}{i\lambda+\xi^2+\eta},\label{Part-1-49}\\ \noalign{\medskip}
			\lambda^2 u+\frac{k_1}{\rho_1} \left(u_x+y\right)_x=-f_2-i\lambda f_1,\label{Part-1-50}\\ \noalign{\medskip}
			\lambda^2 y+\frac{k_2}{\rho_2} y_{xx}-\frac{k_1}{\rho_2}(u_x+y)=-f_4-i\lambda f_3,\label{Part-1-51}\\ \noalign{\medskip}
			-k_1\left(u_x(1)+y(1)\right)=	\mathtt{I}_1(\lambda,\eta,\alpha) u(1)+\mathtt{I}_2(\lambda,\eta,\alpha) f_1(1)+\mathtt{I}_3(f_5,\lambda,\eta,\alpha),\label{Part-1-52}\\ \noalign{\medskip}
			u(0)=y_x(0)=y_x(1)=0,\label{Part-1-53}
			\end{eqnarray}
where
$$\mathtt{I}_1(\lambda,\eta,\alpha)=i\lambda\gamma\kappa(\alpha)\int_{\mathbb{R}}\frac{\mu^2(\xi)}{i\lambda+\xi^2+\eta}d\xi,\ \mathtt{I}_2(\lambda,\eta,\alpha)=-\gamma\kappa(\alpha)\int_{\mathbb{R}}\frac{\mu^2(\xi)}{i\lambda+\xi^2+\eta}d\xi,$$
and
$$\mathtt{I}_3(f_5,\lambda,\eta,\alpha)=\gamma\kappa(\alpha)\int_{\mathbb{R}}\frac{\mu(\xi)f_5(\xi)}{i\lambda+\xi^2+\eta}d\xi.$$
Since $\alpha\in (0,1)$ and $f_5\in L^2(\mathbb{R})$, under the hypotheses of the lemma, it is easy to check that 
\begin{equation*}
 \left|\mathtt{I}_1(\lambda,\eta,\alpha)\right|<\infty,\qquad \left|\mathtt{I}_2(\lambda,\eta,\alpha)\right|<\infty,\  \left|\mathtt{I}_3(f_5,\lambda,\eta,\alpha)\right|<\infty,\ \text{and} \ 
\text{Re}\left(\mathtt{I}_1(\lambda,\eta,\alpha)\right)>0.
\end{equation*}
Let $\left(\varphi,\psi\right)\in \left(H^2(0,1)\cap H_L^1(0,1)\right)\times \left(H^2(0,1)\cap H_*^1(0,1)\right)$ such that
$$\varphi(0)=\varphi(1)=\psi_x(1)=\psi_x(0)=0,\ -k_1\left(\varphi_x(1)+\psi(1)\right)=\mathtt{I}_2(\lambda,\eta,\alpha) f_1(1)	+\mathtt{I}_3(f_5,\lambda,\eta,\alpha).$$
Setting $ \chi=u-\varphi$ and $ \zeta=y-\psi$ in \eqref{Part-1-50}-\eqref{Part-1-53}, we obtain 
			\begin{equation}\label{Part-1-56}
			\left\{\begin{array}{lllll}
			\displaystyle{ \lambda^2  \chi+\frac{k_1}{\rho_1}\left( \chi_x+ \zeta\right)_x}&=&\displaystyle{ -\lambda^2\varphi-\frac{k_1}{\rho_1}\left( \varphi_x+ \psi\right)_x-f_2-i\lambda f_1\in\ L^2(0,1),}
\\ \noalign{\medskip}
			\displaystyle{ \lambda^2  \zeta+\frac{k_2}{\rho_2} \zeta_{xx}-\frac{k_1}{\rho_2}\left( \chi_x+ \zeta\right)}&=&\displaystyle{ -\lambda^2\psi-\frac{k_2}{\rho_2} \psi_{xx}+\frac{k_1}{\rho_2}\left( \varphi_x+ \psi\right)-f_4-i\lambda f_3\in L^2(0,1),}
\\ \noalign{\medskip}
			\displaystyle{  \chi(0)=\zeta_x(0)= \zeta_x(1)}&=&0,
\\ \noalign{\medskip}
			\displaystyle{- k_1\left( \chi_x(1)+ \zeta(1)\right)	}&=&\mathtt{I}_1(\lambda,\eta,\alpha)  \chi(1).
			\end{array}
			\right.
			\end{equation} 
Using Lemma \ref{Theorem-2.7-12}, System \eqref{Part-1-56} has a unique solution $(\chi, \zeta)\in \left(H^2(0,1)\cap H_L^1(0,1)\right)\times \left(H^2(0,1)\cap H_*^1(0,1)\right)$.  Therefore, System \eqref{Part-1-49}-\eqref{Part-1-53} admits a  solution 
$$(u,y):=( \chi+\varphi, \zeta+\psi)\in\left(H^2(0,1)\cap H_L^1(0,1)\right)\times \left(H^2(0,1)\cap H_*^1(0,1)\right).$$ Thus, we define  $v:=i\lambda u-f_1,\ z:=i\lambda y-f_3,$ and 
			$$\omega(\xi):=i\lambda u(1)\frac{\mu(\xi)}{i\lambda+\xi^2+\eta}+\frac{f_5(\xi)}{i\lambda+\xi^2+\eta}-f_1(1)\frac{\mu(\xi)}{i\lambda+\xi^2+\eta},$$
 we conclude that the equation $\left(\ i\lambda -\mathcal{A}_1\right)U=F$ admits a  solution $U:=(u,v,y,z,\omega)\in D\left(\mathcal{A}_1\right)$, hence the thesis.
\end{proof}$\\[0.1in]$
We are now in a position to conclude the proof of Theorem \ref{strongtheorem}.\\[0.1in]
\noindent \textbf{Proof of Theorem \ref{strongtheorem}.} 
Using Lemma \ref{Theorem-1.5}, we have that ${\mathcal{A}_1}$ has non pure imaginary eigenvalues. According to Lemmas \ref{Theorem-1.5}, \ref{Theorem-1.6}, \ref{Theorem-1.7} and with the help of the closed graph theorem of Banach, we deduce that $\sigma({\mathcal{A}_1})\cap\ i\mathbb{R}=\{\emptyset\}$ if $\eta>0$ and $\sigma(\mathcal{A}_1)\cap\ i\mathbb{R}=\{0\}$ if $\eta=0$. Thus, we get the conclusion by applying Theorem \ref{chapter pr-37} of Arendt and Batty. The proof of the theorem is complete.
\xqed{$\square$}
\subsection{Lack of exponential  stability}\label{Section-1.3-Lack of exponential}
\noindent In this part, we use the classical method developed by Littman and Markus in \cite{Littman1988} (see also \cite{CurtainZwart01}), to show that the Timoshenko System \eqref{Part-1-07}-\eqref{Part-1-11}   is not exponentially stable.  

\begin{thm}\label{Theorem-1.8}
\rm{ The semigroup generated by the operator $\mathcal{A}_1$ is not exponentially stable in the energy space $\mathcal{H}_1.$}
\end{thm}
\noindent {For the proof of Theorem  \ref{Theorem-1.8},} we recall the following definitions: the growth bound 
$\omega_0\left(\mathcal{A}_1\right)$ and the 
the spectral bound $s\left(\mathcal{A}_1\right)$ of $\mathcal{A}_1$
are defined respectively as 
$$
\omega_0\left(\mathcal{A}_1\right)= \inf\left\{\omega\in \mathbb{R}:\  
\text{ there exists a constant }  M_\omega \text{ such that } \forall\ t\geq0,\ \left\|e^{t\mathcal{A}_1}\right\|_{\mathcal{L}(\mathcal{H}_1)}\leq M_\omega e^{\omega t}\right\}
$$
and 
$$
s\left(\mathcal{A}_1\right)=\sup\left\{\Re\left(\lambda\right):\ \lambda\in \sigma\left(\mathcal{A}_1\right) \right\}.
$$
Then, according to Theorem 2.1.6 and Lemma 2.1.11 in \cite{CurtainZwart01}, one has that 
$$
s\left(\mathcal{A}_1\right)\leq \omega_0\left(\mathcal{A}_1\right).
$$
By the previous results, one clearly has that $
s\left(\mathcal{A}_1\right)\leq 0$ and the theorem would follow if equality holds in the previous inequality. It therefore amounts to 
show the existence of a sequence of eigenvalues of $\mathcal{A}_1$ whose real parts tend to zero.\\[0.1in]
Since $\mathcal{A}_1$ is dissipative, we fix $\alpha_0>0$ small enough and we study the asymptotic behavior of the eigenvalues $\lambda$ of 
$\mathcal{A}_1$ in the strip 
$$S=\left\{\lambda \in \mathbb{C}:-\alpha_0\leq \text{Re}(\lambda)\leq 0\right\}.$$ 
First, we determine the characteristic equation satisfied by the eigenvalues of $\mathcal{A}_1$. For this aim, let $\lambda\in\mathbb{C}^*$ be an eigenvalue of $\mathcal{A}_1$ and let $U=\left(u,\lambda u,y,\lambda y,\omega\right)\in D(\mathcal{A}_1)$  be an associated eigenvector such that $\|U\|_{\mathcal{H}_1}=1$.
 Then, we have
\begin{eqnarray}
	k_1 u_{xx}-\rho_1 \lambda^2 u+k_1 y_x=0 \label{Part-1-59},\nline
-k_1u_x+{k}_2y_{xx}-\left(k_1 +\rho_2 \lambda^2\right) y=0\label{Part-1-60},\nline
\omega(\xi)=\frac{\lambda u(1)|\xi|^{\frac{2\alpha-1}{2}}}{\xi^2+\eta+\lambda},\label{Part-1-61}
\end{eqnarray}
with the boundary conditions
\begin{eqnarray}
u(0)=y_x(0)=y_x(1)=0,\label{Part-1-62}\nline
u_x(1)+y(1)+\frac{\gamma\kappa(\alpha)}{k_1}\int_{\mathbb{R}}|\xi|^{\frac{2\alpha-1}{2}}\omega(\xi)d\xi=0.\label{Part-1-63}
\end{eqnarray}
From \eqref{Part-1-59}, \eqref{Part-1-60} and using the boundary conditions \eqref{Part-1-62}, we get
\begin{equation}\label{Part-1-64}
-\left(\frac{k_1 }{k_2}+\frac{\rho_2 }{k_2} \lambda^2\right) y(1)=	 u_{xxx}(1)+\left(\frac{k_1}{k_2}-\frac{\rho_1  }{k_1}\lambda^2\right) u_x(1).
\end{equation}
Inserting \eqref{Part-1-61} and \eqref{Part-1-64} in \eqref{Part-1-63}, we get
\begin{equation}\label{Part-1-65}
 u_{xxx}(1)- \left(\frac{\rho_1}{k_1}+\frac{\rho_2}{k_2} \right)\lambda^2 u_x(1)-\frac{\rho_2 \lambda^2+k_1 }{k_1{k}_2}\lambda\gamma\kappa(\alpha)  u(1) \int_{\mathbb{R}}\frac{ \left|\xi\right|^{2\alpha-1}}{|\xi|^2+\eta+\lambda}d\xi=0.
\end{equation}
Therefore,  from \eqref{Part-1-59}, \eqref{Part-1-60}, \eqref{Part-1-62} and \eqref{Part-1-65}, we have 
\begin{equation}\label{Part-1-66}
	\left\{
	\begin{array}{ll}
	\displaystyle{u_{xxxx}-\left( \frac{\rho_2}{{k}_2} +\frac{\rho_1}{k_1}\right)\lambda^2 u_{xx}+\frac{\rho_1\rho_2}{k_1{k}_2}\lambda^2 \left(\lambda^2+\frac{k_1}{\rho_2}\right) u= 0,}\\  \\
	\displaystyle{u\left(0\right)=u_{xx}(0)=0,\  u_{xx}(1)-\frac{\rho_1}{k_1} \lambda^2 u(1)=0},   \\ \\
	\displaystyle{ u_{xxx}(1)- \left(\frac{\rho_1}{k_1}+\frac{\rho_2}{k_2} \right)\lambda^2 u_x(1)-\frac{\rho_2 \lambda^2+k_1 }{k_1{k}_2} \lambda\gamma\kappa(\alpha) u(1) \int_{\mathbb{R}}\frac{ \left|\xi\right|^{2\alpha-1}}{\xi^2+\eta+\lambda}d\xi=0.}

	\end{array}\right.
\end{equation}
The characteristic equation associated with System~\eqref{Part-1-66} is given by 
\begin{equation*}
Q(r):= r^4-\left( \frac{\rho_2}{{k}_2} +\frac{\rho_1}{k_1}\right)\lambda^2 r^2+\frac{\rho_1\rho_2}{k_1{k}_2}\lambda^2 \left(\lambda^2+\frac{k_1}{\rho_2}\right)=0.
\end{equation*}
In order to proceed, we set the following notation. Here and below, in the case where $z$ is a non zero non-real number, we define (and denote) by  $\sqrt{z}$ the square root of $z$, i.e., the unique complex number with positive real part whose square is equal to $z$. \\[0.1in]
Our aim is to study the asymptotic behavior of the large eigenvalues $\lambda $ of ${\mathcal{A}_1}$ in $S$. A careful examination shows that $Q$ admits four distinct roots if $(\rho_2k_1-\rho_1k_2)^2\lambda^2\neq 4\rho_1k_1^2k_2$. In case of equal wave propagation speed (i.e. $\rho_2k_1-\rho_1k_2=0$), this is automatically true and,  in case of different wave propagation speeds, this again holds true by taking $\lambda$ large enough. Hence, the general solution of  \eqref{Part-1-66} is given by
\begin{equation}\label{ggnb}
u(x)=\sum_{j=1}^4K_j e^{r_j x},
\end{equation}
where the $r_j$'s denote the four distinct roots of $Q$, $K_j\in \mathbb{C}$ for all $j=1,\ldots,4$ and
\begin{equation}\label{Part-1-67}
\left\{\begin{array}{ll}
\displaystyle{r_1(\lambda)=\lambda\sqrt{\dfrac{\left( \frac{\rho_2}{{k}_2} +\frac{\rho_1}{k_1}\right)+\sqrt{\left( \frac{\rho_2}{{k}_2} -\frac{\rho_1}{k_1}\right)^2-\frac{4\rho_1}{{k}_2\lambda^2}}}{2}}, \quad r_3(\lambda)=-r_1(\lambda)},
\\ \noalign{\medskip}
 \displaystyle{r_2(\lambda)=\lambda\sqrt{\dfrac{\left( \frac{\rho_2}{{k}_2} +\frac{\rho_1}{k_1}\right)-\sqrt{\left( \frac{\rho_2}{{k}_2} -\frac{\rho_1}{k_1}\right)^2-\frac{4\rho_1}{{k}_2\lambda^2}}}{2}},\quad  r_4(\lambda)=-r_2(\lambda). }

\end{array}\right.
\end{equation}
Here and below, for simplicity we denote $r_j(\lambda)$ by $r_j$. Equation \eqref{ggnb} can be written in the form
 $$u\left(x\right)=c_1 \sinh(r_1 x)+c_2\sinh(r_2 x)+c_3\cosh(r_1 x)+c_4 \cosh(r_2 x),$$ where $c_j\in\mathbb{C}$ for all $j=1,\ldots,4.$
From the boundary conditions in \eqref{Part-1-66} at $x=0$, for $\lambda$ large enough, we get $c_3=c_4=0$. Consequently,
$$u(x)=c_1\sinh(r_1 x)+c_2\sinh(r_2 x).$$
 Moreover, the boundary conditions in \eqref{Part-1-66} at $x=1$ can be expressed by  $$M C=0,$$ where
\begin{equation*}
	M=\begin{pmatrix}
f(r_1)	\sinh(r_1)& f(r_2)\sinh(r_2)\\ \noalign{\medskip}
	g(r_1)\cosh(r_1)+\mathcal{R}_\lambda\sinh(r_1)& g(r_2)\cosh(r_2)+\mathcal{R}_{\lambda}\sinh(r_2)
	\end{pmatrix},\quad C=\begin{pmatrix}
	c_1\\  c_2
\end{pmatrix},
\end{equation*}
and
\begin{equation}\label{Part-1-68}
f(r)=r^2-\frac{\rho_1}{k_1} \lambda^2,\  g (r)= \left(r^2-\left(\frac{\rho_1}{k_1}+\frac{\rho_2}{k_2} \right)\lambda^2\right) r,\ \mathcal{R}_{\lambda}=-\frac{\rho_2 \lambda^2+k_1 }{k_1{k}_2} \lambda\gamma\kappa(\alpha)  \int_{\mathbb{R}}\frac{ \left|\xi\right|^{2\alpha-1}}{\xi^2+\eta+\lambda}d\xi.
\end{equation}
Denoting the determinant of a matrix $M$ by $\det(M)$, one gets that
\begin{equation*}
\det\left(M\right)=\left(f(r_1)-f(r_2) \right) \mathcal{R}_\lambda\sinh(r_1)\sinh(r_2)+f(r_1)g(r_2)\sinh(r_1)\cosh(r_2)-f(r_2)g(r_1)\sinh(r_2)\cosh(r_1).
	\end{equation*}
Equation \eqref{Part-1-66}  admits a non trivial solution if and only if $\det(M)=0$. Next, for the proof of Theorem \ref{Theorem-1.8}, we  recall Lemma 2.1 stated in \cite{Benaissa2017}.
\begin{lem}\label{Theorem-1.9}
\rm{Let $\lambda\in D=\left\{\lambda\in\mathbb{C}\ |\ \text{Re}\left\{\lambda\right\}+\eta>0\right\}\cup \left\{\lambda\in\mathbb{C}\ |\ \text{Im}\left\{\lambda\right\}\neq 0\right\},$ then 
$$\kappa(\alpha)  \int_{\mathbb{R}}\frac{ \left|\xi\right|^{2\alpha-1}}{\xi^2+\eta+\lambda}d\xi=   \left(\lambda+\eta\right)^{\alpha-1}.$$\xqed{$\square$}}
	\end{lem}
\begin{prop} \label{Theorem-1.10}
\rm{Assume that $\frac{k_1}{\rho_1}\neq\frac{k_2}{\rho_2}$. Then there exist $n_0\in \mathbb{N}$ sufficiently large  and two sequences $\left(\lambda^{\left(0\right)}_n\right)_{ |n|\geq n_0} $ and $\left(\lambda^{\left(1\right)}_n\right)_{ |n|\geq n_0} $ of simple roots of $\det(M)$ (that are also simple eigenvalues of $\mathcal{A}_1$) satisfying the following asymptotic behavior:
\begin{equation}\label{Part-1-69}
		\displaystyle{\lambda^{\left(0\right)}_n= i n\pi\sqrt{\frac{k_1}{\rho_1} }+\frac{i\pi}{2}\sqrt{\frac{k_1}{\rho_1} }+o(1), \ \ \forall \ |n|\geq n_0}
\end{equation}
and
\begin{equation}\label{Part-1-70}
		\displaystyle{\lambda^{\left(1\right)}_{n}= i n\pi\sqrt{\frac{k_2}{\rho_2}}+o\left(1\right)}, \ \ \forall\  |n|\geq n_0.
\end{equation}}
	\end{prop}
\begin{proof}  If $\frac{k_1}{\rho_1}\neq\frac{k_2}{\rho_2}$, then using the asymptotic expansion in \eqref{Part-1-67}, we get
\begin{equation}\label{Part-1-73}
r_1=\sqrt{\frac{\rho_1}{k_1}}\lambda+O\left(\lambda^{-1}\right)\ \ \ \text{and} \ \ \ r_2=\sqrt{\frac{\rho_2}{k_2}}\lambda+O\left(\lambda^{-1}\right).
\end{equation}
First, from \eqref{Part-1-73}, we get
\begin{equation}\label{Part-1-74}
\left\{
\begin{array}{ll}

\displaystyle{ \sinh\left(r_1\right)=\sinh\left(\sqrt{\frac{\rho_1}{k_1}}\lambda\right)+O\left(\lambda^{-1}\right),\ \cosh\left(r_1\right)=\cosh\left(\sqrt{\frac{\rho_1}{k_1}}\lambda\right)+O\left(\lambda^{-1}\right),}

\\ \noalign{\medskip}

\displaystyle{ \sinh\left(r_2\right)=\sinh\left(\sqrt{\frac{\rho_2}{k_2}}\lambda\right)+O\left(\lambda^{-1}\right),\ \cosh\left(r_2\right)=\cosh\left(\sqrt{\frac{\rho_2}{k_2}}\lambda\right)+O\left(\lambda^{-1}\right).}

\end{array}
\right.
\end{equation}
Next, inserting \eqref{Part-1-73} in \eqref{Part-1-68}, we get
\begin{equation}\label{Part-1-75}
f(r_1)=O(1),\ f(r_2)=\left( \frac{\rho_2}{k_2}-\frac{\rho_1}{k_1}\right)\lambda^2+O(1),\ g(r_1)=-\frac{\rho_2}{k_2}\sqrt{\frac{\rho_1}{k_1}}\lambda^3+O(\lambda),\ g(r_2)=-\frac{\rho_1}{k_1}\sqrt{\frac{\rho_2}{k_2}}\lambda^3+O(\lambda).
\end{equation}
On the other hand, we have
\begin{equation}\label{Part-1-76}
\left(\lambda+\eta\right)^{\alpha-1}=\frac{1}{\lambda^{1-\alpha}}+O\left(\frac{1}{\lambda^{2-\alpha}} \right).
\end{equation}
From Lemma \ref{Theorem-1.9} and \eqref{Part-1-76}, we get
\begin{equation}\label{Part-1-77}
\mathcal{R}_{\lambda}=-\frac{\rho_2}{k_1 k_2}\gamma \lambda^{2+\alpha}+O\left(\lambda^{1+\alpha}\right).
\end{equation}
Therefore, from \eqref{Part-1-74}, \eqref{Part-1-75} and \eqref{Part-1-77}, we get
\begin{equation}\label{Part-1-78}
\begin{array}{ll}

\displaystyle{
\det(M)=\left( \frac{\rho_2}{k_2}-\frac{\rho_1}{k_1}\right) \frac{\rho_2}{k_2} \sqrt{\frac{\rho_1}{k_1}}\cosh\left(\lambda\sqrt{\frac{\rho_1}{k_1}} \right)\sinh\left(\lambda\sqrt{\frac{\rho_2}{k_2}} \right)\lambda^5}

\\ \noalign{\medskip}\hspace{3cm}

\displaystyle{
+ \frac{\gamma \rho_2}{k_1k_2}\left( \frac{\rho_2}{k_2}-\frac{\rho_1}{k_1}\right) \sinh\left(\lambda\sqrt{\frac{\rho_1}{k_1}} \right)\sinh\left(\lambda\sqrt{\frac{\rho_2}{k_2}} \right) \lambda^{4+\alpha}+O\left(\lambda^4\right).}
\end{array}
\end{equation}
Let $\lambda$ be a large eigenvalue of $\mathcal{A}_1$, then from \eqref{Part-1-78}, $\lambda$ is a large root of the following asymptotic equation  
\begin{equation*}
h(\lambda)=h_0(\lambda)+ \frac{h_1(\lambda)}{\lambda^{1-\alpha}} +O\left(\lambda^{-1}\right)=0,
\end{equation*}
where
$$h_0(\lambda)=\cosh\left(\lambda\sqrt{\frac{\rho_1}{k_1}} \right)\sinh\left(\lambda\sqrt{\frac{\rho_2}{k_2}} \right),\ \ h_1(\lambda)=  \frac{ \gamma }{\sqrt{\rho_1 k_1}}  \sinh\left(\lambda\sqrt{\frac{\rho_1}{k_1}} \right)\sinh\left(\lambda\sqrt{\frac{\rho_2}{k_2}} \right).$$
Note that $h_0$ and  $h_1$ remains bounded in the strip $-\alpha_0\leq \Re(\lambda)\leq 0.$ The roots of  $h_0$ are given by
$$\mu_n^{(0)}= i n\pi\sqrt{\frac{k_1}{\rho_1} }+\frac{i\pi}{2}\sqrt{\frac{k_1}{\rho_1}} \ \ \ \text{and}\  \ \ \mu_n^{(1)}=i n\pi\sqrt{\frac{k_2}{\rho_2}},\quad n\in \mathbb{Z}.$$
Finally, with the help of Rouch\'{e}'s Theorem,  there exists $n_0\in \mathbb{N}^*$ large enough, such that $\forall\ |n|\geq n_0\ \ \left( n\in \mathbb{Z}^*\right),$   the large roots of $h$, denoted by $\lambda_n^{(0)},\ \lambda_n^{(1)}$, are close to those of $h_0$, that is
\begin{equation*}
\lambda^{(0)}_n= i n\pi\sqrt{\frac{k_1}{\rho_1} }+\frac{i\pi}{2}\sqrt{\frac{k_1}{\rho_1}}+\epsilon_n\ \text{ and }\ 
\lambda^{(1)}_n= i n\pi\sqrt{\frac{k_2}{\rho_2}}+\varepsilon_n,\qquad  \lim_{|n|\to+\infty}\epsilon_n=0.
\end{equation*}
Consequently, we get \eqref{Part-1-69} and \eqref{Part-1-70}.    Thus, the proof of the proposition  is complete. 
\end{proof}
\begin{prop} \label{Theorem-1.10new}
\rm{Assume that $\frac{k_1}{\rho_1}=\frac{k_2}{\rho_2}$. Then there exist $n_0\in \mathbb{N}$ sufficiently large  and two sequences $\left(\lambda_{1,n}\right)_{ |n|\geq n_0} $ and $\left(\lambda_{2,n}\right)_{ |n|\geq n_0} $ of simple roots of $\det(M)$ (that are also simple eigenvalues of $\mathcal{A}_1$) satisfying the following asymptotic behavior:\\[0.1in]
\textbf{Case 1.} If $\sqrt{\frac{k_1}{k_2}}\neq k\pi$, $k\in\mathbb{Z}^*$, then
\begin{equation}\label{eq-3.15}
		\displaystyle{\lambda_{1,n}=i n\pi\sqrt{\frac{k_1}{\rho_1}}+\frac{\gamma\, \left(1-\cos\left(\sqrt{\frac{k_1}{k_2}}\right)\right)\left(-\sin\left(\frac{\pi \alpha}{2}\right)+i\cos\left(\frac{\pi \alpha}{2}\right)\right)}{2 \sqrt{\rho_1^{1+\alpha} k_1^{1-\alpha}}\left(n \pi\right)^{1-\alpha}}+o\left(n^{-1+\alpha}\right)
, \ \ \forall \ |n|\geq n_0}
\end{equation}
and
\begin{equation}\label{eq-3.16}
\lambda_{2,n}= i n\pi\sqrt{\frac{k_1}{\rho_1}}+\frac{i\pi \sqrt{\frac{k_1}{\rho_1}}}{2}+\frac{\gamma\, \left(1+\cos\left(\sqrt{\frac{k_1}{k_2}}\right)\right)\left(-\sin\left(\frac{\pi \alpha}{2}\right)+i\cos\left(\frac{\pi \alpha}{2}\right)\right)}{2 \sqrt{\rho_1^{1+\alpha} k_1^{1-\alpha}}\left(n \pi\right)^{1-\alpha}}+o\left(n^{-1+\alpha}\right)
, \ \ \forall\  |n|\geq n_0.
\end{equation}
\textbf{Case 2.} If $\sqrt{\frac{k_1}{k_2}}=2 k\pi$, $k\in\mathbb{Z}^*$, then
\begin{equation}\label{eq-3.17}
		\displaystyle{\lambda_{1,n}=i n\pi \sqrt{\frac{k_1}{\rho_1}} +\frac{\frac{i k_1}{{k_2}}\sqrt{\frac{k_1}{\rho_1}}}{8  n\pi}-\frac{ i \frac{k_1^2}{k_2^2}\sqrt{\frac{k_1}{\rho_1}} }{128\pi^3 n^3}+{\frac {\gamma \sqrt{k_1^{5+\alpha}}  \left( i\cos \left(\frac{\pi\alpha}{2}  \right) -\sin \left( \frac{\pi\alpha}{2} \right)  \right) }{256 k_2^3\sqrt{\rho_1^{1+\alpha}}\,{
\pi}^{5-\alpha}{n}^{5-\alpha}}}+O\left(n^{-5}\right)
, \ \ \forall \ |n|\geq n_0}
\end{equation}
and
\begin{equation}\label{eq-3.18}
	\lambda_{2,n}=	i n\pi\sqrt{\frac{k_1}{\rho_1}}+\frac{i\pi \sqrt{\frac{k_1}{\rho_1}}}{2}+\frac{\gamma\, \left(-\sin\left(\frac{\pi \alpha}{2}\right)+i\cos\left(\frac{\pi \alpha}{2}\right)\right)}{ \sqrt{\rho_1^{1+\alpha} k_1^{1-\alpha}}\left(n \pi\right)^{1-\alpha}}+o\left(n^{-1+\alpha}\right) , \ \ \forall\  |n|\geq n_0.
\end{equation}
\textbf{Case 3.} If $\sqrt{\frac{k_1}{k_2}}=(2 k+1)\pi$, $k\in\mathbb{Z}^*$, then
\begin{equation}\label{eq-3.19}
		\displaystyle{\lambda_{1,n}=i n\pi\sqrt{\frac{k_1}{\rho_1}}+\frac{\gamma\,\left(-\sin\left(\frac{\pi \alpha}{2}\right)+i\cos\left(\frac{\pi \alpha}{2}\right)\right)}{\sqrt{\rho_1^{1+\alpha} k_1^{1-\alpha}}\left(n \pi\right)^{1-\alpha}}+o \left( 
 {{n}^{-1+\alpha}} \right), \ \ \forall \ |n|\geq n_0}
\end{equation}
and
\begin{equation}\label{eq-3.20}
\begin{array}{ll}
		\displaystyle{\lambda_{2,n}=i n\pi\sqrt{\frac{k_1}{\rho_1}}+\frac{i\pi \sqrt{\frac{k_1}{\rho_1}}}{2} +\frac{\frac{i k_1}{{k_2}}\sqrt{\frac{k_1}{\rho_1}}}{8  n\pi}-\frac{\frac{ik_1}{k_2}\sqrt{\frac{k_1}{\rho_1}} }{16\pi n^2}+\frac{i\frac{k_1}{k_2} \sqrt{\frac{k_1}{\rho_1}} \left(4\pi^2-\frac{k_1}{k_2}\right) }{128\pi^3 n^3}}
\\ \\ \hspace{1cm}
\displaystyle{
-\frac{i\frac{k_1}{k_2} \sqrt{\frac{k_1}{\rho_1}} \left(4\pi^2-\frac{3k_1}{k_2}\right) }{256\pi^3 n^4}+{\frac {\gamma \sqrt{k_1^{5+\alpha}}  \left( i\cos \left(\frac{\pi\alpha}{2}  \right) -\sin \left( \frac{\pi\alpha}{2} \right)  \right) }{256 k_2^3\sqrt{\rho_1^{1+\alpha}}\,{
\pi}^{5-\alpha}{n}^{5-\alpha}}}+O\left(n^{-5}\right)
,\ \ \forall\ |n|\geq n_0.}
\end{array}
\end{equation}	}		
\end{prop}	
\begin{proof}  Assume that $\frac{k_1}{\rho_1}=\frac{k_2}{\rho_2}$, then from \eqref{Part-1-67}, \eqref{Part-1-68}, and Lemma \ref{Theorem-1.9}, we get
\begin{equation}\label{eq-3.13}
r_1=\lambda\sqrt{\frac{\rho_1}{k_1}+\frac{i}{\lambda}\sqrt{\frac{\rho_1}{{k}_2}}}
, \quad r_2=\lambda\sqrt{\frac{\rho_1}{k_1}-\frac{i}{\lambda}\sqrt{\frac{\rho_1}{{k}_2}}},
\end{equation}
and
\begin{equation*}
f(r_1)=i\lambda\sqrt{\frac{\rho_1}{k_2}},\ f(r_2)=-f(r_1),\ g(r_1)=-r_1 r_2^2,\ g(r_2)=-r_1^2 r_2,\ \mathcal{R}_{\lambda}=-\gamma\lambda\left(\frac{\rho_1 \lambda^2}{k_1^2} +\frac{1}{k_2} \right)    \left(\lambda+\eta\right)^{\alpha-1}.
\end{equation*}
Hence, we have
\begin{equation}\label{eq-3.14}
\begin{array}{ll}
		\displaystyle{
\det\left(M\right)=-i\lambda\sqrt{\frac{\rho_1}{k_2}}\bigg[2 \gamma\lambda\left(\frac{\rho_1 \lambda^2}{k_1^2} +\frac{1}{k_2} \right)    \left(\lambda+\eta\right)^{\alpha-1}\sinh(r_1)\sinh(r_2)}
 \\ \\ \hspace{3.2cm}
\displaystyle{+r_1^2 r_2\sinh(r_1)\cosh(r_2)+r_1 r_2^2\sinh(r_2)\cosh(r_1)\bigg].}
\end{array}
	\end{equation}
We divide the proof into four steps:\\ [0.1in]
\textbf{Step 1.} In this step, we prove that the  eigenvalues of $\mathcal{A}_1$, are roots of the following function
			\begin{equation}\label{eq-3.21}
		\begin{array}{ll}
		\displaystyle{
H\left(\lambda\right)=\frac{\gamma }{\sqrt{\rho_1 k_1}\lambda^{1-\alpha}} \left(1+ O\left(\lambda^{-1} \right) \right)\left(\cosh(r_1+r_2)-\cosh(r_1-r_2)\right)}

 \\ \\ \hspace{1cm}
\displaystyle{+ \left(1+\frac{5k_1^2}{8 \rho_1 k_2 \lambda^2}+O\left(\lambda^{-3}\right)\right)\sinh(r_1+r_2)+\left(\frac{i k_1}{2\sqrt{\rho_1 k_2}\lambda}+O\left(\lambda^{-3}\right)\right)\sinh(r_1-r_2).}

\end{array}
\end{equation}
First, using the asymptotic expansion in \eqref{eq-3.13}, we get
\begin{equation}\label{eq-3.22}
r_1=\lambda\sqrt{\frac{\rho_1}{k_1}} +\frac{i \sqrt{\frac{k_1}{k_2}} }{2}+\frac{{\frac{k_1}{k_2}}\sqrt{\frac{k_1}{\rho_1}}}{8\lambda}+O \left(\lambda^{-2} \right) \ \ \ \text{and}\ \ \ r_2=\lambda\sqrt{\frac{\rho_1}{k_1}} -\frac{i \sqrt{\frac{k_1}{k_2}} }{2}+\frac{{\frac{k_1}{k_2}}\sqrt{\frac{k_1}{\rho_1}}}{8\lambda}+O \left( \lambda^{-2} \right). 
\end{equation}
From  \eqref{eq-3.22}, we get
\begin{equation}\label{eq-3.23A}
\left\{
\begin{array}{ll}
		\displaystyle{
r_1^2 r_2=\left(\frac{\rho_1}{k_1}\right)^{\frac{3}{2}}\lambda^3 \left(1+\frac{i k_1}{2\sqrt{\rho_1 k_2}\lambda}+\frac{5k_1^2}{8 \rho_1 k_2 \lambda^2}+O\left(\lambda^{-3}\right)\right) ,}
\\ \noalign{\medskip}
 \displaystyle{r_1 r_2^2=\left(\frac{\rho_1}{k_1}\right)^{\frac{3}{2}}\lambda^3 \left(1-\frac{i k_1}{2\sqrt{\rho_1 k_2}\lambda}+\frac{5k_1^2}{8 \rho_1 k_2 \lambda^2}+O\left(\lambda^{-3}\right)\right).}
 \end{array}
 \right.
\end{equation}
Next,  using the asymptotic expansion, we get
$$\frac{1}{(\lambda+\eta)^{1-\alpha}}=\frac{1}{\lambda^{1-\alpha}}\left( 1+O \left(\lambda^{-1} \right)  \right). $$
Consequently, we get
\begin{equation}\label{eq-3.23}
2 \gamma\lambda\left(\frac{\rho_1 \lambda^2}{k_1^2} +\frac{1}{k_2} \right)    \left(\lambda+\eta\right)^{\alpha-1}=\frac{2\gamma \rho_1 \lambda^{2+\alpha}}{k_1^2} \left(1+ O\left(\lambda^{-1} \right) \right). 
\end{equation}
Inserting \eqref{eq-3.23A} and \eqref{eq-3.23} in \eqref{eq-3.14}, then using the fact that 
\begin{equation*}
\left\{
\begin{array}{ll}
\displaystyle{\sinh(r_1) \cosh(r_2)+\cosh(r_1)\sinh(r_2) =\sinh(r_1+r_2),}
\\ \noalign{\medskip}
\displaystyle{ \sinh(r_1) \cosh(r_2)- \cosh(r_1)\sinh(r_2) =\sinh(r_1-r_2),}
\\ \noalign{\medskip}
\displaystyle{2\sinh(r_1)\sinh(r_2)=\cosh(r_1+r_2)-\cosh(r_1-r_2), }

\end{array}
\right.
\end{equation*}
we get
\begin{equation*}
\begin{array}{ll}
		\displaystyle{
\det\left(M\right)=-i\lambda^4 \left(\frac{\rho_1}{k_1}\right)^{\frac{3}{2}}\sqrt{\frac{\rho_1}{k_2}}\bigg[\frac{\gamma }{\sqrt{\rho_1 k_1}\lambda^{1-\alpha}} \left(1+O \left(\lambda^{-1} \right) \right)\left(\cosh(r_1+r_2)-\cosh(r_1-r_2)\right)}

 \\ \\ \hspace{2cm}
\displaystyle{+ \left(1+\frac{5k_1^2}{8 \rho_1 k_2 \lambda^2}+O\left(\lambda^{-3}\right)\right)\sinh(r_1+r_2)+\left(\frac{i k_1}{2\sqrt{\rho_1 k_2}\lambda}+O\left(\lambda^{-3}\right)\right)\sinh(r_1-r_2)\bigg].}

\end{array}
	\end{equation*}
Equation \eqref{Part-1-66}  admits a non trivial solution if and only if $\displaystyle{\det\left(M\right)}=0$, i.e., if and only if the eigenvalues of $\mathcal{A}_1$ are roots of the function $H$, defined by
	\begin{equation*}
\begin{array}{ll}
		\displaystyle{
H\left(\lambda\right)=\frac{\gamma  }{\sqrt{\rho_1 k_1}\lambda^{1-\alpha}} \left(1+ O\left(\lambda^{-1} \right) \right)\left(\cosh(r_1+r_2)-\cosh(r_1-r_2)\right)}

 \\ \\ \hspace{1cm}
\displaystyle{+ \left(1+\frac{5k_1^2}{8 \rho_1 k_2 \lambda^2}+O\left(\lambda^{-3}\right)\right)\sinh(r_1+r_2)+\left(\frac{i k_1}{2\sqrt{\rho_1 k_2}\lambda}+O\left(\lambda^{-3}\right)\right)\sinh(r_1-r_2).}

\end{array}
	\end{equation*}	
Hence, we get \eqref{eq-3.21}. \\[0.1in]
\textbf{Step 2.} We look at the roots of $H(\lambda)$. 
First, from \eqref{eq-3.22}, we have
\begin{equation}\label{eq-3.26}
r_1+r_2=2\lambda\sqrt{\frac{\rho_1}{k_1}} +\frac{{\frac{k_1}{k_2}}\sqrt{\frac{k_1}{\rho_1}}}{4\lambda}+O \left(\lambda^{-2} \right)\ \ \ \text{and}\ \ \ r_1-r_2=i \sqrt{\frac{k_1}{k_2}} +O\left(\lambda^{-2}\right).
\end{equation}From \eqref{eq-3.26} and using the fact that $\Re\{\lambda\}$ is bounded, we get
\begin{equation}\label{eq-3.27}
\left\{
\begin{array}{ll}
\displaystyle{\sinh(r_1+r_2)=\sinh\left(2\lambda\sqrt{\frac{\rho_1}{k_1}}\right)+\frac{\frac{k_1}{k_2}\sqrt{\frac{k_1}{\rho_1}}\, \cosh\left(2\lambda\sqrt{\frac{\rho_1}{k_1}}\right)}{4\lambda}+O\left(\lambda^{-2}\right),}
\\ \noalign{\medskip}
\displaystyle{\sinh(r_1-r_2)=i\sin\left(\sqrt{\frac{k_1}{k_2}}\right)+O\left(\lambda^{-2}\right),}
\\ \noalign{\medskip}
\displaystyle{\cosh(r_1+r_2)-\cosh(r_1-r_2)=\cosh\left(2\lambda\sqrt{\frac{\rho_1}{k_1}}\right)-\cos\left(\sqrt{\frac{k_1}{k_2}}\right)+O\left(\lambda^{-1}\right). }
\end{array}
\right.
\end{equation}
Next, substituting \eqref{eq-3.27} in \eqref{eq-3.21}, we get
\begin{equation}\label{eq-3.28}
\begin{array}{ll}
\displaystyle{H(\lambda)=\sinh\left(2\lambda\sqrt{\frac{\rho_1}{k_1}}\right)+\frac{\gamma\, \left(\cosh\left(2\lambda\sqrt{\frac{\rho_1}{k_1}}\right)-\cos\left(\sqrt{\frac{k_1}{k_2}}\right)\right)}{\sqrt{\rho_1 k_1}\lambda^{1-\alpha}}}
\\ \noalign{\medskip}\hspace{3cm}
\displaystyle{+\frac{k_1\left(\cosh\left(2\lambda\sqrt{\frac{\rho_1}{k_1}}\right)\sqrt{\frac{k_1}{k_2}}-2\sin\left(\sqrt{\frac{k_1}{k_2}}\right)\right)}{4\sqrt{\rho_1 k_2}\lambda}+O\left(\lambda^{-2+\alpha}\right).}
\end{array}
\end{equation}
Indeed, using Rouch\'e's Theorem, and the asymptotic Equation \eqref{eq-3.28}, it is easy to see that the large roots of $f(\lambda)$ (denoted by $\lambda_{1,n}$ and $\lambda_{2,n}$) are simple and close to those of $\sinh\left(2\lambda\sqrt{\frac{\rho_1}{k_1}}\right)$, i.e., there exists $n_0\in \mathbb{N}$, such that for all integers $|n|>n_0$, we have
\begin{eqnarray}
\lambda_{1,n}=i n\pi\sqrt{\frac{k_1}{\rho_1}} +\epsilon_{1,n},\quad \text{where }\lim_{|n|\to+\infty}\epsilon_{1,n}=0,\label{eq-3.24}\\
\lambda_{2,n}=i n\pi\sqrt{\frac{k_1}{\rho_1}}+\frac{i\pi \sqrt{\frac{k_1}{\rho_1}}}{2} +\epsilon_{2,n},\quad \text{where }\lim_{|n|\to+\infty}\epsilon_{2,n}=0.\label{eq-3.25}
\end{eqnarray}
\textbf{Step 3.}  We seek to determine $\epsilon_{1,n}$. Inserting \eqref{eq-3.24} in \eqref{eq-3.28}, we get
 \begin{equation}\label{eq-3.29}
\begin{array}{ll}
\displaystyle{H(\lambda_{1,n})=\sinh\left(2\epsilon_{1,n}\sqrt{\frac{\rho_1}{k_1}}\right)+\frac{\gamma\, \left(\cosh\left(2\epsilon_{1,n}\sqrt{\frac{\rho_1}{k_1}}\right)-\cos\left(\sqrt{\frac{k_1}{k_2}}\right)\right)}{\sqrt{\rho_1 k_1}\left(i n\pi\sqrt{\frac{k_1}{\rho_1}}\right)^{1-\alpha}}}
\\ \noalign{\medskip}\hspace{3cm}
\displaystyle{+\frac{\sqrt{\frac{k_1}{k_2}}\left(\cosh\left(2\epsilon_{1,n}\sqrt{\frac{\rho_1}{k_1}}\right)\sqrt{\frac{k_1}{k_2}}-2\sin\left(\sqrt{\frac{k_1}{k_2}}\right)\right)}{4 i n\pi}+O\left(n^{-2+\alpha}\right).}
\end{array}
\end{equation}
On the other hand, since $\lim_{|n|\to+\infty}\epsilon_{1,n}=0$, we have the asymptotic expansion
\begin{equation}\label{eq-3.30}
\sin\left(2\epsilon_{1,n}\sqrt{\frac{\rho_1}{k_1}}\right)=2\epsilon_{1,n}\sqrt{\frac{\rho_1}{k_1}}+O\left(\epsilon_{1,n}^3\right)\ \text{and}\ \cos\left(2\epsilon_{1,n}\sqrt{\frac{\rho_1}{k_1}}\right)=1+O\left(\epsilon_{1,n}^2\right).
\end{equation}
Inserting \eqref{eq-3.30} in \eqref{eq-3.29}, then using the fact that
$$i^{-1+\alpha}=\sin\left(\frac{\pi \alpha}{2}\right)-i\cos\left(\frac{\pi \alpha}{2}\right),$$
we get
  \begin{equation}\label{eq-3.31}
  \begin{array}{ll}
\displaystyle{
\epsilon_{1,n}-\frac{\gamma\, \left(1-\cos\left(\sqrt{\frac{k_1}{k_2}}\right)\right)\left(-\sin\left(\frac{\pi \alpha}{2}\right)+i\cos\left(\frac{\pi \alpha}{2}\right)\right)}{2 \sqrt{\rho_1^{1+\alpha} k_1^{1-\alpha}}\left(n \pi\right)^{1-\alpha}}}
\\ \noalign{\medskip}\hspace{1cm}
\displaystyle{-\frac{\frac{i k_1}{\sqrt{\rho_1 k_2}}\left(\sqrt{\frac{k_1}{k_2}}-2\sin\left(\sqrt{\frac{k_1}{k_2}}\right)\right)}{8  n\pi}+O\left(n^{-2+\alpha}\right)+O\left({n^{-1+\alpha}}{\epsilon_{1,n}^2}\right)+O(\epsilon_{1,n}^3)=0.}
\end{array}
\end{equation}
 We distinguish two cases:\\[0.1in]
 \textbf{Case 1.} There exists  no integer $k\in \mathbb{Z}$ such that $\sqrt{\frac{k_1}{k_2}}=2k\pi$. Then, we have
 $$1-\cos\left(\sqrt{\frac{k_1}{k_2}}\right)\neq0,$$
 therefore, from \eqref{eq-3.31}, we get
 \begin{equation}\label{eq-3.32}
 \epsilon_{1,n}=\frac{\gamma\, \left(1-\cos\left(\sqrt{\frac{k_1}{k_2}}\right)\right)\left(-\sin\left(\frac{\pi \alpha}{2}\right)+i\cos\left(\frac{\pi \alpha}{2}\right)\right)}{2 \sqrt{\rho_1^{1+\alpha} k_1^{1-\alpha}}\left(n \pi\right)^{1-\alpha}}+o\left({n^{-1+\alpha}}\right).
 \end{equation}
Substituting \eqref{eq-3.32} in \eqref{eq-3.24}, we get the estimates \eqref{eq-3.15} and \eqref{eq-3.19}.\\[0.1in]
\textbf{Case 2.} If there exists $k\in\mathbb{Z}$ such that $\sqrt{\frac{k_1}{k_2}}=2k\pi$, then
$$1-\cos\left(\sqrt{\frac{k_1}{k_2}}\right)=0\ \ \ \text{and}\ \ \ \sin\left(\sqrt{\frac{k_1}{k_2}}\right)=0,$$
 therefore, from \eqref{eq-3.31}, we get
 \begin{equation}\label{eq-3.33}
\epsilon_{1,n}=\frac{\frac{i k_1}{{k_2}}\sqrt{\frac{k_1}{\rho_1}}}{8  n\pi}+o\left(n^{-1}\right).
\end{equation}
 Inserting \eqref{eq-3.33} in \eqref{eq-3.24}, we get
 \begin{equation}\label{eq-3.34}
\lambda_{1,n}=i n\pi \sqrt{\frac{k_1}{\rho_1}} +\frac{\frac{i k_1}{{k_2}}\sqrt{\frac{k_1}{\rho_1}}}{8  n\pi}+\frac{\varepsilon_{1,n}}{n},\quad \text{where }\lim_{|n|\to+\infty}\varepsilon_{1,n}=0,
 \end{equation}
since in this case the  real part of $\lambda_{1,n}$ still does not appear, we need to increase the order of the finite expansion. So, in order to complete the proof of \eqref{eq-3.17}, we need to show that 
$$\varepsilon_{1,n}=-\frac{ i \frac{k_1^2}{k_2^2}\sqrt{\frac{k_1}{\rho_1}} }{128\pi^3 n^2}+{\frac {\gamma \sqrt{k_1^{5+\alpha}}  \left( i\cos \left(\frac{\pi\alpha}{2}  \right) -\sin \left( \frac{\pi\alpha}{2} \right)  \right) }{256 k_2^3\sqrt{\rho_1^{1+\alpha}}\,{
\pi}^{5-\alpha}{n}^{4-\alpha}}}+O\left(n^{-4}\right). $$ 
 For this aim, inserting \eqref{eq-3.34} in \eqref{eq-3.13}, then using the asymptotic expansion, we get
 \begin{equation*}
 r_1+r_2=2in\pi+\frac{2\sqrt{\frac{\rho_1}{k_1}}\varepsilon_{1,n}}{n}-\frac{3i k_1^2-16k_2\sqrt{\rho_1 k_1}\pi \varepsilon_{1,n} }{64k^2 \pi^3 n^3}+O\left(n^{-5}\right),\ r_1-r_2=i \sqrt{\frac{k_1}{k_2}}+\frac{i \left(\frac{k_1}{k_2}\right)^{\frac{3}{2}}}{8\pi^2 n^2}+O\left(n^{-4}\right).
 \end{equation*}
Therefore, we have
\begin{equation}\label{eq-3.35}
\left\{
\begin{array}{ll}
\displaystyle{\sinh(r_1+r_2)=\frac{2\sqrt{\frac{\rho_1}{k_1}}\varepsilon_{1,n}}{n}+\frac{1}{n^3}\left(\frac{4\left(\frac{\rho_1}{k_1}\right)^{\frac{3}{2}}\varepsilon_{1,n}^3}{3}-\frac{3i k_1^2-16k_2\sqrt{\rho_1 k_1}\pi \varepsilon_{1,n} }{64k^2 \pi^3 }\right)+O\left(n^{-5}\right),}
\\ \noalign{\medskip}
\displaystyle{\sinh(r_1-r_2)=\frac{i \left(\frac{k_1}{k_2}\right)^{\frac{3}{2}}}{8\pi^2 n^2}+O\left(n^{-4}\right),}
\\ \noalign{\medskip}
\displaystyle{\cosh(r_1+r_2)-\cosh(r_1-r_2)=\frac{2\rho_1\varepsilon_{1,n}}{k_1n^2}+O\left(n^{-4}\right). }
\end{array}
\right.
\end{equation}
 Inserting \eqref{eq-3.34} and \eqref{eq-3.35} in \eqref{eq-3.21}, then using the asymptotic expansion, we get
 \begin{equation*}
 H(\lambda_{1,n})=\frac{2\sqrt{\frac{\rho_1}{k_1}}}{n}\left(\varepsilon_{1,n}+\frac{ i \frac{k_1^2}{k_2^2}\sqrt{\frac{k_1}{\rho_1}} }{128\pi^3 n^2} +O\left({n^{-3+\alpha}}\right)+O\left({n^{-2}}{\varepsilon_{1,n}}\right)+O\left({n^{-2+\alpha}}{\varepsilon_{1,n}^2}\right) \right)=0.
 \end{equation*}
Consequently, we obtain
\begin{equation}\label{eq-3.36}
\varepsilon_{1,n}=-\frac{ i \frac{k_1^2}{k_2^2}\sqrt{\frac{k_1}{\rho_1}} }{128\pi^3 n^2} +\frac{\zeta_{1,n}}{n^2},\ \text{such that }  \lim_{|n|\to+\infty}\zeta_{1,n}=0.
\end{equation}
Substituting  \eqref{eq-3.36} in \eqref{eq-3.34}, we get
 \begin{equation}\label{eq-3.37}
 \lambda_{1,n}=i n\pi \sqrt{\frac{k_1}{\rho_1}} +\frac{\frac{i k_1}{{k_2}}\sqrt{\frac{k_1}{\rho_1}}}{8  n\pi}-\frac{ i \frac{k_1^2}{k_2^2}\sqrt{\frac{k_1}{\rho_1}} }{128\pi^3 n^3}+\frac{\zeta_{1,n}}{n^3}, 
 \end{equation}
again  the real part of $\lambda_{1,n}$ still does not appear, so  we need to increase the order of the finite expansion.  For this aim, inserting \eqref{eq-3.37} in \eqref{eq-3.13} and using the asymptotic expansion, we get
 \begin{equation*}
 r_1+r_2=2in\pi+\frac{2\sqrt{\frac{\rho_1}{k_1}} \zeta_{1,n}-\frac{i k_1^2}{16k_2^2\pi^3} }{ n^3}+O \left(n^{-5}\right)\ \ \ \text{and}\ \ \ r_1-r_2=i \sqrt{\frac{k_1}{k_2}}+\frac{i\left(\frac{k_1}{k_2}\right)^{\frac{3}{2}} }{8\pi^2 n^2}+O \left( n^{-4}\right).
 \end{equation*}
Therefore, we have
\begin{equation}\label{eq-3.38}
\left\{
\begin{array}{ll}
\displaystyle{\sinh(r_1+r_2)=\frac{2\sqrt{\frac{\rho_1}{k_1}} \zeta_{1,n}-\frac{i k_1^2}{16k_2^2\pi^3} }{ n^3}+O \left( n^{-5}\right),}
\\ \noalign{\medskip}
\displaystyle{\sinh(r_1-r_2)=\frac{i\left(\frac{k_1}{k_2}\right)^{\frac{3}{2}} }{8\pi^2 n^2}+O \left( n^{-4}\right),}
\\ \noalign{\medskip}
\displaystyle{\cosh(r_1+r_2)-\cosh(r_1-r_2)=\frac{\left(\frac{k_1}{k_2}\right)^{3} }{128\pi^4 n^4}+O \left( n^{-6}\right). }
\end{array}
\right.
\end{equation}
 Inserting \eqref{eq-3.37} and \eqref{eq-3.38} in \eqref{eq-3.21}, then using the asymptotic expansion, we get
 \begin{equation*}
 \frac{2\sqrt{\frac{\rho_1}{k_1}}}{n^3}\left(\zeta_{1,n}-{\frac {\gamma k_1^2\left(\frac{k_1}{\rho_1}\right)^{\frac{1+\alpha}{2}}  \left( i\cos \left(\frac{\pi\alpha}{2}  \right) -\sin \left( \frac{\pi\alpha}{2} \right)  \right) }{256 k_2^3 \,{
\pi}^{5-\alpha}{n}^{2-\alpha}}}+O\left(n^{-2}\right)\right)=0.
 \end{equation*}
Consequently, we obtain
\begin{equation}\label{eq-3.39}
\zeta_{1,n}={\frac {\gamma \sqrt{k_1^{5+\alpha}}  \left( i\cos \left(\frac{\pi\alpha}{2}  \right) -\sin \left( \frac{\pi\alpha}{2} \right)  \right) }{256 k_2^3\sqrt{\rho_1^{1+\alpha}}\,{
\pi}^{5-\alpha}{n}^{2-\alpha}}}+O\left(n^{-2}\right).
\end{equation}
Inserting \eqref{eq-3.39} in \eqref{eq-3.37}, we get the estimate \eqref{eq-3.17}. \\[0.1in]
\textbf{Step 4.}  We seek to determine $\epsilon_{2,n}$. Inserting \eqref{eq-3.25} in \eqref{eq-3.28}, we get
 \begin{equation}\label{eq-3.40}
\begin{array}{ll}
\displaystyle{H(\lambda_{2,n})=-\sinh\left(2\epsilon_{2,n}\sqrt{\frac{\rho_1}{k_1}}\right)+\frac{\gamma\, \left(-\cosh\left(2\epsilon_{2,n}\sqrt{\frac{\rho_1}{k_1}}\right)-\cos\left(\sqrt{\frac{k_1}{k_2}}\right)\right)}{\sqrt{\rho_1 k_1}\left(i n\pi\sqrt{\frac{k_1}{\rho_1}}\right)^{1-\alpha}}}
\\ \noalign{\medskip}\hspace{3cm}
\displaystyle{+\frac{\sqrt{\frac{k_1}{k_2}}\left(-\cosh\left(2\epsilon_{2,n}\sqrt{\frac{\rho_1}{k_1}}\right)\sqrt{\frac{k_1}{k_2}}-2\sin\left(\sqrt{\frac{k_1}{k_2}}\right)\right)}{4 i n\pi}+O\left({n^{-2+\alpha}}\right).}
\end{array}
\end{equation}
 On the other hand, since $\lim_{|n|\to+\infty}\epsilon_{2,n}=0$, using the asymptotic expansion in \eqref{eq-3.40}, we get
  \begin{equation}\label{eq-3.42}
  \begin{array}{ll}
\displaystyle{\epsilon_{2,n}-\frac{\gamma\, \left(1+\cos\left(\sqrt{\frac{k_1}{k_2}}\right)\right)\left(-\sin\left(\frac{\pi \alpha}{2}\right)+i\cos\left(\frac{\pi \alpha}{2}\right)\right)}{2 \sqrt{\rho_1^{1+\alpha} k_1^{1-\alpha}}\left(n \pi\right)^{1-\alpha}}}
\\ \noalign{\medskip}\hspace{1cm}
\displaystyle{-\frac{\frac{i k_1}{\sqrt{\rho_1 k_2}}\left(\sqrt{\frac{k_1}{k_2}}+2\sin\left(\sqrt{\frac{k_1}{k_2}}\right)\right)}{8  n\pi}+O\left({n^{-2+\alpha}}\right)+O\left({n^{-1+\alpha}}{\epsilon_{2,n}^2}\right)+O(\epsilon_{2,n}^3)=0.}
\end{array}
\end{equation}
 We distinguish two cases:\\[0.1in]
 \textbf{Case 1.} There exists  no integer $k\in \mathbb{Z}$ such that $\sqrt{\frac{k_1}{k_2}}=\pi+2k\pi$. Then, we have
 $$1+\cos\left(\sqrt{\frac{k_1}{k_2}}\right)\neq0,$$
 therefore, from \eqref{eq-3.42}, we get
 \begin{equation}\label{eq-3.43}
 \epsilon_{2,n}=\frac{\gamma\, \left(1+\cos\left(\sqrt{\frac{k_1}{k_2}}\right)\right)\left(-\sin\left(\frac{\pi \alpha}{2}\right)+i\cos\left(\frac{\pi \alpha}{2}\right)\right)}{2 \sqrt{\rho_1^{1+\alpha} k_1^{1-\alpha}}\left(n \pi\right)^{1-\alpha}}+o\left({n^{-1+\alpha}}\right).
 \end{equation}
Inserting \eqref{eq-3.43} in \eqref{eq-3.25}, we get the estimates \eqref{eq-3.16} and \eqref{eq-3.18}.\\[0.1in]
\textbf{Case 2.} If there exists $k\in\mathbb{Z}$ such that $\sqrt{\frac{k_1}{k_2}}=\pi+2k\pi$, then
$$1+\cos\left(\sqrt{\frac{k_1}{k_2}}\right)=0\ \ \ \text{and}\ \ \ \sin\left(\sqrt{\frac{k_1}{k_2}}\right)=0,$$
 therefore, from \eqref{eq-3.42}, we get
 \begin{equation}\label{eq-3.44}
\epsilon_{2,n}=\frac{\frac{i k_1}{{k_2}}\sqrt{\frac{k_1}{\rho_1}}}{8  n\pi}+\frac{\varepsilon_{2,n}}{n},\ \text{such that } \lim_{|n|\to+\infty}\varepsilon_{2,n}=0.
\end{equation}
Substituting \eqref{eq-3.44} in \eqref{eq-3.25}, we get
 \begin{equation}\label{eq-3.45}
\lambda_{2,n}=i n\pi\sqrt{\frac{k_1}{\rho_1}}+\frac{i\pi \sqrt{\frac{k_1}{\rho_1}}}{2} +\frac{\frac{i k_1}{{k_2}}\sqrt{\frac{k_1}{\rho_1}}}{8  n\pi}+\frac{\varepsilon_{2,n}}{n},
 \end{equation} 
 since in this case the real part of $\lambda_{2,n}$ still does not appear, we need to increase the order of the finite expansion.  Inserting \eqref{eq-3.45} in \eqref{eq-3.13} and using the asymptotic expansion, we get
 \begin{equation*}
 \left\{
 \begin{array}{ll}
 \displaystyle{r_1+r_2=2in\pi+i\pi+\frac{2\sqrt{\frac{\rho_1}{k_1}}\varepsilon_{2,n}}{n}+\frac{i k_1}{8k_2\pi n^2}+O \left(n^{-3}\right)},\\ \noalign{\medskip}
\displaystyle{
 
  r_1-r_2=i \sqrt{\frac{k_1}{k_2}}+\frac{i\left(\sqrt{\frac{k_1}{k_2}}\right)^3}{8\pi^2 n^2}+O \left( n^{-3}\right).}
 \end{array} 
 \right.
 \end{equation*}
 Therefore, we have
\begin{equation}\label{eq-3.46}
\left\{
\begin{array}{ll}
\displaystyle{\sinh(r_1+r_2)=-\frac{2\sqrt{\frac{\rho_1}{k_1}}\varepsilon_{2,n}}{n}-\frac{i k_1}{8k_2\pi n^2}+O \left( n^{-3}\right),}
\\ \noalign{\medskip}
\displaystyle{\sinh(r_1-r_2)=-\frac{i\left(\sqrt{\frac{k_1}{k_2}}\right)^3}{8\pi^2 n^2}+O \left(n^{-3}\right),}
\\ \noalign{\medskip}
\displaystyle{\cosh(r_1+r_2)-\cosh(r_1-r_2)=-\frac{2\rho_1\varepsilon^2_{2,n}}{k_1 n^2}+O\left(n^{-3}\right). }
\end{array}
\right.
\end{equation}
Inserting \eqref{eq-3.45} and \eqref{eq-3.46} in \eqref{eq-3.21}, then using the asymptotic expansion, we get
 \begin{equation*}
 -\frac{2\sqrt{\frac{\rho_1}{k_1}}}{n}\left(\varepsilon_{2,n}+\frac{\frac{ik_1}{k_2}\sqrt{\frac{k_1}{\rho_1}} }{16\pi n}
+O\left(n^{-2}\right)+O\left({n^{-2+\alpha}}{\varepsilon_{2,n}^2}\right) \right)=0.
 \end{equation*}
Consequently, we get
 \begin{equation}\label{eq-3.47}
 \varepsilon_{2,n}=-\frac{\frac{ik_1}{k_2}\sqrt{\frac{k_1}{\rho_1}} }{16\pi n}+\frac{\zeta_{2,n}}{n},\ \ \text{such that }\lim_{|n|\to+\infty}\zeta_{2,n}=0.
 \end{equation}
Inserting \eqref{eq-3.47} in \eqref{eq-3.45}, we get
  \begin{equation}\label{eq-3.48}
\lambda_{2,n}=i n\pi\sqrt{\frac{k_1}{\rho_1}}+\frac{i\pi \sqrt{\frac{k_1}{\rho_1}}}{2} +\frac{\frac{i k_1}{{k_2}}\sqrt{\frac{k_1}{\rho_1}}}{8  n\pi}-\frac{\frac{ik_1}{k_2}\sqrt{\frac{k_1}{\rho_1}} }{16\pi n^2}+\frac{\zeta_{2,n}}{n^2},
 \end{equation} 
 again  the real part of $\lambda_{2,n}$ still does not appear, so  we need to increase the order of the finite expansion.  For this aim, inserting \eqref{eq-3.48} in \eqref{eq-3.13} and using the asymptotic expansion, we get
 \begin{equation*}
 \left\{
 \begin{array}{ll}
 \displaystyle{r_1+r_2=(2n+1)i\pi+\frac{2\sqrt{\frac{\rho_1}{k_1}}\zeta_{2,n}}{n^2}-\frac{i k_1\, (4\pi^2 k_2+3k_1) }{64k_2^2 \pi^3 n^3}-\frac{i \frac{k_1}{k_2}\left(32 i \pi \sqrt{\frac{\rho_1}{k_1}}\zeta_{2,n}-4\pi^2  -9 \frac{k_1}{k_2}\right)}{128\pi^3 n^4}+O \left( n^{-5}\right)},\\ \noalign{\medskip}
\displaystyle{
 
  r_1-r_2=i \sqrt{\frac{k_1}{k_2}}+\frac{i \left(\sqrt{\frac{k_1}{k_2}}\right)^3}{8\pi^2 n^2}-\frac{i \left(\sqrt{\frac{k_1}{k_2}}\right)^3}{8\pi^2 n^3}+O \left(n^{-4}\right).}
 \end{array} 
 \right.
 \end{equation*}
Therefore, we have
\begin{equation}\label{eq-3.49}
\left\{
\begin{array}{ll}
\displaystyle{\sinh(r_1+r_2)=-\frac{2\sqrt{\frac{\rho_1}{k_1}}\zeta_{2,n}}{n^2}+\frac{i k_1\, (4\pi^2 k_2+3k_1) }{64k_2^2 \pi^3 n^3}

-\frac{\frac{ik_1}{k_2}\left(4 \pi^2+  \frac{9k_1}{k_2}-32i\pi  \sqrt{\frac{\rho_1}{k_1}}   \zeta_{2,n}\right) }{128\pi^3 n^4}+O \left( n^{-5}\right),}

\\ \noalign{\medskip}
\displaystyle{\sinh(r_1-r_2)=-\frac{i \left(\sqrt{\frac{k_1}{k_2}}\right)^3}{8\pi^2 n^2}+\frac{i \left(\sqrt{\frac{k_1}{k_2}}\right)^3}{8\pi^2 n^3}+O \left(n^{-4}\right),}
\\ \noalign{\medskip}
\displaystyle{\cosh(r_1+r_2)-\cosh(r_1-r_2)=-\frac{\frac{2\rho_1\zeta^2_{2,n}}{k_1}+\frac{k_1^3}{128k_2^3\pi^4}}{n^4}+O\left(n^{-5}\right). }
\end{array}
\right.
\end{equation}
 Inserting \eqref{eq-3.48} and \eqref{eq-3.49} in \eqref{eq-3.21}, then using the asymptotic expansion, we get
 \begin{equation*}
 \begin{array}{ll}
\displaystyle{
 -\frac{2\sqrt{\frac{\rho_1}{k_1}}}{n^2}\bigg(\zeta_{2,n}
-\frac{i\frac{k_1}{k_2} \sqrt{\frac{k_1}{\rho_1}} \left(4\pi^2-\frac{k_1}{k_2}\right) }{128\pi^3 n}+\frac{i\frac{k_1}{k_2} \sqrt{\frac{k_1}{\rho_1}} \left(4\pi^2-\frac{3k_1}{k_2}\right) }{256\pi^3 n^2}}
\\ \noalign{\medskip}\hspace{3cm}
\displaystyle{

-{\frac {\gamma \sqrt{k_1^{5+\alpha}}  \left( i\cos \left(\frac{\pi\alpha}{2}  \right) -\sin \left( \frac{\pi\alpha}{2} \right)  \right) }{256 k_2^3\sqrt{\rho_1^{1+\alpha}}\,{
\pi}^{5-\alpha}{n}^{2-\alpha}}}
 +O\left(n^{-3}\right)+O\left(n^{-2}{\zeta_{2,n}}\right) \bigg)=0.
  }
\end{array}
 \end{equation*}
Consequently, we get
\begin{equation}\label{eq-3.50}
\zeta_{2,n}=\frac{i\frac{k_1}{k_2} \sqrt{\frac{k_1}{\rho_1}} \left(4\pi^2-\frac{k_1}{k_2}\right) }{128\pi^3 n}-\frac{i\frac{k_1}{k_2} \sqrt{\frac{k_1}{\rho_1}} \left(4\pi^2-\frac{3k_1}{k_2}\right) }{256\pi^3 n^2}+{\frac {\gamma \sqrt{k_1^{5+\alpha}}  \left( i\cos \left(\frac{\pi\alpha}{2}  \right) -\sin \left( \frac{\pi\alpha}{2} \right)  \right) }{256 k_2^3\sqrt{\rho_1^{1+\alpha}}\,{
\pi}^{5-\alpha}{n}^{3-\alpha}}}+O\left(n^{-3}\right). 
\end{equation}
Finally, inserting \eqref{eq-3.50} in \eqref{eq-3.48}, we get \eqref{eq-3.20}. Thus, the proof of the proposition  is complete. \end{proof}$\\[0.1in]$
 \noindent\textbf{Proof of  Theorem \ref{Theorem-1.8}.} From Propositions \ref{Theorem-1.10} and \ref{Theorem-1.10new}, the operator $\mathcal{A}_1$ has two branches of eigenvalues with eigenvalues admitting real parts tending to zero. Hence, the energy corresponding to the first and second branch of eigenvalues  has no exponential decaying. Therefore the total energy of the  Timoshenko System \eqref{Part-1-07}-\eqref{Part-1-11} has no exponential decaying both in the equal speed case, i.e., $\frac{\rho_1}{k_1}=\frac{\rho_2}{k_2}$ or in the different speed case, i.e., when $\frac{\rho_1}{k_1}\neq\frac{\rho_2}{k_2}$.  	\xqed{$\square$}
\subsection{Polynomial stability}\label{Section-1.4-Polynomial stability}
\noindent In the case where  $\left(e^{t\mathcal{A}_1}\right)_{t\geq0}$ is not exponentially stable, we look for a polynomial decay rate.   In this section, we use the  frequency domain approach method to show the polynomial  stability of $\left(e^{t\mathcal{A}_1}\right)_{t\geq0}$ associated with the Timoshenko System \eqref{Part-1-07}-\eqref{Part-1-11}. The frequency domain approach method has been obtained by Batty in \cite{Batty01}, Borichev and  Tomilov in \cite{Borichev01},  Liu and  Rao in \cite{RaoLiu01}.
\begin{thm}\label{chapter pr-44}
\rm{(Batty in \cite{Batty01}, Borichev and  Tomilov in \cite{Borichev01},  Liu and  Rao in \cite{RaoLiu01}). Assume that $\mathcal{A}_1$ is the generator of a strongly continuous semigroup of contractions $\left(e^{t\mathcal{A}_1}\right)_{t\geq0}$  on $\mathcal{H}_1$.   If   $\sigma\left(\mathcal{A}_1\right)\cap\ i\mathbb{R}=\emptyset$, then for a fixed $\ell>0$ the following conditions are equivalent
\begin{enumerate}
\item[1.] $\sup_{\lambda\in\mathbb{R}}\left\|\left(i\lambda Id-\mathcal{A}_1\right)^{-1}\right\|_{\mathcal{L}\left(\mathcal{H}_1\right)}=O\left(|\lambda|^\ell\right),$
\item[2.] $\displaystyle{\|e^{t\mathcal{A}_1}U_0\|_{\mathcal{H}_1} \leq \dfrac{C}{t^{{\frac{1}{\ell}}}} \ \|U_0\|_{D\left(\mathcal{A}_1\right)} \quad\forall\ t>0,\ U_0\in D\left(\mathcal{A}_1\right)},\ $ for some $C>0.$
\end{enumerate}}\xqed{$\square$}
\end{thm}
\noindent  Our results are gathered in  the following  two theorems.
\begin{thm}\label{Theorem-1.11}
\rm{Assume that $\eta>0$,  condition $(\rm A_1)$  holds, and
\begin{equation}\label{Part-1-92}
	\frac{\rho_1}{k_1}=\frac{\rho_2}{k_2}.
\end{equation}
Then there exists  $c>0$ such that, for every $U_0\in D\left(\mathcal{A}_1\right)$, the energy of the System \eqref{Part-1-07}-\eqref{Part-1-11} has the optimal polynomial decay rate, we have
\begin{equation}\label{ee}
	E_1\left(t\right)\leq \frac{c}{t^{\tau(\alpha)}}\left\|U_0\right\|^2_{D\left(\mathcal{A}_1\right)},\quad\ t>0,
\end{equation}
where
\begin{equation*}
\tau(\alpha)=\left\{
\begin{array}{ll}
\frac{2}{1-\alpha},&\text{If } \sqrt{\frac{k_1}{k_2}}\neq k\pi,\ \forall k\in\mathbb{N} ,

\\ \noalign{\medskip}

\frac{2}{5-\alpha},&\text{If } \sqrt{\frac{k_1}{k_2}}= k_0\pi,\  k_0\in\mathbb{N}. 

\end{array}
\right.
\end{equation*}}
\end{thm}
\begin{thm}\label{Theorem-1.12}
\rm{Assume that $\eta>0$, condition $(\rm A_1)$  holds, and
\begin{equation}\label{Part-1-93}
	\frac{\rho_1}{k_1}\neq\frac{\rho_2}{k_2}.
\end{equation} 
Then, for almost all real number $\sqrt{\frac{k_1\rho_2}{k_2 \rho_1}}$,   there exists  $c>0$ such that for every $U_0\in D\left(\mathcal{A}_1\right)$, we have
\begin{equation*}
	E_1\left(t\right)\leq \frac{c}{t^{\frac{2}{5-\alpha}}}\left\|U_0\right\|^2_{D\left(\mathcal{A}_1\right)},\quad\ t>0.
\end{equation*}}
\end{thm}
\noindent Since $\sigma\left(\mathcal{A}_1\right)\cap\ i\mathbb{R}=\emptyset$, for the proof of Theorem \ref{Theorem-1.11} and Theorem \ref{Theorem-1.12}, according to Theorem \ref{chapter pr-44}, we need to prove that
\begin{equation*}\tag{\rm H3}
\sup_{\lambda\in\mathbb{R}}\left\|\left(i\lambda Id-\mathcal{A}_1\right)^{-1}\right\|_{\mathcal{L}\left(\mathcal{H}_1\right)}=O\left(|\lambda|^\ell\right),
\end{equation*}
where $\ell=\frac{2}{\tau(\alpha)}$ if condition \eqref{Part-1-92} holds and $\ell=5-\alpha$ if condition \eqref{Part-1-93} holds.\\[0.1in]
We will argue by contradiction. Therefore suppose  there exists $\left\{(\lambda_n,U_n=\left(u_n,v_n,y_n,z_n,\omega_n\right))\right\}_{n\geq 1}\subset \mathbb{R}\times D\left(\mathcal{A}_1\right)$, with $\lambda_n>\sqrt{\frac{\rho_2}{k_1}}$  and 
\begin{equation}\label{Part-1-94}
\lambda_n\to+\infty,\ \ \|U_n\|_{\mathcal{H}_1}=1,
\end{equation}
such that
\begin{equation}\label{Part-1-95}
\lambda_n^\ell\left(\ i\lambda_n U_n-\mathcal{A}_1U_n\right)=\left(f_{1,n},f_{2,n},f_{3,n},f_{4,n},f_{5,n}\right)\to 0\ \text{ in } \mathcal{H}_1.
\end{equation}
Equivalently, we have 
\begin{eqnarray}
\ i\lambda_n u_n-v_n&=&h_{1,n} ,\label{Part-1-96}
\\ \noalign{\medskip}
\rho_1\lambda_n^2 u_n+k_1 \left[\left(u_n\right)_x+y_n\right]_x
&=&h_{2,n},\label{Part-1-97}
\\ \noalign{\medskip}
 i\lambda_n y_n-z_n&=&h_{3,n} ,\label{Part-1-98}
\\ \noalign{\medskip}
\rho_2\lambda_n^2 y_n   +k_2 \left(y_n\right)_{xx}-k_1\left[\left(u_n\right)_x+y_n\right]&=&h_{4,n},\label{Part-1-99}
\\ \noalign{\medskip}
\left(i\lambda_n+\xi^2+\eta\right)\omega(\xi)-i\lambda_n u(1)\mu(\xi) &=& h_{5,n}(\xi),\label{Part-1-100}
\end{eqnarray}
where
\begin{equation*}
\left\{
\begin{array}{ll}

\displaystyle{ \lambda_{n}^{\ell}h_{1,n}=f_{1,n},\ \lambda_{n}^{\ell}h_{2,n}=-\rho_1 \left(f_{2,n}+\ i\lambda_n f_{1,n}\right),\ \lambda_{n}^{\ell}h_{3,n}=f_{3,n},}

\\ \noalign{\medskip}

\displaystyle{ \lambda_{n}^{\ell}h_{4,n}=-\rho_2\left(  f_{4,n}+\ i\lambda_nf_{3,n}\right),\
\lambda_{n}^{\ell}h_{5,n}(\xi)=f_{5,n}(\xi)-f_{1,n}(1)\mu(\xi).}  

\end{array}
\right.
\end{equation*}
In the following, we will prove that, under Condition {\rm (H3)}, and  
\eqref{Part-1-94}, one also gets that  $\left\| U_n\right\|_{\mathcal{H}_1} =o(1)$, hence reaching the desired contradiction. For clarity, we divide the proof into several lemmas. From  now on, for simplicity, we drop the index $n$.
 From \eqref{Part-1-96} and \eqref{Part-1-98},
we remark that
\begin{equation}\label{Part-1-101}
\left\|u\right\|=O\left(\lambda^{-1}\right)\ \ \ \text{and}\ \ \ \left\|y\right\|=O\left(\lambda^{-1}\right).
\end{equation}
\begin{lem}\label{Theorem-1.13}
\rm{Let $\ell>0$, we have
\begin{equation}\label{Part-1-102}
	\int_{\mathbb{R}}(\xi^2+\eta)\left|\omega(\xi)\right|^2 d\xi=o\left(\lambda^{-\ell}\right).
\end{equation}}
\end{lem}
\begin{proof} Taking the inner product of \eqref{Part-1-95} with $U$ in $\mathcal{H}_1$, then using the fact that $U$ is uniformly bounded in $\mathcal{H}_1$, we get
\begin{equation*}
	\int_{\mathbb{R}}(\xi^2+\eta)\left|\omega(\xi)\right|^2 d\xi=-\text{Re}\left(\left<\mathcal{A}_1U,U\right>_{\mathcal{H}_1}\right)=\text{Re}\left(\left<\ i\lambda U-\mathcal{A}_1U,U\right>_{\mathcal{H}_1}\right)=o\left(\lambda^{-\ell}\right).
\end{equation*}
\end{proof}
\begin{lem}\label{Theorem-1.14}
\rm{Let $\ell>0$,  we have
\begin{eqnarray}
\displaystyle{\left|u_x(1)+y(1)\right|=o\left(\lambda^{-\frac{\ell}{2}}\right),}\label{Part-1-103}\nline
\displaystyle{|\lambda u(1)|=o\left(\lambda^{-\frac{\ell+\alpha-1}{2}}\right).}\label{Part-1-104}
\end{eqnarray}} 
\end{lem}
\begin{proof}
First,  from the boundary condition, we have 
$$
u_x(1)+y(1)=-\frac{\gamma\kappa(\alpha)}{k_1}\int_{\mathbb{R}}\mu(\xi)\omega(\xi)d\xi,
$$
using Cauchy-Shariwz inequality, we get
\begin{equation}\label{Part-1-105}
\left|u_x(1)+y(1)\right|\leq \frac{\gamma\kappa(\alpha)}{k_1}\left(\int_{\mathbb{R}}\frac{\mu^2(\xi)}{\xi^2+\eta}d\xi\right)^{\frac{1}{2}}\left(\int_{\mathbb{R}}\left(\xi^2+\eta\right)|\omega(\xi)|^2d\xi\right)^{\frac{1}{2}}.
\end{equation}
Then, from \eqref{Part-1-102},  \eqref{Part-1-105} and using the fact that $\mu^2(\xi)\left(\xi^2+\eta\right)^{-1}\in L^1\left(\mathbb{R}\right)$ for all $\alpha\in(0,1)$, we obtain  asymptotic estimate of \eqref{Part-1-103}. Next, from \eqref{Part-1-100}, we get
\begin{equation}\label{new-01}
\left|\lambda u(1)\right| \left|\xi\right|^{\alpha-\frac{1}{2}} \leq\left(\lambda+\xi^2+\eta\right)\left|\omega(\xi)\right|+\lambda^{-\ell}\left|f_{5}(\xi)\right|+\lambda^{-\ell} \left|f_{1}(1)\right|\left|\xi\right|^{\alpha-\frac{1}{2}}.
\end{equation}
Multiplying equation \eqref{new-01} by $(\lambda+\xi^2+\eta)^{-2}\left|\xi\right|$, integrating over $\mathbb{R}$ with respect to the variable $\xi$ and applying Cauchy-Schwarz inequality, we obtain
\begin{equation}\label{Part-1-106}
\left|\lambda u(1)\right| A_1  \leq A_2\left(\int_{\mathbb{R}}\left|\xi\omega(\xi)\right|^2d\xi\right)^{\frac{1}{2}}+{A_3}\lambda^{-\ell}\left(\int_{\mathbb{R}}\left|f_{5}(\xi)\right|^2d\xi\right)^{\frac{1}{2}} +{\left|f_{1}(1)\right|}\lambda^{-\ell} A_1,
\end{equation}
where
$$A_1=\int_{\mathbb{R}}\frac{ \left|\xi\right|^{\alpha+\frac{1}{2}}}{(\lambda+\xi^2+\eta)^{2}}d\xi,\ A_2=\left(\int_{\mathbb{R}}\frac{1}{\left(\lambda+\xi^2+\eta\right)^2}d\xi\right)^{\frac{1}{2}},\ A_3=\left(\int_{\mathbb{R}}\frac{ \xi^2}{(\lambda+\xi^2+\eta)^{4}} d\xi\right)^{\frac{1}{2}}.$$
It is easy to check that 
\begin{equation}\label{new-02}
A_2=\sqrt{\frac{\pi}{2}}\frac{1}{\left(\lambda+\eta\right)^{\frac{3}{4}}}\ \ \ \text{and}\ \ \ A_3=\frac{\sqrt{\pi}}{4}\frac{1}{\left(\lambda+\eta\right)^{\frac{5}{4}}}.
\end{equation}
Moreover, we have
\begin{equation}\label{new-03}
I_1=\frac{2}{\left(\lambda+\eta\right)^{2}}\int_{0}^{\infty}\frac{ \xi^{\alpha+\frac{1}{2}}}{\left(1+\frac{\xi^2}{\lambda+\eta}\right)^{2}}d\xi.
\end{equation}
Thus, equation \eqref{new-03} may be simplified by defining a new variable $y=1+\frac{\xi^2}{\lambda+\eta}$.  Substituting $\xi$  by $\left(y-1\right)^{\frac{1}{2}}\left(\lambda+\eta\right)^{\frac{1}{2}}$ in equation \eqref{new-03}, we get
\begin{equation*}
I_1=\left(\lambda+\eta\right)^{\frac{\alpha}{2}-\frac{5}{4}}\int_{1}^{\infty}\frac{\left(y-1\right)^{\frac{\alpha}{2}-\frac{1}{4}}}{y^2}dy.
\end{equation*}
Using the fact that $\alpha\in (0,1)$, it is easy to see that $y^{-2}\left(y-1\right)^{\frac{\alpha}{2}-\frac{1}{4}}\in L^1(1,+\infty)$, we obtain
\begin{equation}\label{Part-1-107}
A_1=c_1\left(\lambda+\eta\right)^{\frac{\alpha}{2}-\frac{5}{4}},
\end{equation}
where $c_1$ is a positive constant number. Inserting \eqref{new-02} and  \eqref{Part-1-107} in  \eqref{Part-1-106}, then using \eqref{Part-1-95}, \eqref{Part-1-102} and the fact that $|f_1(1)|\leq \left\|f_1\right\|_{H^1_L(0,1)}=o(1)$,   we deduce that
\begin{equation}\label{new-04}
 \left|\lambda u(1)\right|   \leq \sqrt{\frac{\pi}{2}}\frac{1}{c_1\left(\lambda+\eta\right)^{\frac{\alpha}{2}-\frac{1}{2}} } o\left(\lambda^{-\frac{\ell}{2}}\right)+\frac{\sqrt{\pi}}{4}\frac{1}{c_1\left(\lambda+\eta\right)^{\frac{\alpha}{2}}}o\left(\lambda^{-\ell}\right)+o\left(\lambda^{-\ell}\right).
\end{equation}
Since $\alpha\in (0,1)$ and $\ell>0$, we have $$\min\left(\frac{\ell+\alpha-1}{2},\ell+\frac{\alpha}{2},\ell\right)=\frac{\ell+\alpha-1}{2},$$ consequently, from \eqref{new-04}, we get  \eqref{Part-1-104}. Thus, the proof of the lemma is complete.
\end{proof}
\begin{lem}\label{Theorem-1.18}
\rm{Let $\ell>0$,  we have
\begin{equation}\label{Part-1-160}
			\begin{array}{c}
			\displaystyle{\int_0^{1}\left(\rho_1\left|\lambda u\right|^2+\rho_2\left|\lambda u\right|^2+k_1\left| u_x\right|^2 +k_2\left| y_x\right|^2\right)dx}

 \\ \noalign{\medskip}
\displaystyle{ =\rho_1 \left|\lambda u(1)\right|^2+k_1 \left| u_x(1)\right|^2+\rho_2 \left|\lambda y(1)\right|^2+	o(1)+o\left(\lambda^{-\ell}\right) \text{Re}\left\{\lambda\overline{ u}(1)  \right\} +o\left(\lambda^{-\ell}\right) \text{Re}\left\{\lambda\overline{ y}(1)  \right\}  }. 
			\end{array}
\end{equation}}
\end{lem}
\begin{proof}  First,  multiplying Equation \eqref{Part-1-97} by $2x  \overline{u}_x$ in $L^2(0,1)$   to get
\begin{equation}\label{Part-1-161}
			\begin{array}{c}
			\displaystyle{-\rho_1\int_0^{1}\left|\lambda u\right|^2dx
				-k_1\int_0^{1}\left| u_x\right|^2dx}
+\displaystyle{2 k_1\text{Re}\left\{\int_0^{1} y_x\overline{u}_xdx\right\}
+\rho_1 \left|\lambda u(1)\right|^2+k_1 \left| u_x(1)\right|^2

}

 \\ \noalign{\medskip}
 
	\displaystyle{	=-2\lambda^{-\ell}\rho_1 \text{Re}\left\{\int_0^{1}x f_2\overline{u}_xdx-\ i\int_0^{1} \left( f_1+ x (f_1)_x\right)\lambda 
				\overline{ u}dx+if_1(1)\lambda\overline{ u}(1) \right\}  }. 
			\end{array}
\end{equation}
Using the fact that $f_1\to 0$ in $H^1_L(0,1)$, $f_2\to 0$ in $L^2(0,1)$, and $u_x,\ \lambda u$ are bounded in $L^2(0,1)$, we get
\begin{equation}\label{Part-1-162}
\int_0^{1}\left|x f_2\overline{u}_x\right|dx=o(1),\ 
\int_0^{1} \left|\left( f_1+x (f_1)_x\right)\lambda 
				\overline{ u}\right|dx=o(1)
\end{equation}
and
\begin{equation}\label{Part-1-163}
|f_1(1)|\leq \left\|f_1\right\|_{L^\infty(0,1)}\leq \left\|f_1\right\|_{H^1_L(0,1)}=o(1).
\end{equation}
Inserting \eqref{Part-1-162} and \eqref{Part-1-163} in  \eqref{Part-1-161}, we get
\begin{equation}\label{Part-1-164}
\begin{array}{c}
			\displaystyle{-\rho_1\int_0^{1}\left|\lambda u\right|^2dx
				-k_1\int_0^{1}\left| u_x\right|^2dx}
+\displaystyle{2 k_1\text{Re}\left\{\int_0^{1} y_x\overline{u}_xdx\right\}
+\rho_1 \left|\lambda u(1)\right|^2+k_1 \left| u_x(1)\right|^2

}

 \\ \noalign{\medskip}
 
	\displaystyle{	=o\left(\lambda^{-\ell}\right)+ o\left(\lambda^{-\ell}\right)\text{Re}\left\{\lambda\overline{ u}(1) \right\}  }. 
			\end{array}
\end{equation}
Next, multiplying Equation \eqref{Part-1-99} by $2x  \overline{y}_x$ in $L^2(0,1)$, then using the fact that $y_x(0)=y_x(1)=0$,  we get
\begin{equation}\label{Part-1-165}
			\begin{array}{c}
			\displaystyle{-\rho_2\int_0^{1}\left|\lambda y\right|^2dx
				-k_2\int_0^{1}\left| y_x\right|^2dx}
-\displaystyle{2 k_1\text{Re}\left\{\int_0^{1}x u_x\overline{y}_xdx\right\}

-2 k_1\text{Re}\left\{\int_0^{1}x y\overline{y}_xdx\right\}+\rho_2 \left|\lambda y(1)\right|^2}	

  \\ \noalign{\medskip}
 
	\displaystyle{	=-{2\lambda^{-\ell}\rho_1} \text{Re}\left\{\int_0^{1}x f_4\overline{y}_xdx-\ i\int_0^{1} \left( f_3+ x (f_3)_x\right)\lambda 
				\overline{ y}dx+if_3(1)\lambda\overline{ y}(1)  \right\}  }. 
			\end{array}
\end{equation}
Using the fact that $f_3\to 0$ in $H^1_*(0,1)$, $f_4\to 0$ in $L^2(0,1)$, $\lambda y,\ y_x$ are bounded in $L^2(0,1)$, and \eqref{Part-1-101}, we get
\begin{equation}\label{Part-1-166}
\int_0^{1}\left|x y\overline{y}_x\right|dx=o(1),\int_0^{1}\left|x f^4\overline{y}_x\right|dx=o(1),\ 
\int_0^{1} \left|\left( f_3+ x (f_3)_x\right)\lambda 
				\overline{ y}\right|dx=o(1)
\end{equation}
and
\begin{equation}\label{Part-1-167}
|f_3(1)|\leq \left\|f_3\right\|_{L^\infty(0,1)}\leq \left\|f_3\right\|_{H^1_*(0,1)}=o(1).
\end{equation}
Substituting \eqref{Part-1-166} and \eqref{Part-1-167} in  \eqref{Part-1-165},  we get
\begin{equation}\label{Part-1-168}
			\begin{array}{c}
			\displaystyle{-\rho_2\int_0^{1}\left|\lambda y\right|^2dx
				-k_2\int_0^{1}\left| y_x\right|^2dx}
-\displaystyle{2 k_1\text{Re}\left\{\int_0^{1}x u_x\overline{y}_xdx\right\}
+\rho_2 \left|\lambda y(1)\right|^2
}

 \\ \noalign{\medskip}
 
	\displaystyle{	=o(1)+o\left(\lambda^{-\ell}\right) \text{Re}\left\{\lambda\overline{ y}(1) \right\}  }. 
			\end{array}
\end{equation}
Finally,  adding \eqref{Part-1-164} and \eqref{Part-1-168}, we get \eqref{Part-1-160}, which concludes the  proof of the lemma.
\end{proof}
$\\[0.1in]$ 
For all $\ell\geq 1-\alpha$, from Lemma \ref{Theorem-1.14}, we obtain  $|\lambda u(1)|=o(1)$.  Let us suppose that $$|\lambda y(1)|=o(1),$$ then from Lemma \ref{Theorem-1.14}, we get $|u_x(1)|=o(1)$. Therefore, from Lemma \ref{Theorem-1.18}, we get 
$$\int_0^{1}\left(\rho_1\left|\lambda u\right|^2+\rho_2\left|\lambda u\right|^2+k_1\left| u_x\right|^2 +k_2\left| y_x\right|^2\right)dx=o(1).$$
 Consequently, we have $\|U\|_{\mathcal{H}_1}=o\left(1\right)$ which contradicts \eqref{Part-1-94}.  So, in order to complete the proof of Theorems \ref{Theorem-1.11}, \ref{Theorem-1.12}, we need to show that
$$|\lambda y(1)|=o(1).$$
For this aim, we need to prove the following lemmas.
\begin{lem}\label{Theorem-1.15}
\rm{ Let $\ell>0$,  we have
\begin{eqnarray}
\left(\frac{S_{11}(1)}{2}-\frac{\lambda\left(\frac{\rho_2}{k_2}+\frac{\rho_1}{k_1}\right) S_{12}(1)}{2 \Delta }\right)y(1)=\mathcal{K}_1,\label{Part-1-108}
\\ \noalign{\medskip}
\left(\frac{\rho_2\lambda^2 S_{11}(1)}{2k_2}-\frac{\lambda\left({\rho_2}\left(\frac{\rho_2}{k_2}-\frac{\rho_1}{k_1}\right)\lambda^2+2{ \rho_1}  \right)S_{12}(1)}{2k_2\Delta} \right)y(1)=\mathcal{K}_4,\label{Part-1-109}
\end{eqnarray}
such that
\begin{equation*}
\begin{array}{ll}

\displaystyle{ \mathcal{K}_1=\left(\frac{S_{11}(1)}{2}+\frac{\left(\left(\frac{\rho_2}{k_2}-\frac{\rho_1}{k_1}\right)\lambda^2-\frac{2k_1}{k_2}\right)S_{12}(1)}{2 \Delta \lambda}\right)o\left(\lambda^{-\frac{\ell}{2}}\right)+\left(\frac{C_{11}(1)}{2}+\frac{\left(\frac{\rho_2}{k_2}-\frac{\rho_1}{k_1}\right)\lambda C_{12}(1)}{2\Delta}\right)o\left({\lambda^{-\frac{\ell+\alpha+1}{2}}}\right)}

\\ \noalign{\medskip}

\hspace{1cm} \displaystyle{-{\frac{1}{2k_1}}\int_0^1 S_{11}(z)h_2(z)dz-\frac{\left(\frac{\rho_2}{k_2}-\frac{\rho_1}{k_1}\right)  \lambda^2-\frac{2k_1}{k_2}}{2k_1\Delta\lambda}\int_0^1 S_{12}(z)h_2(z)dz-\frac{1}{k_2\Delta \lambda}\int_0^1 C_{12}(z)h_4(z)dz}
\end{array}
\end{equation*}
and
\begin{equation*}
\begin{array}{ll}

\displaystyle{ \mathcal{K}_4=\frac{\rho_1 \lambda C_{12}(1)}{k_2\Delta}o\left(\frac{1}{\lambda^{\frac{\ell+\alpha+1}{2}}}\right) +\left(\frac{{k_1} S_{11}(1)}{2k_2} -\frac{{k_1} \left(\frac{\rho_2}{k_2}+\frac{\rho_1}{k_1}\right) \lambda S_{12}(1)}{2k_2\Delta}\right) o\left(\lambda^{-\frac{\ell}{2}}\right)-{\frac{1}{2k_2}}\int_0^1 S_{11}(z)h_2(z)dz}

 \\ \noalign{\medskip}

\hspace{0.5cm} \displaystyle{+\frac{\left(\frac{k_2}{\rho_2}+\frac{k_1}{\rho_1}\right) \lambda}{2k_2\Delta}\int_0^1 S_{12}(z)h_2(z)dz+\frac{1}{2k_2}\int_0^1 C_{11}(z)h_4(z)dz-\frac{\left(\frac{\rho_2}{k_2}-\frac{\rho_1}{k_1}\right)\lambda }{2k_2\Delta}\int_0^1 C_{12}(z)h_4(z)dz

},

\end{array}
\end{equation*}
where
\begin{equation}\label{Part-1-110}
\left\{
\begin{array}{ll}
\displaystyle{C_{11}(z)=\cos(\mu_1 z)+\cos(\mu_2 z),\ C_{12}(z)=\cos(\mu_1 z)-\cos(\mu_2 z)},

\\ \\

\displaystyle{S_{11}(z)=\frac{\sin(\mu_1 z)}{\mu_1}+\frac{\sin(\mu_2 z)}{\mu_2},\ S_{12}(z)=\frac{\sin(\mu_1 z)}{\mu_1 }-\frac{\sin(\mu_2 z)}{\mu_2}},

\end{array}
\right.
\end{equation}
and
\begin{equation}\label{Part-1-111}
\displaystyle{\mu_1= \sqrt{\dfrac{\left(\frac{\rho_1}{k_1}+\frac{\rho_2}{k_2}\right)\lambda^2-\lambda{\Delta}}{2}},\ \mu_2= \sqrt{\dfrac{\left(\frac{\rho_1}{k_1}+\frac{\rho_2}{k_2}\right)\lambda^2+\lambda{\Delta}}{2}},\ \Delta=\sqrt{\left(\frac{\rho_1}{k_1}-\frac{\rho_2}{k_2}\right)^2 \lambda^2+\frac{4\rho_1}{k_2}}}.
\end{equation}}
\end{lem}
\begin{proof} Let $Y=(u,u_x,y,y_x) $, then Equations  \eqref{Part-1-97} and \eqref{Part-1-99}  can be written as 
\begin{equation}\label{Part-1-112}
Y_x=-BY+G,
\end{equation}
where
\begin{equation*}
B=\begin{pmatrix}
0&-1&0&0
\\ \noalign{\medskip}
\frac{\rho_1}{k_1}\lambda^2&0&0&1		
\\ \noalign{\medskip}
	0&0&0&-1
\\ \noalign{\medskip}
0&-\frac{k_1}{k_2}&\frac{\rho_2}{k_2}\lambda^2-\frac{k_1}{k_2}&0
		\end{pmatrix},\quad G=\begin{pmatrix}
0\\ k^{-1}_1h_2\\ 0\\ k^{-1}_2h_4
	\end{pmatrix}.
\end{equation*}
By the variation of constant formula, the solution of Equation \eqref{Part-1-112} is given by
\begin{equation*}
Y(x)=e^{-Bx}Y(0)+\int_0^xe^{-B(x-z)}G(z)dz.
\end{equation*}
Then, we have
$$
Y(1)=e^{-B}Y(0)+\int_0^1e^{-B(1-z)}G(z)dz.
$$
Equivalently, we get 
\begin{equation}\label{Part-1-113}
e^{B}Y(1)=Y(0)+\int_0^1e^{B z}G(z)dz.
\end{equation}
On the other hand, from \eqref{Part-1-103} and \eqref{Part-1-104}, we have
\begin{equation}\label{Part-1-114}
Y(1)=\begin{pmatrix}
o\left(\lambda^{-\frac{\ell+\alpha+1}{2}}\right)& -y(1)+o\left(\lambda^{-\frac{\ell}{2}}\right)&y(1)&0
	\end{pmatrix}^{\mathsf{T}}\ \ \ \text{and}\  \ \ Y(0)=\begin{pmatrix}
0& u_x(0)&y(0)&0
	\end{pmatrix}^{\mathsf{T}}. 
\end{equation}
The eigenvalues $\mu$ of the matrix $B$ are the roots of the 
characteristic equation \eqref{eq:carac0}
whose discriminant is equal to 
\begin{equation*}
\left(\frac{\rho_1}{k_1}-\frac{\rho_2}{k_2}\right)^2 \lambda^4+\frac{4\rho_1}{k_2}\lambda^2>0.
\end{equation*}
Since $\lambda>\sqrt{\frac{\rho_2}{k_1}}$, Equation \eqref{eq:carac0} has four distinct pure imaginary roots
$$i\mu_1,\ -i\mu_1,\ i\mu_2,\ -i\mu_2,$$
where $\mu_1$ and $\mu_2$ are defined in \eqref{Part-1-111}. Since the eigenvalues of $B$ are simple, then $B$ is a diagonalizable matrix. Therefore, using Sylvester's matrix Theorem, we get
\begin{equation*}
e^{B z}=\frac{1}{\mu_2^2-\mu_1^2}\left(\left(\frac{E_{-i\mu_1}e^{i\mu_1 z}-E_{i\mu_1}e^{-i\mu_1 z}}{2i\mu_1}\right)\left(B^2+\mu_2^2 I_{4\times 4}\right)-\left(\frac{E_{-i\mu_2}e^{i\mu_2 z}-E_{i\mu_2}e^{-i\mu_2 z}}{2i\mu_2}\right)\left(B^2+\mu_1^2I_{4\times 4}\right)\right),
\end{equation*}
where $E_{i\mu}=B-i\mu I_{4\times 4}$. Equivalently, we have
\begin{equation}\label{Part-1-116}
e^{B z}=\frac{1}{\lambda \Delta}\left(\frac{\sin(\mu_1 z)}{\mu_1}\mathcal{M}_1-\frac{\sin(\mu_2 z)}{\mu_2}\mathcal{M}_2+\cos(\mu_1 z)\mathcal{M}_3-\cos(\mu_2 z)\mathcal{M}_4  \right),
\end{equation}
where 
$$\mathcal{M}_1= B \left(B^2+\mu_2^2 I_{4\times 4}\right) ,\  \mathcal{M}_2= B \left(B^2+\mu_1^2 I_{4\times 4}\right) ,\ \mathcal{M}_3= B^2+\mu_2^2 I_{4\times 4},\ \mathcal{M}_4= B^2+\mu_1^2 I_{4\times 4}.$$
It is easy to check that
\begin{equation*}
\mathcal{M}_1=\begin{pmatrix}
0&\frac{k_1}{k_2}-\frac{\left(\frac{\rho_2}{k_2}-\frac{\rho_1}{k_1}\right)\lambda^2+\lambda \Delta}{2} &\frac{k_1}{k_2}-\frac{\rho_2}{k_2}\lambda^2 &0
\\ \noalign{\medskip}
\left(\frac{\left(\frac{\rho_2}{k_2}-\frac{\rho_1}{k_1}\right)\lambda^2+\lambda \Delta}{2}-\frac{k_1}{k_2}\right) \frac{\rho_1}{k_1}\lambda^2&0&0&\frac{\lambda \Delta-\left(\frac{\rho_2}{k_2}+\frac{\rho_1}{k_1}\right)\lambda^2}{2}\\ \noalign{\medskip}
\frac{\rho_1}{k_2}\lambda^2&0&0&\frac{\left(\frac{\rho_2}{k_2}-\frac{\rho_1}{k_1}\right)\lambda^2-\lambda \Delta}{2}
\\ \noalign{\medskip}
	0&\frac{\left(\frac{\rho_2}{k_2}+\frac{\rho_1}{k_1}\right){k_1}\lambda^2-{k_1}\lambda \Delta}{2k_2}&\frac{\left(\lambda \Delta-\left(\frac{\rho_2}{k_2}-\frac{\rho_1}{k_1}\right)\lambda^2\right)\left({\rho_2}\lambda^2-{k_1}\right)}{2k_2}&0
\end{pmatrix},
\end{equation*}
$\newline$ 
\begin{equation*}
\mathcal{M}_2=\begin{pmatrix}
0&\frac{k_1}{k_2}-\frac{\left(\frac{\rho_2}{k_2}-\frac{\rho_1}{k_1}\right)\lambda^2-\lambda \Delta}{2} &\frac{k_1}{k_2}-\frac{\rho_2}{k_2}\lambda^2 &0
\\ \noalign{\medskip}
\left(\frac{\left(\frac{\rho_2}{k_2}-\frac{\rho_1}{k_1}\right)\lambda^2-\lambda \Delta}{2}-\frac{k_1}{k_2}\right) \frac{\rho_1}{k_1}\lambda^2&0&0&\frac{-\lambda \Delta-\left(\frac{\rho_2}{k_2}+\frac{\rho_1}{k_1}\right)\lambda^2}{2}\\ \noalign{\medskip}
\frac{\rho_1}{k_2}\lambda^2&0&0&\frac{\left(\frac{\rho_2}{k_2}-\frac{\rho_1}{k_1}\right)\lambda^2+\lambda \Delta}{2}
\\ \noalign{\medskip}
	0&\frac{\left(\frac{\rho_2}{k_2}+\frac{\rho_1}{k_1}\right){k_1}\lambda^2+{k_1}\lambda \Delta}{2k_2}&\frac{\left(-\lambda \Delta-\left(\frac{\rho_2}{k_2}-\frac{\rho_1}{k_1}\right)\lambda^2\right)\left({\rho_2}\lambda^2-{k_1}\right)}{2k_2}&0
\end{pmatrix},
\end{equation*}
$\newline$ 
\begin{equation*}
\mathcal{M}_3=\begin{pmatrix}
\frac{\left(\frac{\rho_2}{k_2}-\frac{\rho_1}{k_1}\right)\lambda^2+\lambda \Delta}{2} &0&0&-1
\\ \noalign{\medskip}
0&\frac{\left(\frac{\rho_2}{k_2}-\frac{\rho_1}{k_1}\right)\lambda^2+\lambda \Delta}{2}-\frac{k_1}{k_2}&\frac{\rho_2}{k_2}\lambda^2-\frac{k_1}{k_2}&0		
\\ \noalign{\medskip}
	0&\frac{k_1}{k_2}&\frac{-\left(\frac{\rho_2}{k_2}-\frac{\rho_1}{k_1}\right)\lambda^2+\lambda \Delta}{2}+\frac{k_1}{k_2}&0
\\ \noalign{\medskip}
-\frac{\rho_1}{ k_2}\lambda^2&0&0&\frac{-\left(\frac{\rho_2}{k_2}-\frac{\rho_1}{k_1}\right)\lambda^2+\lambda \Delta}{2}
		\end{pmatrix},
		\end{equation*} 
and		
\begin{equation*}
\mathcal{M}_4=\begin{pmatrix}
\frac{\left(\frac{\rho_2}{k_2}-\frac{\rho_1}{k_1}\right)\lambda^2-\lambda \Delta}{2} &0&0&-1
\\ \noalign{\medskip}
0&\frac{\left(\frac{\rho_2}{k_2}-\frac{\rho_1}{k_1}\right)\lambda^2-\lambda \Delta}{2}-\frac{k_1}{k_2}&\frac{\rho_2}{k_2}\lambda^2-\frac{k_1}{k_2}&0		
\\ \noalign{\medskip}
	0
	&\frac{k_1}{k_2}&\frac{-\left(\frac{\rho_2}{k_2}-\frac{\rho_1}{k_1}\right)\lambda^2-\lambda \Delta}{2}+\frac{k_1}{k_2}&0
\\ \noalign{\medskip}
-\frac{\rho_1}{ k_2}\lambda^2&0&0&\frac{-\left(\frac{\rho_2}{k_2}-\frac{\rho_1}{k_1}\right)\lambda^2-\lambda \Delta}{2}
		\end{pmatrix}.
		\end{equation*} 
Inserting  $\mathcal{M}_1,\ \mathcal{M}_2,\ \mathcal{M}_3$ and $\mathcal{M}_4$ in \eqref{Part-1-116},  we get 
\begin{equation}\label{Part-1-121}
e^{Bz}=\left(\mathtt{e}_{ij}(z)\right),
\end{equation}
such that
\begin{equation*}
\left\{
\begin{array}{lll}

\displaystyle{\mathtt{e}_{11}(z)=\dfrac{C_{11}(z)}{2}+\dfrac{ \left(\frac{\rho_2}{k_2}-\frac{\rho_1}{k_1}\right) \lambda C_{12}(z)}{2\Delta},\ \mathtt{e}_{12}(z)=-\dfrac{S_{11}(z)}{2}-\dfrac{\left(\left(\frac{\rho_2}{k_2}-\frac{\rho_1}{k_1}\right)\lambda^2-\frac{2k_1}{k_2}\right)S_{12}(z)}{2\lambda \Delta}, }

\\ \noalign{\medskip}

\displaystyle{ \mathtt{e}_{13}(z)=-\dfrac{\left(\rho_2\lambda^2-{k_1}\right) S_{12}(z)}{k_2\lambda\Delta },\ \mathtt{e}_{14}(z)=-\dfrac{C_{12}(z)}{\lambda \Delta } },

\end{array}
\right.
\end{equation*}
\begin{equation*}
\left\{
\begin{array}{lll}

\displaystyle{\mathtt{e}_{21}(z)=\dfrac{{\rho_1}\lambda^2 S_{11}(z)}{2k_1}+\dfrac{{\rho_1}\lambda \left(\left(\frac{\rho_2}{k_2}-\frac{\rho_1}{k_1}\right)\lambda^2-\frac{2k_1}{k_2}\right)S_{12}(z)}{2{k_1}\Delta} ,\ \mathtt{e}_{22}(z)=\dfrac{C_{11}(z)}{2}+\dfrac{\left(\left(\frac{\rho_2}{k_2}-\frac{\rho_1}{k_1}\right)\lambda^2-\frac{2k_1}{k_2}\right)C_{12}(z)}{2\lambda \Delta}, }

\\ \noalign{\medskip}

\displaystyle{ \mathtt{e}_{23}(z)=\dfrac{\left(\rho_2\lambda^2-{k_1}\right) C_{12}(z)}{k_2\lambda\Delta },\ \mathtt{e}_{24}(z)=\dfrac{S_{11}(z)}{2}-\dfrac{\left(\frac{\rho_2}{k_2}+\frac{\rho_1}{k_1}\right)\lambda S_{12}(z)}{2 \Delta} },

\end{array}
\right.
\end{equation*}
\begin{equation*}
\left\{
\begin{array}{lll}

\displaystyle{\mathtt{e}_{31}(z)=\dfrac{\rho_1 \lambda S_{12}(z)}{k_2\Delta},\ \mathtt{e}_{32}(z)=\dfrac{{k_1}  C_{12}(z)}{{k_2} \lambda\Delta}, }

\\ \noalign{\medskip}

\displaystyle{ \mathtt{e}_{33}(z)=\dfrac{C_{11}(z)}{2}-\dfrac{\left(\left(\frac{\rho_2}{k_2}-\frac{\rho_1}{k_1}\right)\lambda^2-\frac{2k_1}{k_2}\right)C_{12}(z)}{2\lambda \Delta},\ \mathtt{e}_{34}(z)=-\dfrac{S_{11}(z)}{2}+\dfrac{\left(\frac{\rho_2}{k_2}-\frac{\rho_1}{k_1}\right)\lambda S_{12}(z)}{2\Delta} },

\end{array}
\right.
\end{equation*}
and
\begin{equation*}
\left\{
\begin{array}{lll}

\displaystyle{\mathtt{e}_{41}(z)=-\dfrac{\rho_1 \lambda C_{12}(z)}{k_2\Delta},\ \mathtt{e}_{42}(z)=-\dfrac{{k_1} S_{11}(z)}{2{k_2}}+\dfrac{{k_1}\left(\frac{\rho_2}{k_2}+\frac{\rho_1}{k_1}\right) \lambda S_{12}(z)}{2{k_2}\Delta}, }

\\ \noalign{\medskip}

\displaystyle{ \mathtt{e}_{43}(z)=\dfrac{\left({\rho_2}\lambda^2-{k_1}\right)S_{11}(z)}{2{k_2}}-\dfrac{\left(\frac{\rho_2}{k_2}-\frac{\rho_1}{k_1}\right)\left({\rho_2}\lambda^2-{k_1}\right)\lambda S_{12}(z)}{2{k_2}\Delta},\ \mathtt{e}_{44}(z)= \dfrac{C_{11}(z)}{2}-\dfrac{\left(\frac{\rho_2}{k_2}-\frac{\rho_1}{k_1}\right)\lambda C_{12}(z)}{2\Delta}}.

\end{array}
\right.
\end{equation*}
Finally, substituting \eqref{Part-1-114} and   \eqref{Part-1-121} in \eqref{Part-1-113}, we obtain
\begin{equation*}
\left(\mathtt{e}_{13}(1)-\mathtt{e}_{12}(1) \right) y(1)=o\left(\lambda^{-\frac{\ell+\alpha+1}{2}} \right) \mathtt{e}_{11}(1)+o\left(\lambda^{-\frac{\ell}{2}} \right) \mathtt{e}_{12}(1)  +\frac{1}{k_1}\int_0^1h_2(z)\mathtt{e}_{12}(z) dz+\frac{1}{k_2}\int_0^1h_4(z)\mathtt{e}_{14}(z) dz 
\end{equation*}
and
\begin{equation*}
\left(\mathtt{e}_{43}(1)-\mathtt{e}_{42}(1) \right) y(1)=o\left(\lambda^{-\frac{\ell+\alpha+1}{2}} \right)\mathtt{e}_{41}(1)+o\left(\lambda^{-\frac{\ell}{2}} \right)\mathtt{e}_{42}(1)  +\frac{1}{k_1}\int_0^1h_2(z)\mathtt{e}_{42}(z) dz+\frac{1}{k_2}\int_0^1h_4(z)\mathtt{e}_{44}(z) dz. 
\end{equation*}
Consequently, we get  \eqref{Part-1-108}-\eqref{Part-1-109}, ending the proof of the lemma. 
\end{proof}
\begin{lem}\label{Theorem-1.16}
\rm{Assume that $\frac{\rho_2}{k_2}=\frac{\rho_1}{k_1}$, we have the following two cases: \\[0.1in]
\textbf{Case 1}.  If there exist no integers $k\in\mathbb{N}$ such that  $\sqrt{\frac{k_1}{k_2}}= k\pi$, then
\begin{equation}\label{Part-1-122}
\left|y(1)\right|=o\left(\lambda^{-\frac{\ell+\alpha+1}{2}} \right).
\end{equation}
\textbf{Case 2}. If there exists  $k_0\in\mathbb{N}$ such that  $\sqrt{\frac{k_1}{k_2}}= k_0\pi$, then
\begin{equation}\label{Part-1-123}
\left|y(1)\right|=o\left(\lambda^{-\frac{\ell+\alpha-3}{2}} \right).
\end{equation}
}
\end{lem}
\begin{proof}
Assume that $\frac{\rho_2}{k_2}= \frac{\rho_1}{k_1}$, then from \eqref{Part-1-111}, we get
\begin{equation}\label{Part-1-124}
\displaystyle{\mu_1=\lambda \sqrt{\frac{\rho_1}{k_1}- \frac{1}{ \lambda}\sqrt{\frac{\rho_1}{k_2}}},\ \mu_2=\lambda \sqrt{\frac{\rho_1}{k_1}+\frac{1}{ \lambda}\sqrt{\frac{\rho_1}{k_2}}},\ \Delta=2\sqrt{\frac{\rho_1}{ k_2} }}.
\end{equation}
Using the asymptotic expansion in \eqref{Part-1-124}, we get
\begin{equation}\label{Part-1-125}
\left\{
\begin{array}{lll}
\displaystyle{\mu_1=\sqrt{\frac{\rho_1}{k_1}} \lambda-\frac{1}{2} \sqrt{\frac{k_1}{k_2}}-\frac{k_1\sqrt{k_1}}{8k_2\sqrt{ \rho_1}\lambda}-\frac{k_1^2\sqrt{k_1}}{16\rho_1k_2\sqrt{k_2} \lambda^2}+O\left(\lambda^{-3}\right)},

\\ \noalign{\medskip}

\displaystyle{\mu_2=\sqrt{\frac{\rho_1}{k_1}} \lambda+\frac{1}{2} \sqrt{\frac{k_1}{k_2}}-\frac{k_1\sqrt{k_1}}{8k_2\sqrt{ \rho_1}\lambda}+\frac{k_1^2\sqrt{k_1}}{16\rho_1k_2\sqrt{k_2} \lambda^2}+O\left(\lambda^{-3}\right)}.

\end{array}
\right.
\end{equation}
Inserting \eqref{Part-1-125} in \eqref{Part-1-110}, then for all $z\in\mathbb{R}$ using the asymptotic expansion, we obtain 
\begin{equation}\label{Part-1-126}
\begin{array}{lll}

\displaystyle{C_{11}(z)=2\cos\left(\frac{z}{2}\sqrt{\frac{k_1}{k_2}}   \right)\cos\left(z\lambda\sqrt{\frac{\rho_1}{k_1}} \right)+
\frac{k_1\sqrt{k_1} z}{4k_2\sqrt{\rho_1}\lambda}
 \cos\left(\frac{z}{2} \sqrt{\frac{k_1}{k_2}}  \right)\sin\left(z \lambda\sqrt{\frac{\rho_1}{k_1}} \right)}

\\ \noalign{\medskip}\hspace{2cm}

\displaystyle{
-\frac{k_1\sqrt{k_1} z}{64\rho_1k_2\sqrt{k_2}\lambda^2}\left( 8\sin\left(\frac{z}{2}\sqrt{\frac{k_1}{k_2}}   \right) +z\sqrt{\frac{k_1}{k_2}}\cos\left(\frac{z }{2}\sqrt{\frac{k_1}{k_2}}  \right)\right) \cos\left(z\lambda\sqrt{\frac{\rho_1}{k_1}} \right)+O\left(\lambda^{-3}\right)},

\end{array}
\end{equation}
\begin{equation}\label{Part-1-127}
\begin{array}{lll}

\displaystyle{C_{12}(z)=2\sin\left(\frac{z}{2}\sqrt{\frac{k_1}{k_2}}   \right)\sin\left(z\lambda\sqrt{\frac{\rho_1}{k_1}} \right)-
\frac{k_1\sqrt{k_1} z}{4k_2\sqrt{\rho_1}\lambda}
 \sin\left(\frac{z}{2} \sqrt{\frac{k_1}{k_2}}  \right)\cos\left(z \lambda\sqrt{\frac{\rho_1}{k_1}} \right)}

\\ \noalign{\medskip}\hspace{1.8cm}

\displaystyle{
-\frac{k_1\sqrt{k_1} z}{64\rho_1k_2\sqrt{k_2}\lambda^2}\left(- 8\cos\left(\frac{z}{2}\sqrt{\frac{k_1}{k_2}}   \right) +z\sqrt{\frac{k_1}{\rho_2}}\sin\left(\frac{z }{2}\sqrt{\frac{k_1}{k_2}}  \right)\right) \sin\left(z\lambda\sqrt{\frac{\rho_1}{k_1}} \right)+O\left(\lambda^{-3}\right)},

\end{array}
\end{equation}
\begin{equation}\label{Part-1-128}
\begin{array}{lll}
\displaystyle{S_{11}(z)=\frac{2\sqrt{{k_1}}}{\sqrt{{\rho_1}} \lambda}\cos\left(\frac{z}{2}\sqrt{\frac{k_1}{k_2}}   \right)\sin\left(z\lambda\sqrt{\frac{\rho_1}{k_1}} \right)}

\\ \noalign{\medskip}\hspace{1cm}

\displaystyle{-\frac{k_1\sqrt{k_1}}{4\rho_1\sqrt{k_2}\lambda^2}\left( 4\sin\left(\frac{z}{2}\sqrt{\frac{k_1}{k_2}}   \right)+z\sqrt{\frac{k_1}{k_2}} \cos\left(\frac{z}{2} \sqrt{\frac{k_1}{k_2}}  \right) \right) \cos\left(z\lambda\sqrt{\frac{\rho_1}{k_1}} \right)+O\left(\lambda^{-4}\right)}

\\ \noalign{\medskip}

\displaystyle{-\frac{k_1^2\sqrt{k_1}}{64 k_2 \rho_1\sqrt{{\rho_1}}\lambda^3} \left(16z\sqrt{\frac{k_1}{k_2}} \sin\left(\frac{z}{2}\sqrt{\frac{k_1}{k_2}}   \right) +\left(\frac{k_1 z^2}{k_2} -48\right)\cos\left(\frac{z}{2}\sqrt{\frac{k_1}{k_2}}   \right) \right)\sin\left(z\lambda\sqrt{\frac{\rho_1}{k_1}} \right)},
\end{array}
\end{equation}
and
\begin{equation}\label{Part-1-129}
\begin{array}{lll}
\displaystyle{S_{12}(z)=-\frac{2\sqrt{{k_1}}}{\sqrt{{\rho_1}} \lambda}\sin\left(\frac{z}{2}\sqrt{\frac{k_1}{k_2}}   \right)\cos\left(z\lambda\sqrt{\frac{\rho_1}{k_1}} \right)}

\\ \noalign{\medskip}\hspace{1cm}

\displaystyle{-\frac{k_1\sqrt{k_1}}{4\rho_1\sqrt{k_2}\lambda^2}\left( -4\cos\left(\frac{z}{2}\sqrt{\frac{k_1}{k_2}}   \right)+z\sqrt{\frac{k_1}{k_2}} \sin\left(\frac{z}{2} \sqrt{\frac{k_1}{k_2}}  \right) \right) \sin\left(z\lambda\sqrt{\frac{\rho_1}{k_1}} \right)+O\left(\lambda^{-4}\right)}

\\ \noalign{\medskip}

\displaystyle{-\frac{k_1^2\sqrt{k_1}}{64 k_2 \rho_1\sqrt{{\rho_1}}\lambda^3}  \left(16z\sqrt{\frac{k_1}{k_2}} \cos\left(\frac{z}{2}\sqrt{\frac{k_1}{k_2}}   \right) -\left(\frac{k_1 z^2}{k_2} -48\right)\sin\left(\frac{z}{2}\sqrt{\frac{k_1}{k_2}}   \right) \right)\cos\left(z\lambda\sqrt{\frac{\rho_1}{k_1}} \right)}.
\end{array}
\end{equation}
Moreover, we have
\begin{equation*}
\int_0^1 S_{11}(z)h_2(z)dz=-\frac{\rho_1}{\lambda^\ell}\int_0^1\left(\frac{\sin(\mu_1 z)}{\mu_1}+\frac{\sin(\mu_2 z)}{\mu_2} \right)f_2(z)dz-\frac{i\rho_1}{\lambda^{\ell-1}}\int_0^1\left(\frac{\sin(\mu_1 z)}{\mu_1}+\frac{\sin(\mu_2 z)}{\mu_2} \right)f_1(z)dz.
\end{equation*}
Using by parts integration, we get 
\begin{equation}\label{Part-1-130}
\begin{array}{ll}

\displaystyle{
\int_0^1 S_{11}(z)h_2(z)dz=-\frac{\rho_1}{\lambda^\ell}\int_0^1\left(\frac{\sin(\mu_1 z)}{\mu_1}+\frac{\sin(\mu_2 z)}{\mu_2} \right)f_2(z)dz+\frac{i\rho_1}{\lambda^{\ell-1}}\left(\frac{\cos(\mu_1 )}{\mu_1^2}+\frac{\cos(\mu_2 )}{\mu_2^2} \right)f_1(1)}

 \\ \noalign{\medskip}

\displaystyle{- \frac{i\rho_1}{\lambda^{\ell-1}}\left(\frac{1}{\mu_1^2}+\frac{1}{\mu_2^2} \right)f_1(0) -\frac{i\rho_1}{\lambda^{\ell-1}}\int_0^1\left(\frac{\cos(\mu_1 z)}{\mu_1^2}+\frac{\cos(\mu_2 z)}{\mu_2^2} \right)(f_1(z))_zdz.}
\end{array}
\end{equation}
From \eqref{Part-1-95} and \eqref{Part-1-125}, we get
\begin{equation}\label{Part-1-131}
\left\{
\begin{array}{ll}

\displaystyle{|f_1(1)|\leq \left\|f_1\right\|_{L^\infty(0,1)}\leq \left\|f_1\right\|_{H^1_L(0,1)}=o(1),\quad |f_1(0)|\leq \left\|f_1\right\|_{L^\infty(0,1)}\leq \left\|f_1\right\|_{H^1_L(0,1)}=o(1),}

 \\ \\

\displaystyle{\frac{1}{\mu_1^2}+\frac{1}{\mu_2^2}=O\left(\lambda^{-2}\right),\ \frac{\cos(\mu_1 )}{\mu_1^2}+\frac{\cos(\mu_2 )}{\mu_2^2}=O\left(\lambda^{-2}\right) },\ \displaystyle{ \frac{\cos(\mu_1 z)}{\mu_1^2}+\frac{\cos(\mu_2 z)}{\mu_2^2}=O\left(\lambda^{-2}\right) }.

\end{array}
\right.
\end{equation}
Substituting  \eqref{Part-1-95}, \eqref{Part-1-128} and \eqref{Part-1-131}   in \eqref{Part-1-130}, we get 
\begin{equation}\label{Part-1-132}
\int_0^1 S_{11}(z)h_2(z)dz=o\left(\lambda^{-\ell-1}\right).
\end{equation}
In the same way, we can check that
 \begin{equation}\label{Part-1-133}
 \left\{
\begin{array}{ll}
\displaystyle{\int_0^1 S_{12}(z)h_2(z)dz=o\left(\lambda^{-\ell-1}\right),\ \int_0^1 S_{11}(z)h_4(z)dz=o\left(\lambda^{-\ell-1}\right),
}

 \\ \noalign{\medskip}

\displaystyle{ \int_0^1 S_{12}(z)h_4(z)dz=o\left(\lambda^{-\ell-1}\right),}

 \\ \noalign{\medskip}

\displaystyle{  \int_0^1 C_{11}(z)h_2(z)dz=o\left(\lambda^{-\ell}\right),\ \int_0^1 C_{11}(z)h_4(z)dz=o\left(\lambda^{-\ell}\right),}
 
  \\ \noalign{\medskip}
 
 \displaystyle{  \int_0^1 C_{12}(z)h_2(z)dz=o\left(\lambda^{-\ell}\right),\ \int_0^1 C_{12}(z)h_4(z)dz=o\left(\lambda^{-\ell}\right).}
\end{array}
\right.
\end{equation} 
Inserting \eqref{Part-1-126}-\eqref{Part-1-129} and \eqref{Part-1-132}-\eqref{Part-1-133} in \eqref{Part-1-108}-\eqref{Part-1-109}, then using the fact that $\frac{\rho_1}{k_1}=\frac{\rho_2}{k_2}$,  we get 

\begin{eqnarray}
y(1)\hspace{0.5mm} \mathcal{J}_1 =o\left(\lambda^{-\frac{\ell+\alpha+1}{2}} \right),\label{Part-1-134}
\\ \noalign{\medskip}
y(1)\hspace{0.5mm}\mathcal{J}_4 =o\left(\lambda^{-\frac{\ell+\alpha+1}{2}} \right),\label{Part-1-135}
\end{eqnarray}
where
\begin{equation*}
\left\{
\begin{array}{ll}

\displaystyle{\mathcal{J}_1 =\left(1-\frac{k_1^2\left(\frac{k_1}{k_2}+16\right)}{128\rho_1 k_2\lambda^2} \right) \sin\left(\frac{1}{2}\sqrt{\frac{k_1}{k_2}}   \right)\cos\left(\sqrt{\frac{\rho_1}{k_1}} \lambda\right)}
 \\ \noalign{\medskip}
\displaystyle{\hspace{2cm}+\frac{k_1}{8\sqrt{\rho_1 k_2}\lambda}\left( 4\cos\left(\frac{1}{2}\sqrt{\frac{k_1}{k_2}}   \right) +\sqrt{\frac{k_1}{k_2}}\sin\left(\frac{1}{2}\sqrt{\frac{k_1}{k_2}}   \right)\right) \sin\left(\sqrt{\frac{\rho_1}{k_1}} \lambda\right) +O\left(\lambda^{-3}\right) },

 \\ \\

\displaystyle{\mathcal{J}_4 =\left(1-\frac{k_1^2\left(\frac{k_1}{k_2}+16\right)}{128\rho_1 k_2\lambda^2}\right) \cos\left(\frac{1}{2}\sqrt{\frac{k_1}{k_2}}   \right)\sin\left(\sqrt{\frac{\rho_1}{k_1}} \lambda\right)}
 \\ \noalign{\medskip}
\displaystyle{\hspace{2cm}+\frac{k_1}{8\sqrt{\rho_1 k_2}\lambda}\left( 4\sin\left(\frac{1}{2} \sqrt{\frac{k_1}{k_2}}  \right) -\sqrt{\frac{k_1}{k_2}}\cos\left(\frac{1}{2}\sqrt{\frac{k_1}{k_2}}   \right)\right) \cos\left(\sqrt{\frac{\rho_1}{k_1}} \lambda\right) +O\left(\lambda^{-3}\right)}.

\end{array}
\right.
\end{equation*}
We distinguish two cases:\\[0.1in]
\textbf{Case 1}. If there exist no integers $k\in\mathbb{N}$ such that  $\sqrt{\frac{k_1}{k_2}}= k\pi$, then 
$$\left|\sin\left(\frac{1}{2}\sqrt{\frac{k_1}{k_2}}\right)\right|\geq c>0\ \ \ \text{and}\ \ \ \left|\cos\left(\frac{1}{2}\sqrt{\frac{k_1}{k_2}}\right)\right|\geq c'> 0,$$
therefore from \eqref{Part-1-134} and \eqref{Part-1-135}, we get
\begin{equation}\label{Part-1-136}
\left\{
\begin{array}{ll}

\displaystyle{\left(\cos\left(\sqrt{\frac{\rho_1}{k_1}} \lambda\right)+O\left(\lambda^{-1}\right) \right) y(1) =o\left(\lambda^{-\frac{\ell+\alpha+1}{2}} \right)},
 \\ \\

\displaystyle{
\left(\sin\left(\sqrt{\frac{\rho_1}{k_1}} \lambda\right)+O\left(\lambda^{-1}\right) \right) y(1) =o\left(\lambda^{-\frac{\ell+\alpha+1}{2}} \right).}
\end{array}
\right.
\end{equation}
Hence, from \eqref{Part-1-136} and using the fact that $\min\left(\frac{\ell+\alpha+1}{2} ,\ell+1\right)=\frac{\ell+\alpha+1}{2}$, we get \eqref{Part-1-122}. \\[0.1in]
\textbf{Case 2}. Assume that  $\sqrt{\frac{k_1}{k_2}}= k_0\pi$, we divide the proof into two cases: Case 2.1, if $\sqrt{\frac{k_1}{k_2}}= 2k_0\pi$ and Case 2.2 if $\sqrt{\frac{k_1}{k_2}}= (2k_0+1)\pi$. Since the argument of two cases is entirely similar, we will only provide one of them.\\[0.1in]
Assume that $\sqrt{\frac{k_1}{k_2}}= 2k_0\pi$, then 
$$\left|\sin\left(\frac{1}{2}\sqrt{\frac{k_1}{k_2}}\right)\right|=0\ \ \ \text{and}\ \ \ \left|\cos\left(\frac{1}{2}\sqrt{\frac{k_1}{k_2}}\right)\right|=1,$$
consequently, from \eqref{Part-1-134} and \eqref{Part-1-135}, we get
\begin{eqnarray}
\left( - \sin\left(\sqrt{\frac{\rho_1}{k_1}} \lambda\right) +O\left(\lambda^{-2}\right)\right) y(1)=o\left(\lambda^{-\frac{\ell+\alpha-1}{2}} \right),\label{Part-1-137}
\\ \noalign{\medskip}
\left(\sin\left(\sqrt{\frac{\rho_1}{k_1}} \lambda\right)-\frac{k_1\sqrt{k_1}}{8k_2\sqrt{\rho_1 }\lambda}\cos\left(\sqrt{\frac{\rho_1}{k_1}} \lambda\right) +O\left(\lambda^{-2}\right)\right) y(1)=o\left(\lambda^{-\frac{\ell+\alpha+1}{2}} \right).\label{Part-1-138}
\end{eqnarray}
Adding \eqref{Part-1-137} and \eqref{Part-1-138}, we get
\begin{equation}\label{Part-1-139}
\left( \cos\left(\sqrt{\frac{\rho_1}{k_1}} \lambda\right) +O\left(\lambda^{-1}\right)\right) y(1)=o\left(\lambda^{-\frac{\ell+\alpha-3}{2}} \right).
\end{equation}
Hence, from \eqref{Part-1-138} and \eqref{Part-1-139},  we get \eqref{Part-1-123}
\end{proof}
\begin{lem}\label{Theorem-1.17}
\rm{Assume that $\frac{\rho_2}{k_2}\neq \frac{\rho_1}{k_1}$, let $\ell=5-\alpha$, for almost all real number $\xi:=\sqrt{\frac{k_1\rho_2}{k_2 \rho_1}}\neq 1$,   we have
\begin{equation}\label{Part-1-140}
\left|\lambda y(1)\right|=o\left(1\right).
\end{equation}
}
\end{lem}
\begin{proof} Assume that $\frac{\rho_2}{k_2}\neq \frac{\rho_1}{k_1}$, then from \eqref{Part-1-111}, we get
\begin{equation}\label{Part-1-141}
\left\{
\begin{array}{lll}

\displaystyle{\mu_1= \sqrt{\dfrac{\left(\frac{\rho_2}{k_2}+\frac{\rho_1}{k_1}\right)\lambda^2-\left(\frac{\rho_2}{k_2}-\frac{\rho_1}{k_1}\right)\lambda^2\sqrt{1+\frac{4\rho_1}{k_2\left(\frac{\rho_2}{k_2}-\frac{\rho_1}{k_1}\right)^2 \lambda^2}} }{2}},}

\\ \noalign{\medskip}

\displaystyle{\mu_2= \sqrt{\dfrac{\left(\frac{\rho_2}{k_2}+\frac{\rho_1}{k_1}\right)\lambda^2+\left(\frac{\rho_2}{k_2}-\frac{\rho_1}{k_1}\right)\lambda^2\sqrt{1+\frac{4{\rho_1}}{{ k_2}\left(\frac{\rho_2}{k_2}-\frac{\rho_1}{k_1}\right)^2 \lambda^2}} }{2}},}

\\ \noalign{\medskip}

\displaystyle{ \Delta=\left(\frac{\rho_2}{k_2}-\frac{\rho_1}{k_1}\right) \lambda\sqrt{1+\frac{4{\rho_1}}{{ k_2}\left(\frac{\rho_2}{k_2}-\frac{\rho_1}{k_1}\right)^2 \lambda^2}}}.

\end{array}
\right.
\end{equation}
Using the asymptotic expansion in \eqref{Part-1-141}, we get
\begin{equation}\label{Part-1-142}
\left\{
\begin{array}{lll}
\displaystyle{\mu_1=\sqrt{\frac{\rho_1}{k_1}} \lambda-\frac{\frac{k_1}{k_2}\sqrt{\frac{\rho_1}{k_1}}}{2\left(\frac{\rho_2}{k_2}-\frac{\rho_1}{k_1}\right)\lambda}+O\left(\lambda^{-3}\right)},\quad

\displaystyle{\mu_2=\sqrt{\frac{\rho_2}{k_2}} \lambda+\frac{\frac{\rho_1}{\rho_2}\sqrt{\frac{\rho_2}{k_2} }}{2\left(\frac{\rho_2}{k_2}-\frac{\rho_1}{k_1}\right)\lambda}
+O\left(\lambda^{-3}\right)},

\\ \noalign{\medskip}

\displaystyle{\Delta=\left(\frac{\rho_2}{k_2}-\frac{\rho_1}{k_1}\right) \lambda+\frac{{2 \frac{\rho_1}{k_2}}}{\left(\frac{\rho_2}{k_2}-\frac{\rho_1}{k_1}\right)\lambda}+O\left(\lambda^{-3}\right)}.
\end{array}
\right.
\end{equation}
Inserting \eqref{Part-1-142} in \eqref{Part-1-110}, then for all $z\in\mathbb{R}$ using the asymptotic expansion, we obtain
\begin{equation}\label{Part-1-143}
C_{11}(z)=\cos\left(\left(\sqrt{\frac{\rho_1}{k_1}} \lambda-\frac{\frac{k_1}{k_2}\sqrt{\frac{\rho_1}{k_1}}}{2\left(\frac{\rho_2}{k_2}-\frac{\rho_1}{k_1}\right)\lambda}\right) z\right)+\cos\left(\left(\sqrt{\frac{\rho_2}{k_2}} \lambda+\frac{\frac{\rho_1}{\rho_2}\sqrt{\frac{\rho_2}{k_2} }}{2\left(\frac{\rho_2}{k_2}-\frac{\rho_1}{k_1}\right)\lambda}
\right) z\right)+O\left(\lambda^{-3}\right),
\end{equation}
\begin{equation}\label{Part-1-144}
C_{12}(z)=\cos\left(\left(\sqrt{\frac{\rho_1}{k_1}} \lambda-\frac{\frac{k_1}{k_2}\sqrt{\frac{\rho_1}{k_1}}}{2\left(\frac{\rho_2}{k_2}-\frac{\rho_1}{k_1}\right)\lambda}\right) z\right)-\cos\left(\left(\sqrt{\frac{\rho_2}{k_2}} \lambda+\frac{\frac{\rho_1}{\rho_2}\sqrt{\frac{\rho_2}{k_2} }}{2\left(\frac{\rho_2}{k_2}-\frac{\rho_1}{k_1}\right)\lambda}
\right) z\right)+O\left(\lambda^{-3}\right),
\end{equation}
\begin{equation}\label{Part-1-145}
\begin{array}{lll}
\displaystyle{S_{11}(z)=\frac{1}{\lambda}{\sqrt{\frac{k_1}{\rho_1}}}\sin\left(\left(\sqrt{\frac{\rho_1}{k_1}} \lambda-\frac{\frac{k_1}{k_2}\sqrt{\frac{\rho_1}{k_1}}}{2\left(\frac{\rho_2}{k_2}-\frac{\rho_1}{k_1}\right)\lambda}\right) z\right)}
\\ \noalign{\medskip}\hspace{3cm}
\displaystyle{+\frac{1}{\lambda}\sqrt{\frac{k_2}{\rho_2}}\sin\left(\left(\sqrt{\frac{\rho_2}{k_2}} \lambda+\frac{\frac{\rho_1}{\rho_2}\sqrt{\frac{\rho_2}{k_2} }}{2\left(\frac{\rho_2}{k_2}-\frac{\rho_1}{k_1}\right)\lambda}
\right) z\right)+O\left(\lambda^{-3}\right),}
\end{array}
\end{equation}
and
\begin{equation}\label{Part-1-146}
\begin{array}{lll}
\displaystyle{S_{12}(z)=\frac{1}{\lambda}\sqrt{\frac{k_1}{\rho_1}}\sin\left(\left(\sqrt{\frac{\rho_1}{k_1}} \lambda-\frac{\frac{k_1}{k_2}\sqrt{\frac{\rho_1}{k_1}}}{2\left(\frac{\rho_2}{k_2}-\frac{\rho_1}{k_1}\right)\lambda}\right) z\right)}
\\ \noalign{\medskip}\hspace{3cm}
\displaystyle{-\frac{1}{\lambda}\sqrt{\frac{k_2}{\rho_2}}\sin\left(\left(\sqrt{\frac{\rho_2}{k_2}} \lambda+\frac{\frac{\rho_1}{\rho_2}\sqrt{\frac{\rho_2}{k_2} }}{2\left(\frac{\rho_2}{k_2}-\frac{\rho_1}{k_1}\right)\lambda}
\right) z\right)+O\left(\lambda^{-3}\right).}
\end{array}
\end{equation}
Substituting \eqref{Part-1-142}-\eqref{Part-1-146} in \eqref{Part-1-108}-\eqref{Part-1-109}, then using the fact that  $\left\|f_2\right\|=o\left(1\right)$ and $\left\|f_4\right\|=o\left(1\right)$, we get 
\begin{equation}\label{Part-1-147}
\begin{array}{ll}
	\displaystyle{		
\left(\sqrt{\frac{\rho_1}{k_1}}\sin\left(\sqrt{\frac{\rho_1}{k_1}} \lambda-\frac{\frac{k_1}{k_2}\sqrt{\frac{\rho_1}{k_1}}}{2\left(\frac{\rho_2}{k_2}-\frac{\rho_1}{k_1}\right)\lambda}\right) -\sqrt{\frac{\rho_2}{k_2}}\sin\left(\sqrt{\frac{\rho_2}{k_2}} \lambda+\frac{\frac{\rho_1}{\rho_2}\sqrt{\frac{\rho_2}{k_2} }}{2\left(\frac{\rho_2}{k_2}-\frac{\rho_1}{k_1}\right)\lambda}
\right)  \right)\lambda y(1) }

 \\ \noalign{\medskip}

\hspace{5cm}	\displaystyle{+O\left(\lambda^{-2 }\right)\lambda y(1)=o\left(\lambda^{-\frac{\ell+\alpha-3}{2}} \right)+o\left(\lambda^{-\ell +2}\right)}
	\end{array}
\end{equation}
and
\begin{equation}\label{Part-1-148}
\left( \sqrt{\frac{\rho_2}{k_2}}\sin\left(\sqrt{\frac{\rho_2}{k_2}} \lambda+\frac{\frac{\rho_1}{\rho_2}\sqrt{\frac{\rho_2}{k_2} }}{2\left(\frac{\rho_2}{k_2}-\frac{\rho_1}{k_1}\right)\lambda}
\right) +O\left({\lambda^{-2} }\right)\right)\lambda y(1) =o\left({\lambda^{-\frac{\ell+\alpha+1}{2}} }\right)+o\left({\lambda^{-\ell}}\right).
\end{equation}
Adding \eqref{Part-1-147} and \eqref{Part-1-148}, we get
\begin{equation}\label{Part-1-149}
\left(\sqrt{\frac{\rho_1}{k_1}}\sin\left( \sqrt{\frac{\rho_1}{k_1}} \lambda-\frac{\frac{k_1}{k_2}\sqrt{\frac{\rho_1}{k_1}}}{2\left(\frac{\rho_2}{k_2}-\frac{\rho_1}{k_1}\right)\lambda}\right)+O\left({\lambda^{-2} }\right)\right) \lambda y(1)=o\left({\lambda^{-\frac{\ell+\alpha-3}{2}} }\right)+o\left({\lambda^{-\ell +2}}\right).
\end{equation}
Let $\ell=5-\alpha$,  from \eqref{Part-1-94}, \eqref{Part-1-103} and \eqref{Part-1-160}, we get $|\lambda y(1)| = O(1).$   Our aim is to show that $|\lambda y(1)| = o(1),$ suppose that there exist two positive constant numbers $c_2\geq c_1>0$ such that  $c_1\leq\left|\lambda y(1)\right|\leq c_2$, then from \eqref{Part-1-148} and \eqref{Part-1-149}, we get
\begin{equation}\label{Part-1-152}
\sin\left(\sqrt{\frac{\rho_2}{k_2}} \lambda+\frac{\frac{\rho_1}{\rho_2}\sqrt{\frac{\rho_2}{k_2} }}{2\left(\frac{\rho_2}{k_2}-\frac{\rho_1}{k_1}\right)\lambda}
\right) =o\left({\lambda^{-1} }\right)\ \ \ \text{and}\ \ \
 \sin\left(\sqrt{\frac{\rho_1}{k_1}} \lambda-\frac{\frac{k_1}{k_2}\sqrt{\frac{\rho_1}{k_1}}}{2\left(\frac{\rho_2}{k_2}-\frac{\rho_1}{k_1}\right)\lambda}\right)=o\left({\lambda^{-1} }\right).
\end{equation}
It follows from Equation \eqref{Part-1-152}, there exists $n,m\in \mathbb{Z}$ such that 
    	\begin{eqnarray}
  \lambda=n\pi \sqrt{\frac{k_2}{\rho_2}} -\frac{\frac{\rho_1}{\rho_2}}{2\left(\frac{\rho_2}{k_2}-\frac{\rho_1}{k_1}\right)\lambda}  +o\left({\lambda^{-1} }\right),\label{Part-1-153}
     \\ \noalign{\medskip}
 \lambda    =m\pi\sqrt{\frac{k_1}{\rho_1}}+\frac{\frac{k_1}{k_2}}{2\left(\frac{\rho_2}{k_2}-\frac{\rho_1}{k_1}\right)\lambda}+o\left({\lambda^{-1} }\right).\label{Part-1-154}
    	\end{eqnarray}
 Subtracting    \eqref{Part-1-153} from \eqref{Part-1-154}, we get 
    \begin{equation*}
\pi m \sqrt{\frac{k_2}{\rho_2}} \left(\frac{n}{m} -\sqrt{\frac{k_1\rho_2}{\rho_1 k_2}}\right)= \frac{\frac{\rho_1}{\rho_2}+\frac{k_1}{k_2}}{2\left(\frac{\rho_2}{k_2}-\frac{\rho_1}{k_1}\right)\lambda}   +o\left({\lambda^{-1} }\right). 
    \end{equation*}
  Equivalently, we have  
    \begin{equation}\label{Part-1-155}
\frac{n}{m} -\sqrt{\frac{k_1\rho_2}{\rho_1 k_2}}= \frac{ \frac{k_1}{\sqrt{\rho_2 k_2}}\left(\frac{k_1 \rho_2}{k_2\rho_1}+1 \right)}{2\pi\left(\frac{k_1\rho_2 }{k_2 \rho_1}-1\right) m \lambda}   +\frac{o(1)}{m\lambda} . 
    \end{equation}	
  From     \eqref{Part-1-154}, we get
 \begin{equation}\label{Part-1-156}
 \frac{1}{\lambda}=\frac{\sqrt{\frac{\rho_1}{k_1}}}{m\pi}+\frac{o\left(1\right)}{m^2 }.
 \end{equation}   	
 Inserting    \eqref{Part-1-156} in \eqref{Part-1-155}, we get	    \begin{equation}\label{Part-1-157}
\frac{n}{m} -\sqrt{\frac{k_1\rho_2}{\rho_1 k_2}}= \frac{ \sqrt{ \frac{k_1 \rho_1}{k_2 \rho_2}}\left(\frac{k_1 \rho_2}{k_2\rho_1}+1 \right)}{2\left(\frac{k_1\rho_2 }{k_2 \rho_1}-1\right) \pi^2 m^2 }   +\frac{o(1)}{m^2} . 
    \end{equation}
From Theorem 1.10 in  \cite{Bugeaud01}, we have for almost all real numbers $\xi$  there exists infinitely many integers $n,\ m$ such that
    \begin{equation}\label{Part-1-158}
\left| \xi-\frac{n}{m}    \right|<\frac{1}{m^2 \ln|m|}. 
    \end{equation}
Let 	$\xi=\sqrt{\frac{k_1\rho_2}{\rho_1 k_2}}$, then from \eqref{Part-1-157} and \eqref{Part-1-158} there exist infinitely many integers $n,\ m$ such that
\begin{equation*}
\left|\frac{ \frac{\rho_1}{\rho_2}\xi\left(\xi^2+1 \right)}{2\left(\xi^2-1\right) \pi^2 m^2 }   +\frac{o(1)}{m^2} \right|=\left|\frac{n}{m} -\xi\right|<\frac{1}{m^2 \ln|m|}.
\end{equation*}
  Equivalently, we have 
 \begin{equation}\label{Part-1-159}
\left|\frac{ \frac{\rho_1}{\rho_2}\xi\left(\xi^2+1 \right)}{2\left(\xi^2-1\right) \pi^2  }   +o(1) \right|<\frac{1}{ \ln|m|}.
\end{equation}   	
 Since $m\sim C_0\vert \lambda\vert$ for a positive constant, then the estimate   \eqref{Part-1-159} can be written as 	\begin{equation*}
\left|\frac{ \frac{\rho_1}{\rho_2}\xi\left(\xi^2+1 \right)}{2\left(\xi^2-1\right) \pi^2  }   +o(1) \right|=o(1).
\end{equation*}    	
Consequently, we have 
  \begin{equation*}
\left|\frac{ \xi\left(\xi^2+1 \right)}{\left(\xi^2-1\right) }  \right|=o(1),
\end{equation*}      	
  which is impossible. Therefore, we get $\lambda y(1)=o(1)$, which concludes the proof of the lemma.
 \end{proof}$\\[0.1in]$
We now turn to the proof of Theorem \ref{Theorem-1.11}.\\[0.1in]
\textbf{Proof of Theorem \ref{Theorem-1.11}}. We divide the proof in two steps: \\[0.1in] \noindent {\bf Step1. The energy decay estimation.} We distinguish two cases:\\[0.05in]
\textbf{Case 1}. If there exist no integers $k\in\mathbb{N}$ such that  $\sqrt{\frac{k_1}{k_2}}= k\pi$, let $\ell=1-\alpha$,  then from \eqref{Part-1-103}, \eqref{Part-1-104} and \eqref{Part-1-122}, we get
\begin{equation}\label{Part-1-169}
\left|u_x(1)\right|=o\left(1\right),\ \left|\lambda u(1)\right|=o\left(1\right),\ \left|\lambda y(1)\right|=o\left(1\right).
\end{equation}
Inserting \eqref{Part-1-169} in \eqref{Part-1-160}, we get
\begin{equation*}
\int_0^{1}\left(\rho_1\left|\lambda u\right|^2+\rho_2\left|\lambda u\right|^2+k_1\left| u_x\right|^2 +k_2\left| y_x\right|^2\right)dx=o(1),
\end{equation*}
then $\|U\|_{\mathcal{H}_1}=o\left(1\right)$ which contradicts \eqref{Part-1-94}. This implies that $$\displaystyle{\sup_{\lambda\in\mathbb{R}}\left\|\left(i\lambda Id-\mathcal{A}_1\right)^{-1}\right\|_{\mathcal{L}\left(\mathcal{H}_1\right)}=O\left(\lambda^{1-\alpha}\right)}.$$
\textbf{Case 2}. If there exists  $k_0\in\mathbb{N}$ such that  $\sqrt{\frac{k_1}{k_2}}= k_0\pi$, let $\ell=5-\alpha$,  then from \eqref{Part-1-103}, \eqref{Part-1-104} and \eqref{Part-1-123}, we get
\begin{equation}\label{Part-1-170}
\left|u_x(1)\right|=o\left(1\right),\ \left|\lambda u(1)\right|=o\left(1\right),\ \left|\lambda y(1)\right|=o\left(1\right).
\end{equation}
Inserting \eqref{Part-1-170} in \eqref{Part-1-160}, we get
\begin{equation*}
\int_0^{1}\left(\rho_1\left|\lambda u\right|^2+\rho_2\left|\lambda u\right|^2+k_1\left| u_x\right|^2 +k_2\left| y_x\right|^2\right)dx=o(1),
\end{equation*}
then $\|U\|_{\mathcal{H}_1}=o\left(1\right)$ which contradicts \eqref{Part-1-94}. This implies that $$\displaystyle{\sup_{\lambda\in\mathbb{R}}\left\|\left(i\lambda Id-\mathcal{A}_1\right)^{-1}\right\|_{\mathcal{L}\left(\mathcal{H}_1\right)}=O\left(\lambda^{5-\alpha}\right)}.$$ 
\noindent {\bf Step 2. The optimality.} For the optimality of \eqref{ee}, let $\epsilon>0$ and set 
\begin{equation*}
S=\left\{\begin{array}{ll}
\displaystyle{1-\alpha-\epsilon,\quad \text{if }\sqrt{\frac{k_1}{k_2}}\not\in  \pi\mathbb{N}},
\\ \\
\displaystyle{5-\alpha-\epsilon,\quad \text{if }\sqrt{\frac{k_1}{k_2}}\in  2\pi\mathbb{N}},
\\ \\
\displaystyle{5-\alpha-\epsilon,\quad \text{if }\sqrt{\frac{k_1}{k_2}}\in  (2\mathbb{N}+1)\pi}.
\end{array}\right.
\end{equation*}
 For $|n|\geq n_0$, let 
\begin{equation*}
\lambda_n=\left\{\begin{array}{ll}
\displaystyle{\lambda_{1,n},\quad \text{if }\sqrt{\frac{k_1}{k_2}}\not\in  \pi\mathbb{N}},
\\ \\
\displaystyle{\lambda_{1,n},\quad \text{if }\sqrt{\frac{k_1}{k_2}}\in  2\pi\mathbb{N}},
\\ \\
\displaystyle{\lambda_{2,n},\quad \text{if }\sqrt{\frac{k_1}{k_2}}\in  (2\mathbb{N}+1)\pi},
\end{array}\right.
\end{equation*}
where $\left(\lambda_{1,n}\right)_{ |n|\geq n_0} $ and $\left(\lambda_{2,n}\right)_{ |n|\geq n_0} $ are the  simple eigenvalues of $\mathcal{A}_1$. Moreover, let $U_{n}\in D(\mathcal{A}_1)$ be the    normalized eigenfunction corresponding to $\lambda_n$. We introduce the following sequence
$$
\beta_{n}=-\text{Im} (\lambda_{n}),\quad |n|\geq n_0.
$$
Therefore, we have
$$ (iI\beta_{n}+\mathcal{A}_1)U_{n}=(iI\beta_{n}+\lambda_{n})U_{n}=\text{Re}\left(\lambda_n\right)U_n,\quad\forall |n|\geq n_0. $$
From Proposition \ref{Theorem-1.10new},  we get
\begin{equation*}
(iI\beta_{n}+\mathcal{A}_1)U_{n}=\left\{\begin{array}{ll}
\displaystyle{\frac{C_1}{n^{1-\alpha}}+o\left(\frac{1}{n^{1-\alpha}}\right)},\quad& \displaystyle{\text{if }\sqrt{\frac{k_1}{k_2}}\not\in  \pi\mathbb{N}},
\\ \\
\displaystyle{\frac{C_2}{n^{5-\alpha}}+o\left(\frac{1}{n^{5-\alpha}}\right),}\quad & \displaystyle{ \text{if }\sqrt{\frac{k_1}{k_2}}\in  2\pi\mathbb{N}},
\\ \\
\displaystyle{\frac{C_3}{n^{5-\alpha}}+o\left(\frac{1}{n^{5-\alpha}}\right),}\quad& \displaystyle{ \text{if }\sqrt{\frac{k_1}{k_2}}\in  (2\mathbb{N}+1)\pi},
\end{array}\right.
\end{equation*}
where $C_1,\ C_2,\ C_3$ are non zero real numbers. Hence 
$$ 
\beta_{n}^{S}\| (i\beta_{n}I+\mathcal{A}_1U_{n} \|_{\mathcal{H}_1}\sim  \dfrac{C}{n^{\epsilon}},\quad \forall |n|\geq n_{0},
$$
where $C>0$. Thus, we deduce 
$$ 
\lim_{|n|\to +\infty}\beta_{n}^{S}\| (i\beta_{k}I+\mathcal{A}_1)U_{n} \|_{\mathcal{H}_1}=0. 
$$
Finally, thanks to  Theorem \ref{chapter pr-44},  we cannot expect the energy decay rate $t^{-\frac{2}{S}}$. Therefore,  estimate  \eqref{ee} is optimal. 
\xqed{$\square$}
\\[0.1in]
\textbf{Proof of Theorem \ref{Theorem-1.12}}. For almost all real numbers $\sqrt{\frac{k_1\rho_2}{k_2 \rho_1}}$,   from \eqref{Part-1-103}, \eqref{Part-1-104} and \eqref{Part-1-140}, we have
\begin{equation}\label{Part-1-171}
\left|u_x(1)\right|=o\left(1\right),\ \left|\lambda u(1)\right|=o\left(1\right),\ \left|\lambda y(1)\right|=o\left(1\right).
\end{equation}
Inserting \eqref{Part-1-171} in \eqref{Part-1-160}, we get
\begin{equation*}
\int_0^{1}\left(\rho_1\left|\lambda u\right|^2+\rho_2\left|\lambda u\right|^2+k_1\left| u_x\right|^2 +k_2\left| y_x\right|^2\right)dx=o(1),
\end{equation*}
then  $\|U\|_{\mathcal{H}_1}=o\left(1\right)$ which contradicts \eqref{Part-1-94}. This implies that $$\displaystyle{\sup_{\lambda\in\mathbb{R}}\left\|\left(i\lambda Id-\mathcal{A}_1\right)^{-1}\right\|_{\mathcal{L}\left(\mathcal{H}_1\right)}=O\left(\lambda^{5-\alpha}\right)}.$$ The result follows from  Theorem \ref{chapter pr-44}. 
\xqed{$\square$}
\section{Exact controllability of the Timoshenko system}\label{Section-2-Exact-controllability}
\noindent In this section, we study the indirect boundary exact controllability of the Timoshenko System \eqref{Part-0-01} with the  boundary conditions \eqref{Part-0-03}. This system  defined in $\left(0,1\right)\times\left(0,+\infty\right)$ takes the following  

\begin{equation}\label{Equation1.1-work4}  
\left\{
\begin{array}{lll}

\displaystyle{\rho_1 u_{tt}-k_1 \left(u_x+y\right)_x=0,}
&\displaystyle{(x,t)\in\left(0,1\right)\times \mathbb{R}_+,}

         \nline
         
\displaystyle{ \rho_2y_{tt}-k_2y_{xx}+k_1\left(u_x+y\right)=0,}  &\displaystyle{(x,t)\in\left(0,1\right)\times \mathbb{R}_+,}

          \nline
          
\displaystyle{u\left(1,t\right)=v\left(t\right)}, &\displaystyle{t\in\mathbb{R}_+,}

           \nline
           
\displaystyle{y_x\left(0,t\right)=y_x\left(1,t\right)=u\left(0,t\right)=0}, &\displaystyle{t\in\mathbb{R}_+,}

\end{array}
\right.
\end{equation}
in addition to  the following initial conditions
\begin{eqnarray*}
u(x,0)=u_0(x),& u_t(x,0)=u_1(x),&x\in   (0,1),\nline
y(x,0)=y_0(x),& y_t(x,0)=y_1(x),&x\in(0,1).
\end{eqnarray*}
The  control  $v$ is applied only on the right boundary of the first equation. The second equation is indirectly controlled by means of the coupling between the equations.\\[0.1in]
\noindent  For a given $T > 0$ and initial data $\left(u_0,u_1,y_0,y_1\right)$ belonging to a suitable space,  the aim of this section is  to find a suitable   control $v$ such  that the solution of the System \eqref{Equation1.1-work4},  given by  $\left(u,u_t, y, y_t\right)$, is driven to zero in time $T$; i.e.,
$$u\left(x,T\right)=u_t\left(x,T\right)= y\left(x,T\right)=y_t\left(x,T\right)=0,\quad\forall x\in\left(0,1\right).$$
\subsection{Spectral compensation for homogeneous  Timoshenko system }\label{Section2-work4}
\noindent The aim of this section is to compute the eigenvalues and the eigenvectors associated to the  homogeneous Timoshenko system. For this aim, we  consider the  homogeneous Timoshenko system
\begin{equation}\label{Equation2.1-work4}  
\left\{
\begin{array}{lll}

\displaystyle{\rho_1\varphi_{tt}- k_1\left(\varphi_x+\psi\right)_x=0,}
\quad\quad &\displaystyle{(x,t)\in\left(0,1\right)\times \mathbb{R}_+,}

            \nline

\displaystyle{ \rho_2\psi_{tt}-k_2\psi_{xx}+k_1\left(\varphi_x+\psi\right)=0,} \quad\quad &\displaystyle{(x,t)\in\left(0,1\right)\times \mathbb{R}_+,}

             \nline
 
\displaystyle{\varphi\left(0,t\right)=\varphi\left(1,t\right)=\psi_x\left(0,t\right)=\psi_x\left(1,t\right)=0},\quad\quad &\displaystyle{t\in\mathbb{R}_+,}

\end{array}
\right.
\end{equation}
with the  following initial conditions
\begin{equation*}
\varphi\left(x,0\right)=\varphi_0\left(x\right),\ \psi\left(x,0\right)=\psi_0\left(x\right),\ 
\varphi_t\left(x,0\right)=\varphi_1\left(x\right),\
\psi_t\left(x,0\right)=\psi_1\left(x\right),\qquad x\in\left(0,1\right),
\end{equation*}
where $k_1,\ k_2,\ \rho_1$ and $\rho_2$ are strictly positive constants.\\[0.1in]
The energy of System \eqref{Equation2.1-work4}   is given by
\begin{equation*}
	E_2\left(t\right)= \rho_1\int_0^{1}\left|\varphi_t\right|^{2}dx+\rho_2\int_0^{1}\left|\psi_t\right|^2dx+k_1\int_0^{1}\left|\varphi_x+\psi\right|^{2}dx+k_2\int_0^{1}\left|\psi_x\right|^{2}dx,
	\end{equation*}
	
a direct computation gives
\begin{equation*}
\frac{dE_2}{dt}\left(t\right)= 0.
\end{equation*}
Thus, the energy of the solution is conserved. Let us define the energy space ${\mathcal{H}_2}$ by
\begin{equation*}
{\mathcal{H}_2}=H_0^1\left(0,1\right)\times L^2\left(0,1\right)\times H^1_{*}\left(0,1\right)\times L^2\left(0,1\right)
\end{equation*}

with the inner product defined by
\begin{equation}\label{inner product-2}
\begin{array}{lll}

\displaystyle{
\left<\Phi,\Phi_1\right>_{\mathcal{H}_2}=\rho_1 \int_0^1\varphi_1\overline{\tilde{\varphi}_1}dx+\rho_2 \int_0^1\psi_1\overline{\tilde{\psi}_1}dx+k_1\int_0^1\left(\varphi_x+\psi\right)\overline{\left(\tilde{\varphi}_x+\tilde{\psi}\right)}dx}\displaystyle{+k_2\int_0^1 \psi_x \overline{\tilde{\psi}_x}dx,}
\end{array}
\end{equation}
for  all $\Phi=\left(\varphi,\varphi_1,\psi,\psi_1\right),\ \Phi_1=\left(\tilde{\varphi},\tilde{\varphi}_1,\tilde{\psi},\tilde{\psi}_1\right)\in\mathcal{H}_2.$ We use 
$\|U\|_{\mathcal{H}_2}$ to denote the corresponding norm.\\
We define the linear unbounded operator ${\mathcal{A}_2}$ in ${\mathcal{H}_2}$ by
\begin{equation*}
{\mathcal{A}_2}(\varphi,\varphi_1,\psi,\psi_1)=
\left(\varphi_1,\frac{k_1}{\rho_1}\left(\varphi_x+\psi\right)_x,
\psi_1,\frac{k_2}{\rho_2}\psi_{xx}-\frac{k_1}{\rho_2}\left(\varphi_x+\psi\right)\right)
\end{equation*}
and
\begin{equation*}
\begin{array}{lll}

\displaystyle{D\left({\mathcal{A}_2}\right)=\bigg\{\Phi=(\varphi,\varphi_1,\psi,\psi_1)\ |\ \varphi\in H^2\left(0,1\right)\cap H^1_0\left(0,1\right),}

             \\

\hspace{3cm}\displaystyle{\psi \in H^2\left(0,1\right)\cap H^1_{*}\left(0,1\right) ,\ \varphi_1,\ \psi_x\in H^1_0\left(0,1\right),\  \psi_1\in H^1_{*}\left(0,1\right)\bigg\}}.
\end{array}
\end{equation*}
Therefore, we can write the System \eqref{Equation2.1-work4} as an evolution equation on the Hilbert space ${\mathcal{H}_2}$:
\begin{equation}\label{Equation2.2-work4}
\left\{
\begin{array}{c}

\displaystyle{ \Phi_t(x,t)={\mathcal{A}_2}\Phi(x,t),}

            \nline
 
\displaystyle{\Phi\left(x,0\right)=\Phi_0(x),}

\end{array}
\right.
\end{equation}
where $\Phi_0\left(x\right)=\left(\varphi_0\left(x\right),\varphi_1\left(x\right),\psi_0\left(x\right),\psi_1\left(x\right)\right)\in{\mathcal{H}_2}.$\\[0.1in]
\noindent One clearly  that ${\mathcal{A}_2}$ is a maximal dissipative operator on ${\mathcal{H}_2}$,  then  by Lumer Philips's Theorem (see Theorem 4.3 in \cite{Pazy01}),  ${\mathcal{A}_2}$ is the infinitesimal generator of a $C_0$-semigroup of contractions $e^{t{\mathcal{A}_2}}$ on ${\mathcal{H}_2}.$ Therefore,  the problem \eqref{Equation2.1-work4} is well-posed and we have the following result.
\begin{thm}\label{Theorem2.1-work4}
\rm{For any $\Phi_0\in D\left({\mathcal{A}_2}\right),$ the problem \eqref{Equation2.2-work4}  admits a unique strong solution 
$$\Phi\in C\left(\mathbb{R}_{+};D\left({\mathcal{A}_2}\right)\right) \cap C^1\left(\mathbb{R}_{+};{\mathcal{H}_2}\right).$$
Moreover, if $\Phi_0\in{\mathcal{H}_2}$,  then the  System \eqref{Equation2.2-work4}  admits a unique weak solution 
$$\Phi\in C\left(\mathbb{R}_{+};{\mathcal{H}_2}\right).$$ 
In addition, we have
$$\left\|\Phi\left(t\right)\right\|_{{\mathcal{H}_2}}=\left\|\Phi_0\right\|_{{\mathcal{H}_2}},\quad \forall\ t\in \mathbb{R}_+.$$}\xqed{$\square$}
\end{thm}
\noindent  Since ${\mathcal{A}_2}$ is a closed operator with a compact resolvent, its spectrum $\sigma({\mathcal{A}_2})$ consists entirely of isolated eigenvalues with finite multiplicities (see Theorem 6.29 in  \cite{Kato01}). Moreover, it is easy to check that $0\in\rho({\mathcal{A}_2})$.\\[0.1in]
We will now  study the spectrum of the system study the spectrum of the System \eqref{Equation2.1-work4}. Let $\lambda\in\mathbb{C^*}$ be an eigenvalue of the operator  ${\mathcal{A}_2}$ and  $E=\left(\varphi,\lambda\varphi,\psi,\lambda\psi\right)\in D({\mathcal{A}_2})$  a corresponding eigenvector. Using the fact that 
$\text{Re}\left\{\lambda\right\}\left\|E\right\|_{{\mathcal{H}_2}}^2= \text{Re}\left<{\mathcal{A}_2} E,E\right>_{{\mathcal{H}_2}}=0,$
we get that $\lambda=i\mu\text{ with } \mu\in\mathbb{R}^*.$  Then, the corresponding eigenvalue problem is given by
\begin{equation}\label{Equation2.3-work4}
\left\{
\begin{array}{ll}

\displaystyle{\frac{k_1}{\rho_1} \varphi_{xx}+\mu^2 \varphi+\frac{k_1}{\rho_1}\psi_x=0,}

 \nline
 
  \displaystyle{-\frac{k_1}{\rho_2}\varphi_x+\frac{k_2}{\rho_2} \psi_{xx}+\mu^2 \psi-\frac{k_1}{\rho_2}\psi=0,}
  
   \nline
  
    \displaystyle{\varphi\left(0\right)=\varphi\left(1\right)=\psi_x\left(0\right)=\psi_x\left(1\right)=0.}

  \end{array}
  \right.
\end{equation}
For some constants $C,\ D\in\mathbb{C}^*$, let
\begin{equation}\label{Equation2.5-work4}
\varphi(x)=C\sin(n\pi x), \hskip1cm \psi(x)=D\cos(n\pi x)
\end{equation}
be a solution of \eqref{Equation2.3-work4}.
It follows that 
\begin{equation}\label{Equation2.3new-work4}
\left\{
\begin{array}{ll}
\displaystyle{\left(\mu^{2}-\frac{k_1}{\rho_1}(n\pi)^2\right)C-\frac{k_1}{\rho_1}(n\pi) D =0} ,\nline

\displaystyle{-\frac{k_1}{\rho_2}(n\pi) C+\left(\mu^{2}-\frac{k_2}{\rho_2}(n\pi)^2-\frac{k_1}{\rho_2}\right)D=0,}
\end{array}
\right.
\end{equation}
which has a non-trivial solution if and only if
\begin{equation}\label{Equation2.6-work4}
\mu^{4}-\left(\left(\frac{k_1}{\rho_1}+\frac{k_2}{\rho_2}\right) (n\pi)^2+\frac{k_1}{\rho_2} \right)\mu^2+\frac{k_1k_2}{\rho_1\rho_2}  (n\pi)^4=0.
\end{equation}
\begin{rk}\rm{
\noindent The solution  $\left(\varphi,\psi\right)$  of \eqref{Equation2.3-work4} is given by $e^{r x}\left(\alpha_1,\alpha_2\right)$, such that
\begin{equation}\label{Timoshenko Equation 2.15}
\begin{pmatrix}
\frac{k_1}{\rho_1}r^2+\mu^2& \frac{k_1}{\rho_1}r\nline -\frac{k_1}{\rho_2}r& \frac{k_2}{\rho_2}r^2+\mu^2-\frac{k_1}{\rho_2}
\end{pmatrix} \begin{pmatrix}
\alpha_1\nline \alpha_2
\end{pmatrix}= \begin{pmatrix}
0\nline 0
\end{pmatrix}.
\end{equation}
System \eqref{Timoshenko Equation 2.15} admits a non zero solution if and only if
\begin{equation}\label{Timoshenko Equation 2.16}
r^4+\left(\frac{\rho_1}{k_1} +\frac{\rho_2}{k_2}  \right)\mu^2r^2+\frac{\rho_1\rho_2}{k_1k_2} \mu^2\left(\mu^2-\frac{k_1}{\rho_2} \right)=0.
\end{equation}
Solving Equation \eqref{Timoshenko Equation 2.16}, we get
\begin{equation*}
\displaystyle{r_1^2=-\dfrac{\mu^2\left( \frac{\rho_2}{{k}_2} +\frac{\rho_1}{k_1}\right)+\sqrt{\left( \frac{\rho_2}{{k}_2} -\frac{\rho_1}{k_1}\right)^2\mu^4+\frac{4\mu^2\rho_1}{{k}_2}}}{2},}\quad  \displaystyle{r_2^2=\dfrac{-\mu^2\left( \frac{\rho_2}{{k}_2} +\frac{\rho_1}{k_1}\right)+\sqrt{\left( \frac{\rho_2}{{k}_2} -\frac{\rho_1}{k_1}\right)^2\mu^4+\frac{4\mu^2\rho_1}{{k}_2}}}{2}.}
	\end{equation*}
From the boundary conditions \eqref{Equation2.3-work4}, the System \eqref{Equation2.3-work4} has a non-trivial solution if and only if $\sinh(r_1)=0$ and/or $\sinh(r_2)=0$. Taking 
\begin{equation}\tag{$A_1$}
\frac{k_1}{\rho_2}\neq
\frac{\left(\frac{k_2 m_1^2}{\rho_2}-\frac{k_1 m_2^2}{\rho_1}\right)\left(\frac{k_1 m_1^2}{\rho_1}-\frac{k_2 m_2^2}{\rho_2}\right)\pi^2}{\left(\frac{k_1}{\rho_1}+\frac{k_2}{\rho_2}\right)\left(m_1^2+m_2^2\right)},\quad\ \forall\ m_1,\ m_2\in \mathbb{Z},
\end{equation}
we get $$\sinh(r_1)=0 \text{ and }\sinh(r_2)\neq 0\quad \text{or}\quad \sinh(r_1)\neq0 \text{ and }\sinh(r_2)= 0.$$
In this case, the solution of \eqref{Equation2.3-work4},  $\left(\varphi,\psi\right)$, is uniquely written  as defined in \eqref{Equation2.5-work4}. }\xqed{$\square$}
\end{rk}
\subsection{Observability and exact controllability under equal speeds wave propagation condition} \label{Section3-work4}
\noindent In this part, assume that the waves propagate with the same speeds; i.e., $\frac{k_1}{\rho_1}=\frac{k_2}{\rho_2}$. In this case, we study exact controllability of a one dimensional Timoshenko System \eqref{Equation1.1-work4}. For this aim, first, we  prove the following Observability theorem.
\begin{thm}\label{Theorem3.2-work4}
\rm{Assume that $\frac{k_1}{\rho_1}=\frac{k_2}{\rho_2}$, $T>4\sqrt{ \frac{\rho_1}{k_1} }$ and condition $(\rm A_1)$ holds. Then, for all solution $\Phi=\left(\varphi,\varphi_t,\psi,\psi_t\right)$ that solve the  Cauchy problem \eqref{Equation2.1-work4} there exists a positive  constant $\ell_0$ depending only on $k_1,\ k_2,\ \rho_1,\ \rho_2$ such that the following direct inequality holds:
\begin{equation}\label{Equation3.2-work4}
\int_{0}^T\left|\varphi_x\left(1,t\right)+\psi(1,t)\right|^2dt\leq \ell_0 \left\|\left(\varphi_0,\varphi_1,\psi_0,\psi_1\right)\right\|^2_{{\mathcal{H}_2}} .
\end{equation}
Moreover, there exists a  positive constant $\ell_1< \ell_0$ depending only on $k_1,\ k_2,\ \rho_1,\ \rho_2$ such that the following inverse observability inequalities hold:\\[0.1in]
 \noindent\textbf{Case 1.}   If  there exist  no integers $k\in \mathbb{N}$ such that $\sqrt{\frac{k_1}{k_2}}=k\pi$, then 
\begin{equation}\label{Equation3.3-work4}
\ell_1 \left\|\left(\varphi_0,\varphi_1,\psi_0,\psi_1\right)\right\|^2_{{\mathcal{H}_2}}\leq\int_{0}^T\left|\varphi_x\left(1,t\right)+\psi(1,t)\right|^2dt.
\end{equation}
 \noindent\textbf{Case 2.} If there exists  $k_0 \in \mathbb{N}^*$ such that $\sqrt{\frac{k_1}{k_2}}=k_0\pi$, then  there exists Hilbert space $\mathcal{D}$, defined by
\begin{equation*}
\mathcal{D}=\left\{ \Phi_0=\sum_{n\neq 0}(\alpha_{1,n}E_{1,n}+\alpha_{2,n} E_{2,n})n^{2}\ |\  \alpha_{1,n},\ \alpha_{2,n}\in\mathbb{C},\ \sum_{n\neq 0}\left(\left|\alpha_{1,n}\right|^2+\left|\alpha_{2,n}\right|^2\right)<\infty \right\}
\end{equation*}
 equipped with the following  norm
$$\sum_{n\neq 0}\left(\left|\alpha_{1,n}\right|^2+\left|\alpha_{2,n}\right|^2\right),$$ 
such that the inverse inverse observability holds:
\begin{equation}\label{Equation3.4-work4}
\ell_1\left\|\left(\varphi_0,\varphi_1,\psi_0,\psi_1\right)\right\|^2_{{\mathcal{D}}}\leq  \int_{0}^T\left|\varphi_x\left(1,t\right)+\psi(1,t)\right|^2dt,
\end{equation}
where $E_{1,n},\ E_{2,n}$ are the eigenfunctions of the operator ${\mathcal{A}_2}$. }
\end{thm}
\noindent For the proof of Theorem \ref{Theorem3.2-work4}, we use the spectrum method. For this aim, we need to study the asymptotic behaviour of the spectrum of ${\mathcal{A}_2}$. We prove the following proposition.
\begin{prop}\label{Proposition3.3-work4}
\rm{Assume that $\frac{k_1}{\rho_1}=\frac{k_2}{\rho_2}$ and condition $(\rm A_1)$ holds. Then, the  eigenvalues of ${\mathcal{A}_2}$ has the following asymptotic behavior
\begin{equation}
\label{Equation3.5-work4}
\lambda_{1,n}= in\pi\sqrt{\frac{k_1}{\rho_1}}+\frac{i}{2}\sqrt{\frac{k_1}{\rho_2}}+i\frac{\sqrt{{k_1\rho_1}}}{8{\rho_2} n\pi}+O\left(n^{-3}\right),
\end{equation}
\begin{equation}
\label{Equation3.6-work4}
\lambda_{2,n}= in\pi\sqrt{\frac{k_1}{\rho_1}}-\frac{i}{2}\sqrt{\frac{k_1}{\rho_2}}+i\frac{\sqrt{{k_1\rho_1}}}{8{\rho_2} n\pi}+O\left(n^{-3}\right),
\end{equation}
with the corresponding eigenfunctions
\begin{equation}
\label{Equation3.7-work4}
\varphi_{1,n}(x)=\frac{\sin(n\pi x)}{n\pi}, \quad \quad \psi_{1,n}(x)=\sqrt{\frac{k_1}{k_2}} \frac{\cos(n\pi x)}{n \pi}+O\left(n^{-2}\right)\cos(n\pi x),
\end{equation}
\begin{equation}
\label{Equation3.8-work4}
\varphi_{2,n}(x)= -\sqrt{\frac{k_2}{k_1}} \frac{\sin(n\pi x)}{n \pi}+O\left(n^{-2}\right)\sin(n\pi x), 
\quad\psi_{2,n}(x)=\frac{\cos(n\pi x)}{n\pi}.
\end{equation}}
\end{prop}
\begin{proof} Assume that $\frac{k_1}{\rho_1}=\frac{k_2}{\rho_2}$, by solving \eqref{Equation2.6-work4}, we get
\begin{equation}\label{Equation3.9-work4}
\left\{\begin{array}{ll}
\displaystyle{\mu^2_{1,n}={(n\pi)^2}\frac{k_1}{\rho_1}+\frac{k_1}{2\rho_2}+ \frac{1}{2}\sqrt{\frac{4k_1^2}{\rho_1\rho_2}(n\pi)^2+\left(\frac{k_1}{\rho_2}\right)^2},}\nline \displaystyle{\mu^2_{2,n}={(n\pi)^2}\frac{k_1}{\rho_1}+\frac{k_1}{2\rho_2}-\frac{1}{2}\sqrt{\frac{4k_1^2}{\rho_1\rho_2}(n\pi)^2+\left(\frac{k_1}{\rho_2}\right)^2}.}
\end{array}\right.
\end{equation}
Using the asymptotic expansion in \eqref{Equation3.9-work4}, we get
\begin{equation}
\label{Equation3.10-work4}
\mu_{1,n}^2={(n\pi)^2}\frac{k_1}{\rho_1}+n\pi \sqrt{\frac{k_1^2}{\rho_1\rho_2}}+\frac{k_1}{2\rho_2}+O\left(n^{-1}\right),
\end{equation} 
\begin{equation}
\label{Equation3.11-work4}
\mu_{2,n}^2={(n\pi)^2}\frac{k_1}{\rho_1}-n\pi \sqrt{\frac{k_1^2}{\rho_1\rho_2}}+\frac{k_1}{2\rho_2}+O\left(n^{-1}\right).
\end{equation}
Using again  asymptotic expansion in \eqref{Equation3.10-work4}-\eqref{Equation3.11-work4}, we get \eqref{Equation3.5-work4}-\eqref{Equation3.6-work4}. Next, for $\lambda=\lambda_{1,n},$ setting
$$
C_{1,n}=\frac{1}{n\pi},\quad D_{1,n}=\frac{C_{1,n}(\frac{\rho_1}{k_1}\mu_{1,n}^2-(n\pi)^2)}{n\pi }=\frac{\sqrt{\frac{k_1}{k_2}}}{n\pi}+O\left(n^{-2}\right)
$$
in \eqref{Equation2.3new-work4}, we get the corresponding eigenfunctions \eqref{Equation3.7-work4}. Similarly for $\lambda=\lambda_{2,n},$ setting 
$$
D_{2,n}=\frac{1}{n\pi},\quad C_{2,n}=\frac{(n\pi)D_{2,n}}{\frac{\rho_1}{k_1}\mu^2_{2,n}-(n\pi)^2}=-\frac{\sqrt{\frac{k_2}{k_1}}}{n\pi}+O\left(n^{-2}\right)
$$
in \eqref{Equation2.3new-work4}, we get the corresponding eigenfunctions \eqref{Equation3.8-work4}.
\end{proof}
\begin{rk}\label{Remark3.4-work4}
\rm{ If $\frac{k_1}{\rho_1}=\frac{k_2}{\rho_2}$, then from Equation \eqref{Equation3.9-work4}, we can easily check that the eigenvalues $\lambda_{1,n}, \lambda_{2,n}$   are simple and different from zero. Then, we set the eigenfunctions of the operator ${\mathcal{A}_2}$ as
\begin{equation*}
E_{1,n}=(\varphi_{1,n},\lambda_{1,n}\varphi_{1,n},\psi_{1,n},\lambda_{1,n}\psi_{1,n}),\quad
E_{2,n}=(\varphi_{2,n},\lambda_{2,n}\varphi_{2,n},\psi_{2,n},\lambda_{2,n}\psi_{2,n}).
\end{equation*}
From the asymptotic expansions \eqref{Equation3.5-work4}-\eqref{Equation3.6-work4} and \eqref{Equation3.7-work4}-\eqref{Equation3.8-work4}, we can easily prove that $\big\{E_{1,n},$  $E_{2,n}\big\}_{n\in \mathbb{Z}^*}$ form a Riesz basis in the energy space ${\mathcal{H}_2}.$ We distinguish different types of observability inequalities, while depending on the constants $\rho_1,\ \rho_2,\ k_1,\ k_2.$ In fact,  we are going to see in Proposition \ref{Proposition3.5-work4} that if there exist no integers $k\in\mathbb{N}$ such that $\sqrt{\frac{k_1}{k_2}}=k\pi,$  then the eigenvalues  satisfy a uniform gap condition. In this case, we will apply the usual Ingham's Theorem (see Theorems 4.3, 9.2 in \cite{Komornik01}) in order to get observability inequalities hold in the energy space ${\mathcal{H}_2}$. In the case where there exists  $k_0\in\mathbb{N}^*$ such that $\sqrt{\frac{k_1}{k_2}}=k_0\pi,$ then the eigenvalues of the same branch satisfy a uniform gap condition, while on different branches they can be asymptotically close at a rate of order $\frac{1}{n^2}$ (see Proposition \ref{Proposition3.5-work4}). Thus, the usual Ingham's Theorem used in the case $\sqrt{\frac{k_1}{k_2}}\neq k\pi$ is no longer valid, and therefore we will use a general Ingham-type Theorem, which tolerates  asymptotically close eigenvalues (see Theorem 9.4 in \cite{Komornik01}).}\xqed{$\square$}
\end{rk}
\begin{prop}\label{Proposition3.5-work4}
\rm{Assume that $\frac{k_1}{\rho_1}=\frac{k_2}{\rho_2}$ and condition $(\rm A_1)$ holds, then there exist two constants  ${\gamma}>0,\ N_0>0$  depending only on the constants $\rho_1,\ \rho_2, \ k_1,\ k_2$ such that
\begin{equation}
\label{Equation3.12-work4}
\left|\lambda_{1,m}-\lambda_{1,n}\right|\geq  2\gamma,\quad \left|\lambda_{2,m}-\lambda_{2,n}\right|\geq  2\gamma
\end{equation}
and
\begin{equation}
\label{Equation3.13-work4}
\left|\lambda_{1,m}-\lambda_{2,n}\right|\geq  2\gamma,\quad\forall  |n|,\ |m|\leq N_0.
\end{equation}
Moreover, we have the following two cases:\\[0.1in]
\textbf{Case 1.} If  there exist  no integers $k\in \mathbb{N}$ such that $\sqrt{\frac{k_1}{k_2}}=k\pi,$ then there exists a constant $\tilde{\gamma}>0$ depending only on the constants $\rho_1,\ \rho_2, \ k_1,\ k_2$ such that the two branches of eigenvalues of ${\mathcal{A}_2}$ satisfy  a uniform gap condition
\begin{equation}
\label{Equation3.14-work4}
\tilde{\gamma}:=\inf\left|\lambda_{1,m}-\lambda_{2,n}\right|>0.
\end{equation}
\textbf{Case 2.} If there exists  $k_0 \in \mathbb{N}^*$ such that $\sqrt{\frac{k_1}{k_2}}=k_0\pi$, then there exist constants $c_2>c_1>0$ depending only on the constants $\rho_1,\ \rho_2, \ k_1,\ k_2$ such that for all $\left|m\right|, \ \left|n\right| \geq N_0$, for $N_0$ large enough,  we have 
\begin{equation}
\label{Equation3.15-work4}
\left|\lambda_{1,m}-\lambda_{2,n}\right|\geq\frac{c_1}{m^2} \quad \text{ and }\quad \left|\lambda_{1,m}-\lambda_{2,n}\right|\geq\frac{c_1}{n^2}, 
\end{equation}
and there exist infinitely many integers $m,\ n$ such that 
\begin{equation}
\label{Equation3.16-work4}
\left|\lambda_{1,m}-\lambda_{2,n}\right|\leq \frac{c_2}{m^2} \quad \text{and} \quad \left|\lambda_{1,m}-\lambda_{2,n}\right|\leq \frac{c_2}{n^2}.
\end{equation}}
\end{prop}
\begin{proof} First, from \eqref{Equation3.9-work4} and the fact that all the eigenvalues $\lambda_{1,n},\ \lambda_{2,n}$ are simple, it follows that \eqref{Equation3.12-work4} and \eqref{Equation3.13-work4}. We now  divide the proof into two cases: \\[0.1in]
\textbf{Case 1.} There exists  no integer $k\in \mathbb{Z}$ such that $\sqrt{\frac{k_1}{k_2}}=k\pi.$ First, from the asymptotic expansions \eqref{Equation3.5-work4}-\eqref{Equation3.6-work4}, we have
\begin{equation}
\label{Equation3.17-work4}
\left|\lambda_{1,m}-\lambda_{2,n}\right|
=\sqrt{\frac{k_1}{\rho_1}}\left|\pi(m-n)+\sqrt{\frac{k_1}{k_2}}\right|+O\left(|n|^{-1}\right)+ O\left(|m|^{-1}\right).
\end{equation}
Since there exists  no integer $k\in \mathbb{N}$ such that $\sqrt{\frac{k_1}{k_2}}=k\pi,$ then there exists $c>0$, such that
$$ \left|\pi(m-n)+\sqrt{\frac{k_1}{k_2}}\right|>c_1,\quad\forall\ n,\ m\in\mathbb{N}^*.$$
Therefore, from \eqref{Equation3.17-work4}, we get \eqref{Equation3.14-work4}.\\[0.1in]
\textbf{Case 2.}  There exists  $k_0 \in \mathbb{N}^*$ such that $\sqrt{\frac{k_1}{k_2}}=k_0\pi$. Again from the asymptotic expansions \eqref{Equation3.5-work4}-\eqref{Equation3.6-work4}, we have 
\begin{equation}
\label{Equation3.18-work4}
\left|\lambda_{1,m}-\lambda_{2,n}\right|=\sqrt{\frac{k_1}{\rho_1}}\pi \left|m-n+k_0 +\frac{k_0^2 }{8}\left(\frac{1}{m}-\frac{1 }{n}\right)+O\left(|n|^{-3}\right)+O\left(|m|^{-3}\right)\right|.
\end{equation}
 We distinguish two cases:
\begin{enumerate}
\item[1.] If $m-n+k_0\neq0$, then there exists $c_1>0$ such that
$$\left|m-n+k_0\right|>c_1,$$
and therefore from \eqref{Equation3.18-work4}, it follows that \eqref{Equation3.15-work4}.
\item[2.] If $m-n+k_0=0$, then from \eqref{Equation3.18-work4}, we obtain 
\begin{equation*}
\left|\lambda_{1,m}-\lambda_{2,n}\right|=\sqrt{\frac{k_1}{\rho_1}}\pi \left|\frac{k_0^3 }{8}\frac{1}{n(n-k_0)}+O\left(|n|^{-3}\right)\right|\geq \frac{c_1}{n^2}.
\end{equation*}
Consequently, we get \eqref{Equation3.15-work4}.
\end{enumerate}
Moreover, if $m-n+k_0=0$, then from the previous inequality there exists $c_2>0$  such that  \eqref{Equation3.16-work4} holds, which concludes the proof of the proposition.
\end{proof}
\begin{prop}\label{Proposition3.6-work4}
\rm{Assume that $\frac{\rho_1}{k_1}=\frac{\rho_2}{k_2}$ and condition $(\rm A_1)$ holds. If there exists  $k_0 \in \mathbb{N}^*$ such that $\sqrt{\frac{k_1}{k_2}}=k_0\pi$, then we adjust the  branches of eigenvalues into one sequence $\left(\lambda_n\right)_{n}$  such that  $\left(\text{Im}\lambda_n\right)_{n}$ is strictly  increasing. If 
\begin{equation}\label{Equation3.19-work4}
0<\text{Im}\lambda_{n+1}-\text{Im}\lambda_n\leq \gamma,
\end{equation}
then 
\begin{equation}\label{Equation3.20-work4}
\text{Im}\lambda_{n}-\text{Im}\lambda_{n-1}> \gamma\ \ \ \text{and}\ \ \ \text{Im}\lambda_{n+2}-\text{Im}\lambda_{n+1}>\gamma.
\end{equation}
We say that $\text{Im}\lambda_{n},\ \text{Im}\lambda_{n+1}$ is a chain of close exponents relative to $\gamma$ of length 2.}
\end{prop}
\begin{proof} When there exists  $k_0 \in \mathbb{N}^*$ such that $\sqrt{\frac{k_1}{k_2}}=k_0\pi$, since the eigenvalues of the same branch satisfy a uniform gap condition, then from  \eqref{Equation3.19-work4}, we deduce that $\lambda_n,\ \lambda_{n+1}$ belong to different branches.  If $\lambda_{n-1},\ \lambda_n$ belong to the same branch of eigenvalues, then from \eqref{Equation3.12-work4}, we get
\begin{equation*}
\text{Im}\lambda_{n}-\text{Im}\lambda_{n-1}>2\gamma.
\end{equation*}
Thus, we obtain the first assertion of \eqref{Equation3.20-work4}.  Otherwise, if
 $\lambda_{n-1},\ \lambda_n$ belong to different branches, then using the  fact that 
$\lambda_n,\ \lambda_{n+1}$ belong to different branches, then  $\lambda_{n-1},\ \lambda_{n+1}$ belong to the same branch of eigenvalues. In this case, from \eqref{Equation3.12-work4}, it follows that
\begin{equation}\label{Equation3.21-work4}
\text{Im}\lambda_{n+1}-\text{Im}\lambda_{n-1}>2\gamma.
\end{equation}
From \eqref{Equation3.19-work4} and \eqref{Equation3.21-work4}, we get
\begin{equation*}
\text{Im}\lambda_{n}-\text{Im}\lambda_{n-1}=\left(\text{Im}\lambda_{n+1}-\text{Im}\lambda_{n-1}\right)-\left(\text{Im}\lambda_{n+1}-\text{Im}\lambda_{n}\right)>2\gamma-\gamma=\gamma.
\end{equation*}
Therefore,  we obtain the first assertion of \eqref{Equation3.20-work4}.  The same argument verifies   the second assertion of \eqref{Equation3.20-work4}. Thus, the proof of the proposition  is complete.
\end{proof}$\\[0.1in]$
\noindent From Proposition \ref{Proposition3.5-work4}, it follows that 
\begin{equation}\label{Equation3.22-work4}
\left(\varphi_{1,n}\right)_x(1)+\psi_{1,n}(1)=O\left(1\right)\ \ \ \text{and}\ \ \ \left(\varphi_{2,n}\right)_x(1)+\psi_{2,n}(1)=O\left(1\right).
\end{equation}
On the other hand,  if  there exist  no integers $k\in \mathbb{N}$ such that $\sqrt{\frac{k_1}{k_2}}=k\pi,$ due to the fact that the eigenvalues  $\lambda_{1,m},\lambda_{2,n}$ satisfy a uniform gap condition, then the inverse observability inequality is  true in the energy space ${\mathcal{H}_2}.$ Otherwise, if there exists  integer $k_0 \in \mathbb{N}^*$ such that $\sqrt{\frac{k_1}{k_2}}=k_0\pi$,  due to the fact that the eigenvalues can be asymptotically close, then the inverse observability inequality is not true in the energy space ${\mathcal{H}_2}.$ For this reason,  we define the following weighted spectral space
\begin{equation}\label{Equation3.23-work4}
\mathcal{D}=\left\{ \Phi_0=\sum_{n\neq 0}(\alpha_{1,n}E_{1,n}+\alpha_{2,n} E_{2,n})n^{2}\ |\  \alpha_{1,n},\ \alpha_{2,n}\in\mathbb{C},\ \sum_{n\neq 0}\left(\left|\alpha_{1,n}\right|^2+\left|\alpha_{2,n}\right|^2\right)<\infty \right\}.
\end{equation}
Since the System $\left\{E_{1,n}, E_{2,n}\right\}_{n\in\mathbb{Z}^*}$ is a 
Riesz basis in the energy space ${\mathcal{H}_2},$ the space $\mathcal{D}$ is obviously a Hilbert space equipped with the  norm 
$$\sum_{n\neq 0}\left(\left|\alpha_{1,n}\right|^2+\left|\alpha_{2,n}\right|^2\right).$$
\noindent We are now ready to prove our observability inequalities results.\\[0.1in]
\textbf{Proof of Theorem \ref{Theorem3.2-work4}.} We divide the  proof into two main steps.\\ 
\noindent Our first aim is to prove the  direct inequality  \eqref{Equation3.2-work4}. Given any initial data $\Phi_0\in {\mathcal{H}_2},$ such as
  $$\Phi_0=(\varphi_{0},\varphi_{1},\psi_{0},\psi_{1})=\sum_{n\in\mathbb{Z}^*}\left({\alpha}_{1,n}{E}_{1,n},+ {\alpha}_{2,n} {E}_{2,n}\right),$$
 then the solution of \eqref{Equation2.2-work4} can be written as
$$\Phi\left(t\right)=(\varphi \left(t\right),\varphi_{t}\left(t\right), \psi \left(t\right),\psi_{t}\left(t\right))=\sum_{n\in\mathbb{Z}^*}\left(\alpha_{1,n}E_{1,n}e^{i\mu_{1,n}t}+\alpha_{2,n}E_{2,n}e^{i\mu_{2,n}t}\right).$$
 Therefore, we have
\begin{equation}\label{Equation3.24-work4}
\varphi_x\left(1,t\right)+\psi(1,t)=f_1\left(t\right) +f_2\left(t\right),
\end{equation}
where 
$$ f_1\left(t\right)=\sum_{n\in\mathbb{Z}^*}\alpha_{1,n}\left(\left(\varphi_{1,n}\right)_x (1)+\psi_{1,n}(1)\right) e^{i\mu_{1,n}t},\quad f_2\left(t\right)=\sum_{n\in\mathbb{Z}^*}\alpha_{2,n}\left(\left(\varphi_{2,n}\right)_x(1)+\psi_{2,n}(1)\right)e^{i\mu_{2,n}t}.$$
Since the eigenvalues of the same branch satisfy a uniform gap condition, applying   the usual Ingham's Theorem  (see Theorem  9.2 in \cite{Komornik01}), we get
\begin{equation}\label{Equation3.25-work4}
\left\{\begin{array}{ll}

  \displaystyle{\int_{0}^T|f_1\left(t\right)|^2dt\sim \sum_{n\in\mathbb{Z}^*}\left|\alpha_{1,n}\left(\left(\varphi_{1,n}\right)_x (1)+\psi_{1,n}(1)\right)\right|^2,}
  
               \nline
               
\displaystyle{ \int_{0}^T|f_2\left(t\right)|^2dt\sim\sum_{n\in\mathbb{Z}^*}\left|\alpha_{2,n}\left(\left(\varphi_{2,n}\right)_x (1)+\psi_{2,n}(1)\right)\right|^2}.
\end{array}\right.
\end{equation}
On the other hand, from \eqref{Equation3.24-work4}, we get
\begin{equation}\label{Equation3.26-work4}
\int_{0}^T|\varphi_x\left(1,t\right)+\psi(1,t)|^2dt\leq 2\int_{0}^T|f_1\left(t\right)|^2dt
+2\int_{0}^T|f_2\left(t\right)|^2dt.
\end{equation}
Inserting \eqref{Equation3.25-work4} in \eqref{Equation3.26-work4} and using \eqref{Equation3.22-work4}, we get
\begin{equation*}
\int_{0}^T\left|\varphi_x\left(1,t\right)+\psi(1,t)\right|^2dt\leq \ell_0 \sum_{n\in\mathbb{Z}^*}\left(\left|\alpha_{1,n}\right|^2 +\left|\alpha_{2,n}\right|^2\right).
\end{equation*}
Hence, we get the inequality \eqref{Equation3.2-work4}.\\[0.1in]
\noindent  Our next aim is to prove the inverse observability inequalities. We divide the proof into two cases:\\[0.1in]
 \noindent\textbf{Case 1.} There exists  no integer $k\in \mathbb{N}$ such that $\sqrt{\frac{k_1}{k_2}}=k\pi.$ Given any initial data $\Phi_0\in {\mathcal{H}_2},$ such as
  $$\Phi_0=(\varphi_{0},\varphi_{1},\psi_{0},\psi_{1})=\sum_{n\in\mathbb{Z}^*}\left({\alpha}_{1,n}{E}_{1,n}+ {\alpha}_{2,n} {E}_{2,n}\right),$$
then the solution of \eqref{Equation2.2-work4} can be written as
$$\Phi\left(t\right)=(\varphi \left(t\right),\varphi_{t}\left(t\right), \psi \left(t\right),\psi_{t}\left(t\right))=\sum_{n\in\mathbb{Z}^*}(\alpha_{1,n}E_{1,n}e^{i\mu_{1,n}t}+\alpha_{2,n}E_{2,n}e^{i\mu_{2,n}t}).$$
 Therefore, we have
\begin{equation}\label{Equation3.27-work4}
\varphi_x\left(1,t\right)+\psi(1,t)=\sum_{n\in\mathbb{Z}^*}\left(\alpha_{1,n}\left(\left(\varphi_{1,n}\right)_x (1)+ \psi_{1,n}(1)\right)e^{i\mu_{1,n}t} +\alpha_{2,n}\left(\left(\varphi_{2,n}\right)_x(1)+\psi_{2,n}(1)\right)e^{i\mu_{2,n}t}\right).
\end{equation}
From \eqref{Equation3.22-work4}, we can rewrite \eqref{Equation3.27-work4} as 
$$\varphi_{x}\left(1,t\right)+\psi(1,t)\sim \sum_{n\neq 0} \left(\alpha_{1,n}e^{i\mu_{1,n}t}+\alpha_{2,n}e^{i\mu_{2,n}t}\right).$$
Following a generalization of Ingham's Theorem (see Theorem  9.2 in \cite{Komornik01}), the sequence $\left\{e^{i \mu_{1,n} t}, e^{i \mu_{2,n} t}\right\}_{n\neq 0}$ forms a Riesz basis in $L^2(0,T)$ provided that $T>2\pi D^+$, where $D^{+}$ is the upper density of the sequence $(\lambda_n)_{n\in \mathbb{Z}^*},$ defined as 
$$D^{+}(r)=\lim_{n\rightarrow\infty} \frac{n^{+}(r)}{r},$$
 where $n^{+}(r)$ denotes the largest number of terms of the sequence $(\lambda_n)_{n\neq 0}$ contained
in an interval of length $r.$ To be more precise, 
$D^{+}=\frac{2}{\pi}\sqrt{\frac{\rho_1}{k_1}}.$
Therefore, 
\begin{equation*}
\label{ew62}
\int_{0}^{T}\left|\varphi_{x}\left(1,t\right)+\psi(1,t)\right|^2 dt\sim \sum_{n\neq 0}\left(\left|\alpha_{1,n}\right|^2+\left|\alpha_{2,n}\right|^2\right). 
\end{equation*}
Hence, we get \eqref{Equation3.3-work4}.\\[0.1in]
\textbf{Case 2.}  There exists  $k_0 \in \mathbb{N}^*$ such that $\sqrt{\frac{k_1}{k_2}}=k_0\pi$. Given any initial data $\Phi_0\in \mathcal{D},$ such as
  $$\Phi_0=(\varphi_{0},\varphi_{1},\psi_{0},\psi_{1})=\sum_{n\in\mathbb{Z}^*}\left({\alpha}_{1,n}{E}_{1,n}+ {\alpha}_{2,n} {E}_{2,n}\right)n^2,$$
consequently, the solution of \eqref{Equation2.2-work4} can be written as
$$\Phi\left(t\right)=(\varphi \left(t\right),\varphi_{t}\left(t\right), \psi \left(t\right),\psi_{t}\left(t\right))=\sum_{n\in\mathbb{Z}^*}\left(\alpha_{1,n}E_{1,n}e^{\lambda_{1,n}t}+\alpha_{2,n}E_{2,n}e^{\lambda_{2,n}t}\right).$$
 Therefore, we have
\begin{equation}\label{Equation3.28-work4}
 \varphi_x\left(1,t\right)+\psi(1,t)=\sum_{n\in\mathbb{Z}^*}\left(\alpha_{1,n}\left(\left(\varphi_{1,n}\right)_x (1) +\psi_{1,n}(1)\right)e^{\lambda_{1,n}t} +\alpha_{2,n}\left(\left(\varphi_{2,n}\right)_x(1)+\psi_{2,n}(1)\right)e^{\lambda_{2,n}t}\right)n^2.
 \end{equation}
We now arrange the two branches of eigenvalues  $(\lambda_{1,n})_{n\neq0},(\lambda_{2,n})_{n\neq 0}$ into one sequence $(\lambda_{n})_{n\neq 0}$ such that the sequence $(\text{Im}\lambda_{n})_{n\neq 0}$ is strictly increasing. From Proposition \ref{Proposition3.6-work4}, all the chain $ \text{Im}\lambda_{n}, \ \text{Im}\lambda_{n+1}$ of close exponents relative to $\gamma$ is of length 2. Moreover, we denote by $a_{n}$ the coefficient before  $e^{\lambda_{1,n}t}$ or 
$e^{\lambda_{2,n}t}$ in \eqref{Equation3.28-work4}.
Let $A$ and $B$ be defined as
$$A=\left\{n\in \mathbb{Z}^*\text{ such that condition \eqref{Equation3.19-work4} holds}\right\}\quad\text{and}\quad B=\mathbb{Z}^*\setminus \left\{n,n+1\text{ such that } n\in A\right\}.$$
Then,   we can rewrite \eqref{Equation3.28-work4} as 
\begin{equation*}
\begin{array}{ll}
\displaystyle{ \varphi_{x}\left(1,t\right)+\psi(1,t)}&=\displaystyle{\sum_{n\in B} a_{n}e ^{\lambda_{n}t}+\sum_{n\in A} a_{n}e ^{\lambda_{n}t}+a_{n+1}e ^{\lambda_{n+1}t}}\nline

&=\displaystyle{ \sum_{n\in B} a_{n}e ^{\lambda_{n}t}+\sum_{n\in A} \left(\left(a_{n}+a_{n+1}\right)e ^{\lambda_{n}t}+\left(\lambda_{n+1}-\lambda_{n}\right)a_{n+1}e_{n+1}\left(t\right)\right),}

\end{array}
\end{equation*}
where $e_{n+1}\left(t\right)$ denotes the divided difference of the chain of close exponents $\lambda_{n}, \lambda_{n+1}$ relative to $\gamma$
$$
e_{n+1}\left(t\right)=\frac{e^{\lambda_{n+1}t}-e^{\lambda_{n}t}}{\lambda_{n+1}-\lambda_{n}}.
$$
It follows from Theorem  9.4 in \cite{Komornik01}, that the sequence
$$\left(e^{\lambda_{n}t}\right)_{n\in B},\left(e^{\lambda_{n}t},e_{n+1}\left(t\right)\right)_{n\in A}$$
forms a Riesz sequence in $L^2(0,T)$ provided that $T>2\pi D^{+}=4\sqrt{\frac{\rho_1}{k_1}}.$ Thus, we have
 \begin{equation}
\label{Equation3.29-work4}
\int_{0}^{T}\left|\varphi_{x}\left(1,t\right)+\psi(1,t)\right|^2 dt \sim \sum_{n\in B} \left|a_{n}\right|^2+\sum_{n\in A}\left(\left|a_{n}+a_{n+1}\right|^2+\left|\lambda_{n+1}-\lambda_{n}\right|^2
\left|a_{n+1}\right|^2\right).
\end{equation}
On the other hand, from \eqref{Equation3.16-work4}, we get
\begin{equation}
\label{Equation3.30-work4}
\left|a_{n}+a_{n+1}\right|^2+\left|\lambda_{n+1}-\lambda_{n}\right|^2
\left|a_{n+1}\right|^2\geq \tilde{\ell}_1\left(\frac{\left|a_n\right|^2}{\left|n\right|^4}+\frac{\left|a_{n+1}\right|^2}{\left|n+1\right|^4}\right).
\end{equation}
Inserting \eqref{Equation3.30-work4} into \eqref{Equation3.29-work4} and returning to the previous notations, we get
$$
\int_{0}^{T}\left|\varphi_{x}\left(1,t\right)+\psi(1,t)\right|^2 dt\geq  \sum_{n\in B}\frac{\left|a_n\right|^2}{\left|n\right|^4}+\tilde{\ell}_1\sum_{n\in A}\left(\frac{\left|a_n\right|^2}{\left|n\right|^4}+\frac{\left|a_{n+1}\right|^2}{\left|n+1\right|^4}
\right)\geq \ell_1 \sum_{n\in \mathbb{Z}^*}\frac{\left|a_n\right|^2}{\left|n\right|^4}.
$$
Therefore, we have
\begin{equation}
\label{Equation3.31-work4}
\int_{0}^{T}\left|\varphi_{x}\left(1,t\right)+\psi(1,t)\right|^2 dt\geq \ell_1 \sum_{n\neq 0}\left(\left|\alpha_{1,n}\left(\varphi_{1,n}\right)_x(1)+\psi_{1,n}(1)\right|^2+\left|\alpha_{2,n}\left(\left(\varphi_{2,n}\right)_x(1)+\psi_{2,n}(1)\right)\right|^2\right).
\end{equation}
Inserting \eqref{Equation3.22-work4} in \eqref{Equation3.31-work4}, we get
\begin{equation*}
\int_{0}^T\left|\varphi_x\left(1,t\right)+\psi(1,t)\right|^2dt\geq \ell_1 \sum_{n\in\mathbb{Z}^*}\left(\left|\alpha_{1,n}\right|^2 +\left|\alpha_{2,n}\right|^2\right).
\end{equation*}
Consequently, we obtain  the inequality \eqref{Equation3.4-work4}. Thus, the proof of the theorem  is complete. \xqed{$\square$}
\begin{rk}\label{Remark3.7-work4}
\rm{It is very important to ask the question about the optimality of the  observability inequality  \eqref{Equation3.4-work4}. In our opinion, from inequity   \eqref{Equation3.16-work4}, we may find a dense subspace ${\mathtt{D}}$ of $\mathcal{D}$  such that 
$$\int_{0}^T\left|\varphi_{x}\left(1,t\right)+\psi(1,t)\right|^2 dt \sim \left\|(\varphi_{0},\varphi_{1},\psi_{0},\psi_{1})\right\|^2_{\mathcal{D}},$$
for all $(\varphi_{0},\varphi_{1},\psi_{0},\psi_{1})\in {\mathtt{D}}$. }\xqed{$\square$}
\end{rk}

 In the case where there exist  no integers $k\in \mathbb{N}$ such that $\sqrt{\frac{k_1}{k_2}}=k\pi,$ the inverse observability inequality is  true in the energy space ${\mathcal{H}_2}.$ Otherwise, in the case where there exists  $k_0 \in \mathbb{N}^*$ such that $\sqrt{\frac{k_1}{k_2}}=k_0\pi$,   the inverse observability inequality  holds in weighted spectral space $\mathcal{D}$. The aim of this part is to get the observability or exact controllability in usual functional spaces.\\[0.1in]
\textbf{Observability inequality  in usual spaces.}
For the observability inequality,  we first recall Theorem 3.1 stated  in  \cite{RaoLiu02}.
\begin{thm}\label{Theoremk3.8-work4}
\label{ewtheoliurao}
\rm{Let $(x_{n})_{n\neq 0}$ and $(y_{n})_{n\neq 0}$ be Riesz basis of Hilbert spaces $X$ and $Y$ respectively, and $(f_{n})_{n\neq 0}$ and $(g_{n})_{n\neq 0}$ be Bessel sequences of $X$ and $Y$ with suitably small bounds respectively. Define
$$D=\left\{ (x,y)=\alpha_{n}(x_{n},g_{n})+\beta_{n}(f_{n},y_{n}) \ | \ \sum_{n \neq 0}\left(\left|\alpha_{n}\right|^2+\left|\beta_{n}\right|^2\right)< \infty\right\}.$$
Then, we have $D=X\times Y.$}\xqed{$\square$}
\end{thm}
\noindent Using the asymptotic expansions \eqref{Equation3.7-work4}-\eqref{Equation3.8-work4}, we have
$$E_{1,n}=(x_n,g_n),\quad E_{2,n}=(f_n,y_n),$$
with
 \begin{equation*}
\displaystyle{x_n=\left(\frac{\sin(n\pi x)}{n\pi},i\sqrt{\frac{k_1}{\rho_1}} \sin(n\pi x)\right),\quad
g_n=\left(\sqrt{\frac{k_1}{k_2}} \frac{\cos(n\pi x)}{n \pi},i\sqrt{\frac{k_1}{\rho_2}}\cos(n\pi x)\right)},
\end{equation*}
and
 \begin{equation*}
\displaystyle{y_n=\left(-\sqrt{\frac{k_2}{k_1}} \frac{\sin(n\pi x)}{n \pi}, -i\sqrt{\frac{k_2}{\rho_1}} {\sin(n\pi x)}\right),\quad
f_n=\left(\frac{\cos(n\pi x)}{n\pi}, i\sqrt{\frac{k_1}{\rho_1}}\cos(n\pi x)\right).}
\end{equation*}
For any $s\geq 0,$ we define the space
\begin{equation*}
X_{s}=\left\{(\hat{\phi},\hat{\psi})=\sum_{n\neq 0} \beta_n n^{s} x_n\right\}, \quad \left\|(\hat{\phi},\hat{\psi})\right\|^2_{X_{s}}=\sum_{n\neq 0} \left|\beta_n\right|^2.
\end{equation*}
According to Theorem \ref{ewtheoliurao}, we can state the following result.
\begin{cor}\label{Corollary3.9-work4}
\rm{Assume that $\frac{\rho_1}{k_1}=\frac{\rho_2}{k_2}$,   $\sqrt{\frac{k_1}{k_2}}=k_0\pi$  $(k_0 \in \mathbb{N}^*)$ and condition $(\rm A_1)$ holds. Then, we have the following identifications
\begin{equation}\label{eq 43vm}
D_2=X_{2}\times X_{2}.
\end{equation}}
\end{cor}
\begin{proof} We see that $(n^2x_n)_{n\neq 0}$ and $(n^2y_n)_{n\neq 0}$ are Riesz basis in $X_2\times X_2$. On the other hand, we have    $(n^2 f_n)_{n\neq 0}$ and $( n^2 g_n)_{n\neq 0}$ are Bessel sequences in $X_2\times X_2$. Then, \eqref{eq 43vm} follows directly from Theorem \ref{Theoremk3.8-work4}. 
\end{proof}\\[0.1in]
\noindent Furthermore, for any $s\geq 0,$ we define 
\begin{equation*}
V_{s}=\left\{ f=\sum_{n>0}\alpha_n\frac{\sin(n\pi x)}{n^{s}}\right\}, \quad \quad \left\|f\right\|^2_{V_{s}}=\sum_{n>0}\left|\alpha_n\right|^2.
\end{equation*}
Thus, with the pivot space $L^2\left(0,1\right),$ we have
$$X_{2}=V'_1\times V'_2.$$
Then, it follows  that
$$\mathcal{D}=V'_{1}\times V'_2\times V'_{1}\times V'_2.$$
Consequently, we have the following observability result. 
\begin{thm}\label{Theorem3.10-work4}
\rm{Assume that $\frac{\rho_1}{k_1}=\frac{\rho_2}{k_2}$,   $\sqrt{\frac{k_1}{k_2}}=k_0\pi$  $(k_0 \in \mathbb{N}^*)$ and condition $(\rm A_1)$ holds. Let  $T>4\sqrt{\frac{\rho_1}{k_1}}$, then 
there exists a constant $c_3>0$ such that the following direct inequality holds:
\begin{equation*}
\int_0^T\left|\varphi_{x}\left(1,t\right)+\psi(1,t)\right|^2 dt \leq c_3\left\|(\varphi_0,\varphi_1,\psi_0,\psi_1)\right\|^2_{{\mathcal{H}_2}},
\end{equation*}
 for all solution $\Phi=(\varphi,\varphi_t,\psi,\psi_t)$ that solve  the  homogeneous Cauchy problem \eqref{Equation2.1-work4}. Moreover, there exists a constant $0<c_4<c_3$, such that the following inverse observability inequality holds: 
\begin{equation*}
c_4\left\|(\varphi_{0},\varphi_{1},\psi_{0},\psi_{1})\right\|^2_{V'_{1}\times V'_2\times V'_{1}\times V'_2}\leq \int_{0}^T\left|\varphi_{x}\left(1,t\right)+\psi(1,t)\right|^2 dt.
\end{equation*}}
\end{thm}
\begin{rk}\label{Remark3.11-work4}
\rm{Assume that $\frac{\rho_1}{k_1}=\frac{\rho_2}{k_2}$,  if $\sqrt{\frac{k_1}{k_2}}=k_0\pi$  $(k_0 \in \mathbb{N}^*)$, then the two branches of eigenvalues are close in the order $\frac{1}{n^2}$. Due to the closeness of the eigenvalues, the observability space losses two derivatives. Consequently, the observability holds in the space of type
$$\mathcal{D}=(H^{-1}\left(0,1\right)\times H^{-2}\left(0,1\right))^2.$$
Moreover, the control space are of type
 $$(H^{2}\left(0,1\right)\times H^{1}\left(0,1\right))^2$$
 with suitable boundary conditions.}\xqed{$\square$}
\end{rk}
\noindent\textbf{Exact controllability  in usual spaces.}
In this part, using HUM mrthod, we establish exact controllability result for the System  \eqref{Equation1.1-work4}. Our main result in this part is the following theorem.
\begin{thm}\label{Theorem3.1-work4} 
\rm{
Assume that $\frac{k_1}{\rho_1}=\frac{k_2}{\rho_2}$, $T>4\sqrt{ \frac{\rho_1}{k_1} }$, and condition $(\rm A_1)$ holds:\\[0.1in]
 \noindent\textbf{Case 1.} There exists  no integer $k\in \mathbb{N}$ such that $\sqrt{\frac{k_1}{k_2}}=k\pi$. Let
$$\left(u_0,u_1,y_0,y_1\right)\in  \mathcal{X} , $$
then there exists $v\in L^2(0,T)$ such that the solution of the System \eqref{Equation1.1-work4}   satisfies the null final conditions
\begin{equation*}
u(x,T)=u_t(x,T)=y(x,T)=y_t(x,T)=0,\quad \forall x\in\left(0,1\right).
\end{equation*}
 \noindent\textbf{Case 2.}  There exists  $k_0 \in \mathbb{N}^*$ such that $\sqrt{\frac{k_1}{k_2}}=k_0\pi$. Let
 $$\left(u_0,u_1,y_0,y_1\right)\in  V_2\times V_1\times V_2\times V_1,$$
then there exists $v\in L^2(0,T)$ such that the solution of the  System \eqref{Equation1.1-work4}   satisfies the null final conditions
\begin{equation*}
u(x,T)=u_t(x,T)=y(x,T)=y_t(x,T)=0,\quad \forall x\in\left(0,1\right).
\end{equation*}
}
\end{thm} 
{\noindent For the proof of Theorem  \ref{Theorem3.1-work4}, first,  we will prove that the System \eqref{Equation1.1-work4}  admits a  unique solution. For this aim, let $(\varphi,\varphi_t,\psi,\psi_t)$ be a solution of \eqref{Equation2.1-work4}  and let $v\in L^2(0,T)$. After 
 multiplying first and second equation of \eqref{Equation1.1-work4}  by 
$\overline{\varphi}$ and  $\overline{\psi}$, respectively, and integrating their sum over $\left(0,1\right)\times(0,t)$ (where $t\in[0,T]$), we get 
\begin{equation}\label{s123.6}
\begin{array}{l}
\displaystyle{
\rho_1\int_{0}^1u_t(x,t)\overline{\varphi}(x,t)dx+\rho_2\int_{0}^1y_t(x,t)\overline{\psi}(x,t) dx-
\rho_1\int_{0}^1u(x,t)\overline{\varphi_t}(x,t) dx}

\\ \\

\displaystyle{- \rho_2 \int_{0}^1y(x,t)\overline{\psi_t}(x,t) dx=
\rho_1\int_{0}^1u_t(x,0)\overline{\varphi}(x,0)dx+\rho_2\int_{0}^1y_t(x,0)\overline{\psi}(x,0)dx}

\\ \\

\displaystyle{ -\rho_1\int_{0}^1u(x,0)\overline{\varphi_t}(x,0)dx- \rho_2\int_{0}^1y(x,0)\overline{\psi_t}(x,0)dx- k_1 \int_0^t\overline{\left( \varphi_{x}(1,t)+\psi(1,t)\right)}v(t)  dt.}
\end{array}
\end{equation}
We introduce  the linear form $L_t$ by
\begin{equation*}
L_t(\Phi_0)=\left< (\rho_1u_1,-\rho_1 u_0,\rho_2 y_1,-\rho_2 y_0) , \Phi_0 \right>_{{\mathcal{H}_2}',{\mathcal{H}_2}}
 -k_1 \int_0^t \overline{\left( \varphi_{x}(1,t)+\psi(1,t)\right)}v(t)  dt.
\end{equation*}
From \eqref{s123.6}, we obtain a weak formulation of the System \eqref{Equation1.1-work4}  
\begin{equation}\label{0ecl2.6}
\left<\left(\rho_1  u_t(x,t),-\rho_1 u(x,t),\rho_2 y_t(x,t),-\rho_2 y(x,t)\right), \Phi(t) \right>_{{\mathcal{H}_2}',{\mathcal{H}_2}} =L_t(\Phi_0),\quad\forall\ \Phi_0 \in{\mathcal{H}_2}.
\end{equation}
\begin{thm}\label{Th56}
\rm{Assume that $\frac{\rho_1}{k_1}=\frac{\rho_2}{k_2}$,   $\sqrt{\frac{k_1}{k_2}}=k_0\pi$  $(k_0 \in \mathbb{N}^*)$, $T>4\sqrt{\frac{\rho_1}{k_1}}$ and condition $(\rm A_1)$ holds.    Then, for all initial data 
$U_0=(u_0,u_1,y_0,y_1)\,  \in \mathcal{X}:=L^2\left(0,1\right)\times H^{-1}\left(0,1\right)\times L^2\left(0,1\right)\times \left(H^{1}_*\left(0,1\right)\right)'$ and for all  $v\in L^2(0,T)$, the System \eqref{Equation1.1-work4}  admits a unique weak solution
$$U(x,t) =(u,u_t,y,y_t)\in
 C^0 \left( [ 0 , T ];\mathcal{X}\right)$$
 in the sense that the variational equation \eqref{0ecl2.6} is satisfied for all  $\Phi_0\in{\mathcal{H}_2}$ on the interval $[0, T].$
Moreover, the linear mapping $(  U_0 , v ) \longrightarrow  U$ is continuous  from
$$\mathcal{X}\times L^2(0,T)\text{ into } \mathcal{X}.$$}
\end{thm}
\begin{proof}
 For  every fixed $T>4\sqrt{\frac{\rho_1}{k_1}}$, using the direct observability inequality \eqref {Equation3.2-work4}, we deduce that 
\begin{equation*}
\displaystyle \left\| L_t\right\|_{ \mathcal{L}({\mathcal{H}_2},\mathbb{R})}\leq c\left( \left\| v\right\|_{L^2(0,T)}+\left\|U_0\right\|_{\mathcal{X}}\right),\quad \forall\ t\in [0,T],
\end{equation*}
where $c>0$. Hence,  the linear form $L_t$ is bounded on ${\mathcal{H}_2}$. Furthermore, it follows from Theorem \ref{Theorem2.1-work4}  that the linear map
$$\Phi\left(t\right) \longrightarrow \Phi_0$$
is an isomorphism from ${\mathcal{H}_2}$ onto itself. Therefore, the linear form 
$$ \Phi\left(t\right) \longrightarrow L_t\left(\Phi_0\right)$$
is also bounded on ${\mathcal{H}_2}.$ Using Riesz representation Theorem,  for each $0\leq t\leq T,$ there exists a unique element 
$\left(\rho_1 u_t(t),-\rho_1 u(t),\rho_2 y_t(t),-\rho_2 y(t)\right)\in   {\mathcal{H}_2}'$, with is a solution of 
\begin{equation}\label{bnn}
L_t(\Phi_0)= \left<\left(\rho_1 u_t(t),-\rho_1 u(t),\rho_2 y_t(t),-\rho_2 y(t)\right),\Phi\left(t\right)\right>_{{\mathcal{H}_2}',
{\mathcal{H}_2}}.
\end{equation}
putting $$U(t):=\left(u(t),u_t(t),y(t),y_t(t)\right).$$ 
From \eqref{bnn}, we deduce that $U (t)$ satisfy the problem \eqref{0ecl2.6}  for all $0\leq t\leq T$.  In addition, for all $0\leq t\leq T$, we have
\begin{equation*}
\begin{array}{ll}
\displaystyle{\left\|U(t)\right\|_{\mathcal{X}}}\simeq
\displaystyle{\left\|\left(\rho_1u_t(t),-\rho_1 u(t),\rho_2 y_t(t),-\rho_2y(t)\right)\right\|_{{\mathcal{H}_2}'}}\\ \\
 \hspace{4.2cm}=\displaystyle{\left\| L_T\right\|_{ \mathcal{L}({\mathcal{H}_2},\mathbb{R})}\leq c\left( \left\| v\right\|_{L^2(0,T)}+\left\|U_0\right\|_{\mathcal{X}}\right),}
\end{array}
\end{equation*}
which implies the continuity of the linear mapping, which concludes the  proof of the theorem.
\end{proof}$\\[0.1in]$
We now turn to the proof of Theorem \ref{Theorem3.1-work4}.\\[0.1in]
\noindent \textbf{Proof of Theorem \ref{Theorem3.1-work4}.}
 Assume that $\frac{k_1}{\rho_1}=\frac{k_2}{\rho_2}$, $T>4\sqrt{ \frac{\rho_1}{k_1} }$, and condition $(\rm A_1)$ holds. We divide the proof into two cases: Case 1, if $\sqrt{\frac{k_1}{k_2}}\neq k\pi$ (for all $k\in\mathbb{N}$), Case 2  if $\sqrt{\frac{k_1}{k_2}}=k_0\pi$ $(k_0\in\mathbb{N}^*)$.  Since the argument of two cases is entirely similar,  we will only provide one of them.\\[0.1in]
Suppose that $\sqrt{\frac{k_1}{k_2}}=k_0\pi$ $(k_0\in\mathbb{N}^*)$. Let $\Phi_0= (\varphi_0,\varphi_1,\psi_0,\psi_1)\in {\mathcal{H}_2}$ and $\Phi= (\varphi,\varphi_t,\psi,\psi_t)\in {\mathcal{H}_2}$ be the  solution of the problem \eqref{Equation2.1-work4}. Thanks to the inverse inequality \eqref{Equation3.4-work4}, we can define a norm on ${\mathcal{H}_2}$ by
\begin{equation}
\Vert \Phi_0 \Vert_\mathcal{F}=\left(\int_0^T\vert \varphi_{x}(1,t)+\psi(1,t)\vert^2dt\right)^{1/2}.\label{1e355'2}
\end{equation}
We denote by $\mathcal{F}$ the completion of ${\mathcal{H}_2}$ by this norm. It is clear that $\mathcal{F}$ is a Hilbert space. Thanks to the direct and the inverse observability inequalities, we have the following continuous and dense embeddings
$${\mathcal{H}_2}\subset \mathcal{F}\subset \mathcal{D}=V'_{1}\times V'_2\times V'_{1}\times V'_2 .$$
By choosing the control $v(t)=\varphi_x(1,t)+\psi(1,t)$, we will solve the backward problem \eqref{s12.6}
\begin{equation}\left\lbrace
\begin{array}{ll}
\rho_1\chi_{tt}-k_1(\chi_{x} + \zeta)_x =  0, &\text{ in } (0,1)\times \mathbb{R}_+,\\\rho_2\zeta_{tt}-k_2 \zeta_{xx} +k_1 (\chi_{x} + \zeta)  =  0, &\text{ in } (0,1)\times \mathbb{R}_+,\\ 
\chi(0,t)= \zeta_x(0,t)  = \zeta_x(1,t)=0, & \text{ in } \mathbb{R}_+,\\ 
\chi(1,t)=-\varphi_{x}(1,t)-\psi(1,t), & \text{ in } \mathbb{R}_+,\\ 
\chi(x,T)=\chi_t(x,T)=\zeta(x,T)=\zeta_t(x,T)=0, & \text{ in }(0,1). 
\label{s12.6}
\end{array}
\right.
\end{equation}
Using Theorem \ref{Th56}, the backward problem \eqref{s12.6} admits a unique weak solution $$(\chi,\chi_t,\zeta,\zeta_t)\in C^0\left([0,T];(L^{2}(0,1)\times H^{-1}(0,1))^2\right).$$
We define the operator: $\Lambda : {\mathcal{H}_2}\xrightarrow{\hspace*{1cm}} {\mathcal{H}_2}' $ 
$$\ \ \ \ \ \Lambda (\varphi_0,\varphi_1,\psi_0,\psi_1)=(\rho_1\chi_t(0),-\rho_1\chi(0),\rho_2\zeta_t(0),-\rho_2\zeta(0)).$$
From \eqref{0ecl2.6} and \eqref{s12.6}, it follows that 
\begin{equation}
\left<\Lambda \Phi_0,\tilde{\Phi}_0 \right>_{{\mathcal{H}_2}',{\mathcal{H}_2}}=\int_0^T \left(\varphi_{x}(1,t)+\psi(1,t)\right) \left(\tilde{\varphi}_{x}(1,t)+\tilde{\psi}(1,t)\right) dt=\left<\Phi_0,\tilde{\Phi}_0 \right>_\mathcal{F}, \label{1e3594}
\end{equation}
$\forall\ \Phi_0,\ \tilde{\Phi}_0\in{\mathcal{H}_2},$ where $\left(\cdot,\cdot\right)_\mathcal{F}$ denotes the scalar product associated with the norm $\left\|\cdot\right\|_\mathcal{F}$. Therefore, we have
\begin{equation}
\left|\left<\Lambda\Phi_0,\tilde{\Phi}_0\right>_{{\mathcal{H}_2}',{\mathcal{H}_2}}\right|\leq\Vert \Phi_0\Vert_\mathcal{F}\Vert\tilde{\Phi}_0\Vert_\mathcal{F}\quad\forall \ \Phi_0,\  \tilde{\Phi}_0\in {\mathcal{H}_2}.
\label{1e3595}
\end{equation}
Since ${\mathcal{H}_2}$ is dense in $\mathcal{F}$, the mapping $\Lambda$ can be extended to a continuous mapping from $\mathcal{F}$ into $\mathcal{F}'$. In particular, we have
\begin{equation}
\left|\left<\Lambda\Phi_0,\tilde{\Phi}_0\right>_{\mathcal{F}',\mathcal{F}}\right|\leq\Vert \Phi_0\Vert_\mathcal{F}\Vert\tilde{\Phi}_0\Vert_\mathcal{F}\quad\forall \ \Phi_0,\  \tilde{\Phi}_0\in \mathcal{F}
\label{new101}
\end{equation}
and
\begin{equation}
\langle\Lambda \Phi_0,{\Phi}_0\rangle_{\mathcal{F}', \mathcal{F}}=\left\|\Phi_0\right\|^2_\mathcal{F}\quad\forall\  \Phi_0\in\mathcal{F}.
\label{1e3596}
\end{equation}
Therefore, the bilinear form
$$\left(\Phi_0,\tilde{\Phi}_0\right) \rightarrow\left<\Lambda\Phi_0,\tilde{\Phi}_0\right>_{\mathcal{F}',\mathcal{F}}$$
is continuous and coercive on $\mathcal{F}\times \mathcal{F}$. Thanks to the Lax-Milgram Theorem, we deduce that $\Lambda$ is an isomorphism from $\mathcal{F}$ onto $\mathcal{F}'$. In particular, for all  $(u_1,u_0,y_1,y_0)\in V_{1}\times V_2\times V_{1}\times V_2\subset \mathcal{F}'$, there exists a unique element $(\varphi_0,\varphi_1,\psi_0,\psi_1)\in \mathcal{F}$ such that
$$\Lambda(\varphi_0,\varphi_1,\psi_0,\psi_1):=(\rho_1\chi_t(0),-\rho_1\chi(0),\rho_2\zeta_t(0),-\rho_2\zeta(0))=(u_1,u_0,y_1,y_0).$$
According to the uniqueness  of the solution of the problem \eqref{s12.6}, we get
$$u(x,T)=u_{t}(x,T)=y(x,T)=y_{t}(x,T)=0.$$
\noindent {  Thus, the proof of the  theorem is complete.}
\xqed{$\square$}}
\subsection{Observability and exact controllability when the speeds of  propagation are different}\label{Section4-work4} 
\noindent In this part, we study the exact controllability of a one dimensional    Timoshenko System \eqref{Equation1.1-work4}   in the case when the speeds of  propagation are different; i.e, $\frac{k_1}{\rho_1}\neq \frac{k_2}{\rho_2}$. 
\noindent Similar to subsection \ref{Section3-work4}, we use the spectrum method. For this aim, we need to study the asymptotic behavior of the spectrum of ${\mathcal{A}_2}$. We prove the following proposition.
\begin{prop}\label{Proposition4.3-work4}
\rm{Assume that $\frac{k_1}{\rho_1}\neq \frac{k_2}{\rho_2}$ and condition $(\rm A_1)$ holds. Then, the  eigenvalues of ${\mathcal{A}_2}$ asymptotic behavior
\begin{eqnarray}
\lambda_{1,n}= in\pi\sqrt{\frac{k_1}{\rho_1}}+\frac{i\frac{k_1}{\rho_2}\sqrt{\frac{k_1}{\rho_1}}}{2\left(\frac{k_1}{\rho_1}-\frac{k_2}{\rho_2}\right)\pi n}+O\left(n^{-3}\right),\label{Equation4.3-work4}\nline
\lambda_{2,n}=in\pi\sqrt{\frac{k_2}{\rho_2}}-\frac{i\frac{k_1}{\rho_2}\sqrt{\frac{k_2}{\rho_2}}}{2\left(\frac{k_1}{\rho_1}-\frac{k_2}{\rho_2}\right)\pi n}+O\left(n^{-3}\right),\label{Equation4.4-work4}
\end{eqnarray}
with the corresponding eigenfunctions
\begin{eqnarray}
\varphi_{1,n}(x)=\frac{\sin(n\pi x)}{n\pi}, \quad \quad \psi_{1,n}(x)=\frac{\frac{k_1}{\rho_2}}{\left(\frac{k_1}{\rho_1}-\frac{k_2}{\rho_2}\right)\pi^2 n^2}\cos(n\pi x)+O\left(n^{-4}\right)\cos(n\pi x),\label{Equation4.5-work4}\\
\varphi_{2,n}(x)=-\frac{\frac{k_1}{\rho_1}}{\left(\frac{k_1}{\rho_1}-\frac{k_2}{\rho_2}\right)\pi^2 n^2}\sin(n\pi x)+O\left(n^{-4}\right)\sin(n\pi x), 
\quad\psi_{2,n}(x)=\frac{\cos(n\pi x)}{n\pi}.\label{Equation4.6-work4}
\end{eqnarray}}
\end{prop}
\begin{proof} First, by solving  \eqref{Equation2.6-work4} and using the fact that  $\frac{k_1}{\rho_1}\neq \frac{k_2}{\rho_2}$, we get
\begin{equation}\label{Equation4.7-work4}
\left\{
\begin{array}{ll}
\displaystyle{\mu^2_{1,n}=\frac{(n\pi)^2}{2}\left(\frac{k_1}{\rho_1}+\frac{k_2}{\rho_2}\right)+\frac{k_1}{2\rho_2}+ \frac{\left(\frac{k_1}{\rho_1}-\frac{k_2}{\rho_2}\right)}{2}\sqrt{(n\pi)^4+\frac{2\left(\frac{k_1}{\rho_1}+\frac{k_2}{\rho_2}\right)\frac{k_1}{\rho_2}(n\pi)^2+\left(\frac{k_1}{\rho_2}\right)^2}{\left(\frac{k_1}{\rho_1}-\frac{k_2}{\rho_2}\right)^2}}},\\ \\

\displaystyle{\mu^2_{2,n}=\frac{(n\pi)^2}{2}\left(\frac{k_1}{\rho_1}+\frac{k_2}{\rho_2}\right)+\frac{k_1}{2\rho_2}- \frac{\left(\frac{k_1}{\rho_1}-\frac{k_2}{\rho_2}\right)}{2}\sqrt{(n\pi)^4+\frac{2\left(\frac{k_1}{\rho_1}+\frac{k_2}{\rho_2}\right)\frac{k_1}{\rho_2}(n\pi)^2+\left(\frac{k_1}{\rho_2}\right)^2}{\left(\frac{k_1}{\rho_1}-\frac{k_2}{\rho_2}\right)^2}}}.
\end{array}\right.
\end{equation}
Using the asymptotic expansion in \eqref{Equation4.7-work4}, we get
\begin{eqnarray}
\mu_{1,n}^2=\frac{k_1}{\rho_1}(n\pi)^2+\frac{\frac{k_1}{\rho_1\rho_2}}{\frac{k_1}{\rho_1}-\frac{k_2}{\rho_2}}
+O\left(n^{-2}\right),\label{Equation4.8-work4}
\nline
\mu_{2,n}^2=\frac{k_2}{\rho_2}(n\pi)^2-\frac{\frac{k_1 k_2}{\rho_2^2}}{\frac{k_1}{\rho_1}-\frac{k_2}{\rho_2}}
+O\left(n^{-2}\right).\label{Equation4.9-work4}
\end{eqnarray}
Using again  asymptotic expansion in \eqref{Equation4.8-work4}-\eqref{Equation4.9-work4}, we get \eqref{Equation4.3-work4}-\eqref{Equation4.4-work4}. Next, for $\lambda=\lambda_{1,n},$ setting
$$
C_{1,n}=\frac{1}{n\pi},\quad D_{1,n}=\frac{C_{1,n}(\frac{\rho_1}{k_1}\mu_{1,n}^2-(n\pi)^2)}{n\pi }=\frac{\frac{k_1}{\rho_2}}{\left(\frac{k_1}{\rho_1}-\frac{k_2}{\rho_2}\right)n^2\pi^2}+O\left(n^{-4}\right)
$$
in \eqref{Equation2.3new-work4}, we get the corresponding eigenfunctions \eqref{Equation4.5-work4}. Similarly for $\lambda=\lambda_{2,n},$ setting 
$$
D_{2,n}=\frac{1}{n\pi},\quad C_{2,n}=\frac{(n\pi)D_{2,n}}{\frac{\rho_1}{k_1}\mu^2_{2,n}-(n\pi)^2}=\frac{\frac{k_1}{\rho_1}}{\left(\frac{k_1}{\rho_1}-\frac{k_2}{\rho_2}\right)n^2\pi^2}+O\left(n^{-4}\right)
$$
 in \eqref{Equation2.3new-work4}, we get the corresponding eigenfunctions \eqref{Equation4.6-work4}.
\end{proof}
\begin{rk}\label{Remark4.4-work4}
 \rm{If  $\frac{k_1}{\rho_1}\neq \frac{k_2}{\rho_2}$, then from Equation \eqref{Equation4.7-work4}, we can easily check that the eigenvalues $\lambda_{1,n}, \lambda_{2,n}$   are simple and different from zero. Then, we set the eigenfunctions of the operator ${\mathcal{A}_2}$ as
\begin{equation*}
E_{1,n}=(\varphi_{1,n},\lambda_{1,n}\varphi_{1,n},\psi_{1,n},\lambda_{1,n}\psi_{1,n}),\quad
E_{2,n}=(\varphi_{2,n},\lambda_{2,n}\varphi_{2,n},\psi_{2,n},\lambda_{2,n}\psi_{2,n}).
\end{equation*}
In fact, using the asymptotic expansions \eqref{Equation4.3-work4}-\eqref{Equation4.4-work4} and \eqref{Equation4.5-work4}-\eqref{Equation4.6-work4}, we can easily prove that $\left\{E_{1,n}, E_{2,n}\right\}_{n\in \mathbb{Z}^*}$ form a Riesz basis in the energy space ${\mathcal{H}_2}.$}\xqed{$\square$}
\end{rk}
\begin{rk}\label{Remark4.5-work4}
\rm{ Similar to subsection \ref{Section3-work4},  the eigenvalues of the same branch satisfy a uniform gap condition, but the eigenvalues of different branches can be asymptotically close at rate depends on the parameters $\rho_1,\ \rho_2,\ k_1,\ k_2$ (see Proposition \ref{Proposition4.6-work4}).  Again in this subsection we will use a general Ingham-sort Theorem.}\xqed{$\square$}
\end{rk}
\begin{prop}\label{Proposition4.6-work4}
\rm{Assume that $\frac{k_1}{\rho_1}\neq \frac{k_2}{\rho_2}$ and condition $(\rm A_1)$ holds. Then, there exists a constant $\gamma>0$ depending only on $\rho_1,\ \rho_2,\ k_1,\ k_2$ such that
\begin{equation}
\label{Equation4.10-work4}
\left|\lambda_{j,m}-\lambda_{l,n}\right|\le 2\gamma \Longrightarrow j\neq l.
\end{equation}
Moreover, there exist constants $c_2>c_1>0$ depending only on $\rho_1,\rho_2, \ k_1$ and $k_2$ such that\\[0.1in]
\textbf{1.} If $ \frac{k_1\rho_2}{k_2\rho_1}$ is a rational number different from $p^2/q^2$ for all integers $p,\ q,$ then for all $\left|m\right|,\ \left|n\right|\geq N_0$, for $N_0$ large enough, we have
\begin{equation}
\label{Equation4.11-work4}
\left|\lambda_{1,m}-\lambda_{2,n}\right|\geq \frac{c_1}{\left|m\right|} \quad \text{and} \quad \left|\lambda_{1,m}-\lambda_{2,n}\right|\geq \frac{c_1}{\left|n\right|},
\end{equation} 
and there exist infinitely many integers $m$, $n$ such that 
\begin{equation}
\label{Equation4.12-work4}
\left|\lambda_{1,m}-\lambda_{2,n}\right|\leq \frac{c_2}{\left|m\right|} \quad \text{and} \quad \left|\lambda_{1,m}-\lambda_{2,n}\right|\leq \frac{c_2}{\left|n\right|}.
\end{equation}
\textbf{2.} If $ \frac{k_1\rho_2}{k_2\rho_1}=p_0^2/q_0^2 \neq 1$ for some integers $p_0$, $q_0,$ then for all $\left|n\right|,\  \left|m\right|\geq N_0$, for $N_0$ large enough, we have
\begin{equation}
\label{Equation4.13-work4}
\left|\lambda_{1,m}-\lambda_{2,n}\right|\geq \frac{c_1}{\left|m\right|} \quad \text{and} \quad \left|\lambda_{1,m}-\lambda_{2,n}\right|\geq \frac{c_1}{\left|n\right|},
\end{equation} 
and there exist infinitely many integers $m$, $n$ such that 
\begin{equation}
\label{Equation4.14-work4}
\left|\lambda_{1,m}-\lambda_{2,n}\right|\leq \frac{c_2}{\left|m\right|} \quad \text{and} \quad \left|\lambda_{1,m}-\lambda_{2,n}\right|\leq \frac{c_2}{\left|n\right|}.
\end{equation}
\textbf{3.} For almost all positive irrational number $ \frac{k_1\rho_2}{k_2\rho_1}$ and all $\left|n\right|,\  \left|m\right|\geq N_0$, for $N_0$ large enough, we have
\begin{equation}
\label{Equation4.15-work4}
\left|\lambda_{1,m}-\lambda_{2,n}\right|\geq \frac{c_1}{\left|m\right|\ln^2|m|} \quad \text{and} \quad \left|\lambda_{1,m}-\lambda_{2,n}\right|\geq \frac{c_1}{\left|n\right|\ln^2|n|},
\end{equation} 
and there exist infinitely many integers $m$, $n$ such that 
\begin{equation}
\label{Equation4.16-work4}
\left|\lambda_{1,m}-\lambda_{2,n}\right|\leq \frac{c_2}{\left|m\right|\ln|m|} \quad \text{and} \quad \left|\lambda_{1,m}-\lambda_{2,n}\right|\leq \frac{c_2}{\left|n\right|\ln|n|}.
\end{equation}}
\end{prop}
\begin{proof} The assertion \eqref{Equation4.10-work4} follows directly from the asymptotic expansions \eqref{Equation4.7-work4} and the fact that all the eigenvalues are geometrically simple. Using the asymptotic expansions \eqref{Equation4.3-work4}-\eqref{Equation4.4-work4}, we have
\begin{equation}
\label{Equation4.17-work4}
\left|\frac{\lambda_{1,m}-\lambda_{2,n}}{m}\right|=\pi\sqrt{\frac{k_2}{\rho_2}}\left|\sqrt{\frac{k_1\rho_2}{k_2\rho_1}}-\frac{n}{m}\right|+\frac{O\left(k_1,k_2,\rho_1,\rho_2\right)}{m^2}+\frac{O\left(k_1,k_2,\rho_1,\rho_2\right)}{\left|mn\right|}.
\end{equation}
If $\left|\sqrt{\frac{k_1\rho_2}{k_2\rho_1}}-\frac{n}{m}\right|\geq \frac{1}{2} \sqrt{\frac{k_1\rho_2}{k_2\rho_1}},$ then the estimates \eqref{Equation4.11-work4}, \eqref{Equation4.13-work4} and \eqref{Equation4.15-work4} are trivial. Otherwise, if 
$$\left|\sqrt{\frac{k_1\rho_2}{k_2\rho_1}}-\frac{n}{m}\right|< \frac{1}{2} \sqrt{\frac{k_1\rho_2}{k_2\rho_1}},$$ 
then $m \sim n$ and \eqref{Equation4.17-work4} becomes
\begin{equation}
\label{Equation4.18-work4}
\left|\frac{\lambda_{1,m}-\lambda_{2,n}}{m}\right|=\pi\sqrt{\frac{k_2}{\rho_2}}\left|\sqrt{\frac{k_1\rho_2}{k_2\rho_1}}-\frac{n}{m}\right|+O\left(m^{-2}\right).
\end{equation}
Therefore, it is  sufficient to consider the leading term in \eqref{Equation4.18-work4}. \\[0.1in]
\textbf{1.} Let $ \frac{k_1\rho_2}{k_2\rho_1} = p_0 /q _0$ be a reduced rational number. Then, $\sqrt{\frac{k_1\rho_2}{k_2\rho_1}}$ is a root of the integer polynomial $q_0 x^2-p_0$ of second degree. Since $ \frac{k_1\rho_2}{k_2\rho_1} \neq p^2/ q^2$ for all integers $p, q,$ then the integer polynomial $q_0 x^2-p_0$ is irreducible. This means that $\sqrt{\frac{k_1\rho_2}{k_2\rho_1}}$ is
a quadratic algebraic number. Thanks to the Liouville’s Theorem on the approximation of algebraic numbers (see Theorem  1.2 in  \cite{Bugeaud01}), there exists a constant $c_1>0$, 
depending only on $ \frac{k_1\rho_2}{k_2\rho_1}$, such that for all $\left|n\right|,\ \left|m\right| \geq N_0$, we have
$$ \left|\sqrt{\frac{k_1\rho_2}{k_2\rho_1}}-\frac{n}{m}\right|\geq\frac{c_1}{m^2}.$$
On the other hand, since $\sqrt{\frac{k_1\rho_2}{k_2\rho_1}}$ is an irrational number, using the Dirichlet’s classic Theorem on number theory (see Theorem  1.1 in  \cite{Bugeaud01}), there exist infinitely many
integers $m, n$ such that
$$\left|\sqrt{\frac{k_1\rho_2}{k_2\rho_1}}-\frac{n}{m}\right|<\frac{1}{m^2}.$$
Therefore, we get the estimates \eqref{Equation4.11-work4}-\eqref{Equation4.12-work4}.\\[0.1in]
\textbf{2.} Let $\sqrt{\frac{k_1\rho_2}{k_2\rho_1}}=\frac{p_0}{q_0},$ be a reduced rational number. We return to \eqref{Equation4.17-work4}, we get 
\begin{equation}\label{Equation4.19-work4}
\left|\lambda_{1,m}-\lambda_{2,n}\right|=\frac{\pi}{q_0}\sqrt{\frac{k_2}{\rho_2}}\left|p_0m-nq_0\right|+O\left(|m|^{-1}\right)+O\left(|n|^{-1}\right).
\end{equation}
If $n\neq k q_0$ or $m\neq k p_0$ for all $k\in\mathbb{Z}^*$, then from \eqref{Equation4.19-work4}, we get
$$\left|\lambda_{1,m}-\lambda_{2,n}\right|\geq c_1.$$
Otherwise, if $n= k q_0$ and $m= k p_0$, then  from \eqref{Equation4.3-work4}-\eqref{Equation4.4-work4}, we deduce that 
$$\left|\lambda_{1,m}-\lambda_{2,n}\right|\geq \frac{c_1}{\left|m\right|}.$$
On the other hand, by taking $m=q_0k$ and $n=p_0k,$ $k\in\mathbb{Z}^*,$ and using the asymptotic expansions \eqref{Equation4.3-work4}-\eqref{Equation4.4-work4}, we easily get that 
$$\left|\lambda_{1,q_0k}-\lambda_{2,p_0k}\right|\leq \frac{c_2}{\left|m\right|}.$$
Therefore, we get the estimates \eqref{Equation4.13-work4}-\eqref{Equation4.14-work4}.\\[0.1in]
\textbf{3.} Let $ \frac{k_1\rho_2}{k_2\rho_1}\not\in \mathbb{Q}$. Firstly, from Khintchine's Theorem on Diophantine approximation
(see Theorem 1.10 in  \cite{Bugeaud01}),  for almost all irrational number $\sqrt{\frac{k_1\rho_2}{k_2\rho_1}}$,  there exist only finitely many integers $m, n$ such that
$$\left|\sqrt{\frac{k_1\rho_2}{k_2\rho_1}}-\frac{n}{m}\right|<\frac{1}{m^2 (\ln \left|m\right|)^2}.$$
It follows from \eqref{Equation4.18-work4}, that for almost all irrational number $\sqrt{\frac{k_1\rho_2}{k_2\rho_1}}$, there exists a constant $c_1>0$ and $N_0\in \mathbb{N}$, $N_0$ large enough, such that, for all  $\left|m\right|$, $\left|n\right|\geq N_0$, we have  
 $$\left|{\frac{\lambda_{1,m}-\lambda_{2,n}}{m}}\right|\geq \frac{c_1}{m^2\ln^2|m|}.$$
This gives the estimate \eqref{Equation4.15-work4}. Secondly, from Khintchine's Theorem on Diophantine approximation (see Theorem 1.10 in  \cite{Bugeaud01}) for almost all irrational real number $\sqrt{\frac{k_1\rho_2}{k_2\rho_1}}$, there exist infinitely many integers $m$, $n>0$ such that
$$\left|\sqrt{\frac{k_1\rho_2}{k_2\rho_1}}-\frac{n}{m}\right|\leq \frac{1}{m^2\ln|m|}.$$
Therefore, we get the estimate \eqref{Equation4.16-work4}, which concludes the  proof of the proposition.
\end{proof}\\[0.1in]
\noindent Similar to Proposition \ref{Proposition3.6-work4}, we can prove the following proposition.
\begin{prop}\label{Proposition4.7-work4}
\rm{Assume that $\frac{k_1}{\rho_1}\neq\frac{k_2}{\rho_2}$ and condition $(\rm A_1)$ holds. We rearrange the two branches of eigenvalues into one sequence $(\lambda_{n})_{n\neq 0}$ such that $(\text{Im}\lambda_{n})_{n\neq 0}$ is strictly increasing. If
\begin{equation}\label{Equation4.20-work4}
0<\text{Im}\lambda_{n+1}-\text{Im}\lambda_{n}\leq \gamma,
\end{equation}
then 
\begin{equation}\label{Equation4.21-work4}
\text{Im}\lambda_{n}-\text{Im}\lambda_{n-1}>\gamma, \quad \text{Im}\lambda_{n+2}-\text{Im}\lambda_{n+1}>\gamma.
\end{equation}
Note that $\text{Im}\lambda_{n},\text{Im}\lambda_{n+1}$ is called a chain of close exponents relative to $\gamma$ of length 2.}\xqed{$\square$}
\end{prop}
\noindent From Proposition \ref{Proposition4.3-work4}, it follows that 
\begin{equation}\label{Equation4.22-work4}
\left(\varphi_{1,n}\right)_x(1)+\psi_{1,n}(1)=O\left(1\right)\ \ \ \text{and}\ \ \ \left(\varphi_{2,n}\right)_x(1)+\psi_{2,n}(1)=O\left(n^{-1}\right).
\end{equation}
Similar to \ref{Section3-work4}, we define the following weighted spectral spaces
\begin{equation}\label{Equation4.23-work4}
\normalsize
\mathcal{D}_1=\left\{ \Phi_0=\sum_{n\neq 0}(\alpha_{1,n}E_{1,n}+\alpha_{2,n}n E_{2,n})n\ ;\  \alpha_{1,n},\ \alpha_{2,n}\in\mathbb{C},\ \sum_{n\neq 0}\left(\left|\alpha_{1,n}\right|^2+\left|\alpha_{2,n}\right|^2\right)<\infty \right\} 
\end{equation}
and
\begin{equation}\label{Equation4.23-work4tilde}
\normalsize
\tilde{\mathcal{D}}_1=\left\{ \Phi_0=\sum_{n\neq 0}(\alpha_{1,n}E_{1,n}+\alpha_{2,n}n E_{2,n})n\ln^2|n|\ ;\  \alpha_{1,n},\ \alpha_{2,n}\in\mathbb{C},\ \sum_{n\neq 0}\left(\left|\alpha_{1,n}\right|^2+\left|\alpha_{2,n}\right|^2\right)<\infty \right\}. 
\end{equation}
The factor $\ln^2|n|$ in \eqref{Equation4.23-work4tilde} will be omitted for $|n|=1.$ Since the System $\left\{E_{1,n}, E_{2,n}\right\}_{n\in\mathbb{Z}^*}$ is a 
Riesz basis in the energy space ${\mathcal{H}_2},$ we get that the spaces $\mathcal{D}_1$ and $\tilde{\mathcal{D}}_1$ are obviously a Hilbert space equipped respectively  with the norm 
$$\sum_{n\neq 0}\left(\left|\alpha_{1,n}\right|^2+\left|\alpha_{2,n}\right|^2\right).$$
\noindent In fact,  to get the observability we need to use a weaker norm for the second equation in order that $\left(\varphi_{2,n}\right)_x\left(1\right)+\psi_{2,n}(1)$  has the same order as $\left(\varphi_{1,n}\right)_x\left(1\right)+\psi_{1,n}(1)$. For this reason   we  multiplied  the eigenvector $E_{2,n}$ by $n$   in the  spaces  ${\mathcal{D}}_1$ and $\tilde{\mathcal{D}}_1$. \\[0.1in]
\noindent We are now ready to prove our observability inequalities results.
\begin{thm}\label{Theorem4.2-work4}
\rm{Assume that $\frac{k_1}{\rho_1}\neq \frac{k_2}{\rho_2}$ and condition $(\rm A_1)$ holds, let 
$$T>2\left(\sqrt{\frac{\rho_1}{k_1}}+\sqrt{\frac{\rho_2}{k_2}}\right).$$
Then, for all solution $\Phi=(\varphi,\varphi_t,\psi,\psi_t)$ that solve the  problem \eqref{Equation2.1-work4} there exists a constant $\ell_0>0$ such that the following direct inequality holds:
\begin{equation} 
\label{Equation4.1-work4}
\int_0^T\left|\varphi_{x}\left(1,t\right)+\psi(1,t)\right|^2 dt \leq \ell_0\left\|(\varphi_0,\varphi_1,\psi_0,\psi_1)\right\|^2_{{\mathcal{H}_2}}.
\end{equation}
Moreover, there exists a constant $0<\ell_1<\ell_0$ depending only on $\rho_1,\ \rho_2,\ k_1$ and $k_2$ such that the following inverse observability inequalities   hold: \\[0.1in]
\noindent\textbf{Case 1.}   If $\frac{k_1\rho_2}{k_2\rho_1}$  is a rational, then 
\begin{equation}
\label{Equation4.2-work4} 
\ell_1\left\|(\varphi_0,\varphi_1,\psi_0,\psi_1)\right\|^2_{\mathcal{D}_1}\leq \int_0^T\left|\varphi_{x}\left(1,t\right)+\psi(1,t)\right|^2 dt.
\end{equation}
 \noindent\textbf{Case 2.} For almost all irrational number $\frac{k_1\rho_2}{k_2\rho_1}$,  we have
\begin{equation}
\label{Equation4.2-work4new} 
\ell_1\left\|(\varphi_0,\varphi_1,\psi_0,\psi_1)\right\|^2_{\tilde{\mathcal{D}}_1}\leq \int_0^T\left|\varphi_{x}\left(1,t\right)+\psi(1,t)\right|^2 dt.
\end{equation}
}
\end{thm}
\begin{proof}
  Similar to subsection \ref{Section3-work4}, we can prove the direct inequality \eqref{Equation4.1-work4}.\\[0.1in] 
\noindent Our next aim is to prove the inverse observability inequalities: \\[0.1in]
\noindent\textbf{Case 1.} Let    $\frac{k_1\rho_2}{k_2\rho_1}$  be a rational. Given any initial data such as
$$
(\varphi_{0},\varphi_{1},\psi_{0},\psi_{1})=\sum_{n\neq 0}\left(\alpha_{1,n}E_{1,n}+\alpha_{2,n}n E_{2,n}\right)n\in \mathcal{D}_1,
$$
using the Riesz property the solution of \eqref{Equation2.1-work4} can be written as
$$
(\varphi (x,t),\varphi_{t}(x,t), \psi (x,t),\psi_{t}(x,t))=\sum_{n\neq 0}\left(\alpha_{1,n}E_{1,n}e^{\lambda_{1,n}t}+
\alpha_{2,n}n E_{2,n}e^{\lambda_{2,n}t}\right)n.$$
Hence, we have
\begin{equation}
\label{Equation4.24-work4}
\varphi_{x}\left(1,t\right)+\psi(1,t)=\sum_{n\neq 0}\left(\alpha_{1,n}\left((\varphi_{1,n})_{x}(1)+\psi_{1,n}(1)\right)e^{\lambda_{1,n}t} +\alpha_{2,n}n\left((\varphi_{2,n})_{x}(1)+\psi_{2,n}(1)\right)e^{\lambda_{2,n}t}\right)n.
\end{equation}
We now rearrange the two branches of eigenvalues $(\lambda_{1,n})_{n\neq 0},(\lambda_{2,n})_{n\neq 0}$ into one sequence $(\lambda_{n})_{n\neq 0}$ such that the sequence $(\text{Im} \lambda_{n})_{n \neq 0}$ is strictly increasing. From Proposition \ref{Proposition4.7-work4}, it follows that all chain $\text{Im} \lambda_{n}, \text{Im} \lambda_{n+1}$ of close exponents relative to $\gamma$ is of length 2. Then, let $A$ denotes the set of integers $n\in \mathbb{Z}^*$ such that the condition \eqref{Equation4.20-work4} holds true and let
$$
B=\mathbb{Z}^*\setminus \left\{n,n+1:\quad n\in A\right\}.
$$
We denote by $a_{n}$ the coefficient before $e^{\lambda_{1,n}t}$ or 
$e^{\lambda_{2,n}t}$ in \eqref{Equation4.24-work4}. We can rewrite it into
\begin{equation*}
\begin{array}{ll}
\displaystyle{\varphi_{x}\left(1,t\right)+\psi(1,t)}& =\displaystyle{\sum_{n\in B} a_{n}e ^{\lambda_{n}t}+\sum_{n\in A}\left( a_{n}e ^{\lambda_{n}t}+a_{n+1}e ^{\lambda_{n+1}t}\right)}\\ \\
&=\displaystyle{\sum_{n\in B} a_{n}e ^{\lambda_{n}t}+\sum_{n\in A} \left((a_{n}+a_{n+1})e ^{\lambda_{n}t}+(\lambda_{n+1}-\lambda_{n})a_{n+1}e_{n+1}\left(t\right)\right),}
\end{array}
\end{equation*}
where $e_{n+1}\left(t\right)$ denotes the divided difference of the chain of exponents $\lambda_{n}, \lambda_{n+1}$ relative to $\gamma$
$$
e_{n+1}\left(t\right)=\frac{e^{\lambda_{n+1}t}-e^{\lambda_{n}t}}{\lambda_{n+1}-\lambda_{n}}.
$$
From Theorem 9.4 in \cite{Komornik01}, the sequence 
$(e^{\lambda_{n}t})_{n\in B},(e^{\lambda_{n}t},e_{n+1}\left(t\right))_{n\in A}$ 
forms a Riesz sequence in $L^2(0,T)$ provided that $T>2\pi D^{+}=2\left(\sqrt{\frac{\rho_1}{k_1}}+\sqrt{\frac{\rho_2}{k_2}}\right).$ Thus, it follows that
\begin{equation}
\label{Equation4.25-work4}
\int_{0}^{T}\left|\varphi_{x}\left(1,t\right)+\psi(1,t)\right|^2 dt \sim \sum_{n\in B} \left|a_{n}\right|^2+\sum_{n\in A}\left(\left|a_{n}+a_{n+1}\right|^2+\left|\lambda_{n+1}-\lambda_{n}\right|^2
\left|a_{n+1}\right|^2\right).
\end{equation}
The assertions \eqref{Equation4.11-work4} and \eqref{Equation4.13-work4}   of Proposition \ref{Proposition4.6-work4}, imply that
\begin{equation}
\label{Equation4.26-work4}
\left|a_{n}+a_{n+1}\right|^2+\left|\lambda_{n+1}-\lambda_{n}\right|^2
\left|a_{n+1}\right|^2\geq \ell_1\left(\frac{\left|a_n\right|^{2}}{\left|n\right|^{2}}+\frac{\left|a_{n+1}\right|^2}{\left|n+1\right|^{2}}\right).
\end{equation}
Inserting \eqref{Equation4.26-work4} into \eqref{Equation4.25-work4} and returning to the previous notations, we get
$$
\int_{0}^{T}\left|\varphi_{x}\left(1,t\right)+\psi(1,t)\right|^2 dt\geq \ell_1 \sum_{n\in B}\frac{\left|a_n\right|^2}{\left|n\right|^{2}}+\ell_1\sum_{n\in A}\left(\frac{\left|a_n\right|^2}{\left|n\right|^{2}}+\frac{\left|a_{n+1}\right|^2}{\left|n+1\right|^{2}}
\right)\geq \ell_1 \sum_{n\in \mathbb{Z}^*}\frac{\left|a_n\right|^2}{\left|n\right|^{2}}.
$$
Hence, we get
\begin{equation}
\label{Equation4.27-work4}
\int_{0}^{T}\left|\varphi_{x}\left(1,t\right)+\psi(1,t)\right|^2 dt\geq \ell_1 \sum_{n\neq 0}\left(\left|\alpha_{1,n}\left((\varphi_{1,n})_{x}(1)+\psi_{1,n}(1)\right)\right|^2+\left|\alpha_{2,n} n\left((\varphi_{2,n})_{x}(1)+\psi_{2,n}(1)\right) \right|^2\right).
\end{equation}
Then, by inserting \eqref{Equation4.22-work4} into \eqref{Equation4.27-work4}, we get
\begin{equation*}
\int_{0}^{T}\left|\varphi_{x}\left(1,t\right)+\psi(1,t)\right|^2 dt\geq \ell_1 \sum_{n \neq 0}\left(\left|\alpha_{1,n}\right|^2+\left|{\alpha_{2,n}}\right|^2\right).
\end{equation*}
Therefore, we get the inequality \eqref{Equation4.2-work4}.\\[0.1in]
\noindent\textbf{Case 2.} For almost all irrational number $\frac{k_1\rho_2}{k_2\rho_1}$. 
Given any initial data such as
$$
(\varphi_{0},\varphi_{1},\psi_{0},\psi_{1})=\sum_{n\neq 0}\left(\alpha_{1,n}E_{1,n}+\alpha_{2,n}n E_{2,n}\right)n\ln^2|n|\in \tilde{\mathcal{D}}_1,
$$
then  the solution of \eqref{Equation2.1-work4} can be written as
$$
(\varphi (x,t),\varphi_{t}(x,t), \psi (x,t),\psi_{t}(x,t))=\sum_{n\neq 0}\left(\alpha_{1,n}E_{1,n}e^{\lambda_{1,n}t}+
\alpha_{2,n}n E_{2,n}e^{\lambda_{2,n}t}\right)n \ln^2|n|.$$
Therefore, we have
\begin{equation}
\label{Equation4.24-work4new}
\varphi_{x}\left(1,t\right)+\psi(1,t)=\sum_{n\neq 0}\left(\alpha_{1,n}\left((\varphi_{1,n})_{x}(1)+\psi_{1,n}(1)\right) e^{\lambda_{1,n}t} +\alpha_{2,n}n\left((\varphi_{2,n})_{x}(1)+\psi_{2,n}(1)\right)e^{\lambda_{2,n}t}\right)n \ln^2|n|.
\end{equation}
Similar to case 1, we get 
\begin{equation}
\label{Equation4.25-work4new}
\int_{0}^{T}\left|\varphi_{x}\left(1,t\right)+\psi(1,t)\right|^2 dt \sim \sum_{n\in B} \left|a_{n}\right|^2+\sum_{n\in A}\left(\left|a_{n}+a_{n+1}\right|^2+\left|\lambda_{n+1}-\lambda_{n}\right|^2
\left|a_{n+1}\right|^2\right),
\end{equation}
where $a_n$ denoted the coefficient before $e^{\lambda_{1,n}t}$ or 
$e^{\lambda_{2,n}t}$ in \eqref{Equation4.24-work4}. Using \eqref{Equation4.15-work4}   of Proposition \ref{Proposition4.6-work4}, we get
\begin{equation}
\label{Equation4.26-work4new}
\left|a_{n}+a_{n+1}\right|^2+\left|\lambda_{n+1}-\lambda_{n}\right|^2
\left|a_{n+1}\right|^2\geq \ell_1\left(\frac{\left|a_n\right|^{2}}{\left|n\right|^{2}\ln^4|n|}+\frac{\left|a_{n+1}\right|^2}{\left|n+1\right|^{2} \ln^4|n+1|}\right).
\end{equation}
Inserting \eqref{Equation4.26-work4new} in \eqref{Equation4.25-work4new}, we get
$$
\displaystyle{\int_{0}^{T}\left|\varphi_{x}\left(1,t\right)+\psi(1,t)\right|^2 dt}\geq \ell_1 \displaystyle{\sum_{n\in \mathbb{Z}^*}\frac{\left|a_n\right|^2}{\left|n\right|^{2}\ln^4|n|}}.
$$
Therefore, we have
\begin{equation}
\label{Equation4.27-work4new}
\int_{0}^{T}\left|\varphi_{x}\left(1,t\right)+\psi(1,t)\right|^2 dt\geq \ell_1 \sum_{n\neq 0}\left(\left|\alpha_{1,n}\left((\varphi_{1,n})_{x}(1)+\psi_{1,n}(1)\right)\right|^2+\left|\alpha_{2,n} n\left((\varphi_{2,n})_{x}(1)+\psi_{2,n}(1)\right) \right|^2\right).
\end{equation}
Then, by inserting \eqref{Equation4.22-work4} into \eqref{Equation4.27-work4new}, we get
\begin{equation*}
\int_{0}^{T}\left|\varphi_{x}\left(1,t\right)+\psi(1,t)\right|^2 dt\geq \ell_1 \sum_{n \neq 0}\left(\left|\alpha_{1,n}\right|^2+\left|{\alpha_{2,n}}\right|^2\right).
\end{equation*}
Hence, we get the inequality \eqref{Equation4.2-work4}. {  Thus, the proof of the theorem  is complete.}
\end{proof}
\begin{rk}\label{Remark4.8-work4}
\rm{It is very important to ask the question about the optimality of the  observability inequality  \eqref{Equation4.2-work4}. In our opinion, from inequalities    \eqref{Equation4.12-work4} and \eqref{Equation4.14-work4}, we may find a dense subspace ${\mathtt{D}}_1$ of $\mathcal{D}_1$  such that 
$$\int_{0}^T\left|\varphi_{x}\left(1,t\right)+\psi(1,t)\right|^2 dt \sim \left\|(\varphi_{0},\varphi_{1},\psi_{0},\psi_{1})\right\|^2_{\mathcal{D}_1},$$
for all $(\varphi_{0},\varphi_{1},\psi_{0},\psi_{1})\in {\mathtt{D}_1}$. }\xqed{$\square$}
\end{rk}
\noindent The weighted spectral spaces $\mathcal{D}_1$ and $\tilde{\mathcal{D}}_1$  are defined by means of the eigenvectors $(E_{1,n})_{n\neq 0}$ and $(E_{2,n})_{n\neq 0}$ with weights. 
Our aim is to get the observability or exact controllability in usual functional spaces. For this aim, let
$$E_{1,n}=(x_n,g_n),\quad E_{2,n}=(f_n,y_n),$$
with
 \begin{equation*}
\left\{
\begin{array}{ll}
\displaystyle{x_n=\left(\frac{\sin(n\pi x)}{n\pi},i\sqrt{\frac{k_1}{\rho_1}} \sin(n\pi x)\right),\quad
g_n=\left(\frac{O\left(k_1,k_2,\rho_1,\rho_2\right)}{n^2}\cos(n\pi x),\frac{O\left(k_1,k_2,\rho_1,\rho_2\right)}{n}\cos(n\pi x)\right)},\nline
\displaystyle{y_n=\left(\frac{O\left(k_1,k_2,\rho_1,\rho_2\right)}{n^2} \sin(n\pi x), \frac{O\left(k_1,k_2,\rho_1,\rho_2\right)}{n}\sin(n\pi x)\right),\quad
f_n=\left(\frac{\cos(n\pi x)}{n\pi}, i\sqrt{\frac{k_2}{\rho_2}}\cos(n\pi x)\right).}
\end{array}
\right.
\end{equation*}
For any $s\geq 0,$ we define the spaces
\begin{equation*}
X_{s}=\left\{(\hat{\phi},\hat{\psi})=\sum_{n\neq 0} \beta_n n^{s} x_n\right\}, \quad \left\|(\hat{\phi},\hat{\psi})\right\|^2_{X_{s}}=\sum_{n\neq 0} \left|\beta_n\right|^2
\end{equation*}
and
\begin{equation*}
\tilde{X}_{s}=\left\{(\hat{\phi},\hat{\psi})=\sum_{n\neq 0} \beta_n n^{s}\ln^2|n| x_n\right\}, \quad \left\|(\hat{\phi},\hat{\psi})\right\|^2_{X_{s}}=\sum_{n\neq 0} \left|\beta_n\right|^2.
\end{equation*}
\begin{cor}\label{Corollary4.9-work4}
\rm{Assume that $\frac{k_1}{\rho_1}\neq \frac{k_2}{\rho_2}$ and condition $(\rm A_1)$ holds. Then, we have 
\begin{equation}\label{Equation4.28-work4}
\mathcal{D}_1= X_1\times X_2
\end{equation}
and
\begin{equation}\label{Equation4.28-work4tilde}
\tilde{\mathcal{D}}_1= \tilde{X}_1\times \tilde{X}_2.
\end{equation}
}
\end{cor}
\begin{proof} We see that $(nx_n)_{n\neq 0}$ and $(n^2y_n)_{n\neq 0}$ are Riesz basis in $X_1\times X_2$, respectively $(n^2 f_n)_{n\neq 0}$ and $( n g_n)_{n\neq 0}$ are Bessel sequences in $X_1\times X_2$. Then, \eqref{Equation4.28-work4} follows directly from Theorem \ref{Theoremk3.8-work4}.  The assertion \eqref{Equation4.28-work4tilde} can be obtained in the same way.
\end{proof}\\[0.1in]
\noindent Furthermore, for any $s\geq 0,$ we define the spaces
\begin{equation*}
V_{s}=\left\{ f=\sum_{n>0}\alpha_n\frac{\sin(n\pi x)}{n^{s}}\right\}, \quad \quad \left\|f\right\|^2_{V_{s}}=\sum_{n>0}\left|\alpha_n\right|^2
\end{equation*}
and 
\begin{equation*}
\tilde{V}_{s}=\left\{ f=\sum_{n>0}\alpha_n\frac{\sin(n\pi x)}{n^{s}\ln^2|n|}\right\}, \quad \quad \left\|f\right\|^2_{V_{s}}=\sum_{n>0}\left|\alpha_n\right|^2.
\end{equation*}
Thus, with the pivot space $L^2\left(0,1\right),$ we have
$$X_{s}=V'_{s-1}\times V'_s,\quad \tilde{X}_{2}=\tilde{V}'_{s-1}\times \tilde{V}'_s.$$
Then, it follows  that
$$\mathcal{D}_1=V'_{0}\times V'_1\times V'_{1}\times V'_2$$
and
$$\tilde{\mathcal{D}}_1=\tilde{V}'_{0}\times \tilde{V}'_1\times \tilde{V}'_{1}\times \tilde{V}'_2.$$
\begin{rk}\label{Remark4.10-work4}
\rm{Assume that $\frac{k_1}{\rho_1}\neq \frac{k_2}{\rho_2}$ and condition $(\rm A_1)$ holds.  In the case 1, since the two branches of eigenvalues are close in the order of $1/n,$  then the observability space of the first equation losses one derivative because of the closeness of eigenvalues, while that of the second equation losses two derivatives due to the closeness of eigenvalues and the transmission of the modes between the two equations. Therefore, the observability holds in the space of type
$$\mathcal{D}_1=L^2\left(0,1\right)\times H^{-1}\left(0,1\right)\times H^{-1}\left(0,1\right)\times H^{-2}\left(0,1\right).$$
Moreover, the control space are of type
 $$H^{1}\left(0,1\right)\times L^2\left(0,1\right)\times H^{2}\left(0,1\right)\times H^1\left(0,1\right)$$
 with suitable boundary conditions.}
\xqed{$\square$}
\end{rk}
\noindent It is interesting to notice that the observability of the System  \eqref{Equation2.1-work4}  suggests the exact controllability of the System \eqref{Equation1.1-work4}   (see Theorems \ref{Th56}, \ref{Theorem3.1-work4}  and in
\cite{Lions01,Lions02,Lions03,Lions04,Komornik01,Komornik02}). Then, from Theorem \ref{Theorem4.2-work4},  we get the following result.
\begin{thm}\label{Theorem4.1-work4}
\rm{Assume that  $\frac{k_1}{\rho_1}\neq \frac{k_2}{\rho_2}$ and condition $(\rm A_1)$ holds. Let $$T>2\left(\sqrt{\frac{\rho_1}{k_1}}+\sqrt{\frac{\rho_2}{k_2}}\right).$$ 
 \noindent\textbf{Case 1.}    If $\frac{k_1\rho_2}{k_2\rho_1}$  be a rational, let
$$\left(u_0,u_1,y_0,y_1\right)\in  V_1\times V_0\times V_2\times V_1, $$
then there exists $v\in L^2(0,T)$ such that the solution of System \eqref{Equation1.1-work4}   satisfies the null final conditions
\begin{equation*}
u(x,T)=u_t(x,T)=y(x,T)=y_t(x,T)=0,\quad \forall x\in\left(0,1\right).
\end{equation*}
 \noindent\textbf{Case 2.}   For almost all irrational number $\frac{k_1\rho_2}{k_2\rho_1}$.  Let
 $$\left(u_0,u_1,y_0,y_1\right)\in  \tilde{V}_1\times \tilde{V}_0\times \tilde{V}_2\times \tilde{V}_1,$$
then there exists $v\in L^2(0,T)$ such that the solution of the  System \eqref{Equation1.1-work4}   satisfies the null final conditions
\begin{equation*}
u(x,T)=u_t(x,T)=y(x,T)=y_t(x,T)=0,\quad \forall x\in\left(0,1\right).
\end{equation*}
 }\xqed{$\square$}
\end{thm}

\end{document}